\documentclass[10pt,reqno]{amsart}
\usepackage{amssymb,mathrsfs,graphicx,extpfeil,amsmath}
\usepackage{ifthen,latexsym,float,colortbl}
\usepackage[margin=1in]{geometry}
\usepackage{caption}
\usepackage{sidecap}
\usepackage{rotating,dsfont,stackengine}
\usepackage{hyperref}
\usepackage{colortbl}
\definecolor{black}{rgb}{0.0, 0.0, 0.0}
\definecolor{red}{rgb}{1.0, 0.5, 0.5}
\provideboolean{shownotes} 
\setboolean{shownotes}{true} 
\newcommand{\margnote}[1]{
\ifthenelse{\boolean{shownotes}}%
{\marginpar{\raggedright\tiny\texttt{#1}}}%
{}%
}
\newcommand{\hole}[1]{
\ifthenelse{\boolean{shownotes}}%
{\begin{center} \fbox{ \rule {.25cm}{0cm} \rule[-.1cm]{0cm}{.4cm}
\parbox{.85\textwidth}{\begin{center} \texttt{#1}\end{center}} \rule
{.25cm}{0cm}}\end{center}} {} }

\graphicspath{{pics/}}

\title[Hydrodynamic limit of Vlasov-type equations with nonlocal forces]{Quantifying the hydrodynamic limit of Vlasov-type equations with alignment and nonlocal forces}

\author[Carrillo]{Jos\'{e} A. Carrillo}
\address[Jos\'{e} A. Carrillo]{\newline Department of Mathematics
    \newline Mathematical Institute, University of Oxford, Oxford OX2 6GG, UK}
\email{carrillo@maths.ox.ac.uk}

\author[Choi]{Young-Pil Choi}
\address[Young-Pil Choi]{\newline Department of Mathematics\newline
Yonsei University, 50 Yonsei-Ro, Seodaemun-Gu, Seoul 03722, Republic of Korea}
\email{ypchoi@yonsei.ac.kr}

\author[Jung]{Jinwook Jung}
\address[Jinwook Jung]{\newline Research Institute of Basic Sciences \newline Seoul National University, Seoul  08826, Korea (Republic of)}
\email{warp100@snu.ac.kr}

\numberwithin{equation}{section}

\newtheorem{theorem}{Theorem}[section]
\newtheorem{lemma}{Lemma}[section]
\newtheorem{corollary}{Corollary}[section]
\newtheorem{proposition}{Proposition}[section]
\newtheorem{remark}{Remark}[section]
\newtheorem{definition}{Definition}[section]

\newcommand{\R}{\mathbb R}
\newcommand{\N}{\mathbb N}

\newcommand{\bbn}{\mathbb N}
\newcommand{\om}{\Omega}

\newcommand{\T}{\mathbb T}

\newcommand{\mc}{\mathcal C}

\newcommand{\bq}{\begin{equation}}
\newcommand{\eq}{\end{equation}}
\newcommand{\e}{\varepsilon}
\newcommand{\lt}{\left}
\newcommand{\rt}{\right}

\newcommand{\pa}{\partial}

\newcommand{\mh}{\mathcal{H}}
\newcommand{\me}{\mathcal{E}}
\newcommand{\mf}{\mathcal{F}}
\newcommand{\md}{\mathcal{D}}

\newcommand{\mw}{\mathcal{W}}
\newcommand{\into}{\int_\om}
\newcommand{\intoo}{\int_{\om \times \om}}
\newcommand{\intr}{\int_{\R^d}}
\newcommand{\intor}{\int_{\om \times \R^d}}
\newcommand{\intorr}{\int_{\om^2 \times \R^{2d}}}

\makeatletter
\def\moverlay{\mathpalette\mov@rlay}
\def\mov@rlay#1#2{\leavevmode\vtop{%
   \baselineskip\z@skip \lineskiplimit-\maxdimen
   \ialign{\hfil$\m@th#1##$\hfil\cr#2\crcr}}}
\newcommand{\charfusion}[3][\mathord]{
    #1{\ifx#1\mathop\vphantom{#2}\fi
        \mathpalette\mov@rlay{#2\cr#3}
      }
    \ifx#1\mathop\expandafter\displaylimits\fi}
\makeatother

\begin{document}
\allowdisplaybreaks

\date{\today}


\begin{abstract} In this paper, we quantify the asymptotic limit of collective behavior kinetic equations arising in mathematical biology modeled by Vlasov-type equations with nonlocal interaction forces and alignment. More precisely, we investigate the hydrodynamic limit of a kinetic Cucker--Smale flocking model with confinement, nonlocal interaction, and local alignment forces, linear damping and diffusion in velocity. We first discuss the hydrodynamic limit of our main equation under strong local alignment and diffusion regime, and we rigorously derive the isothermal Euler equations with nonlocal forces. We also analyze the hydrodynamic limit corresponding to strong local alignment without diffusion. In this case, the limiting system is pressureless Euler-type equations. Our analysis includes the Coulombian interaction potential for both cases and explicit estimates on the distance towards the limiting hydrodynamic equations. The relative entropy method is the crucial technology in our main results, however, for the case without diffusion, we combine a modulated macroscopic kinetic energy with the bounded Lipschitz distance to deal with the nonlocality in the interaction forces. For the sake of completeness, the existence of weak and strong solutions to the kinetic and fluid equations are also established. 
\end{abstract}

\maketitle \centerline{\date}

\tableofcontents

%
%
%
%
\section{Introduction}

Collective self-organized motions of autonomous individuals, such as flocks of birds, crowd dynamics, and aggregation of bacteria, etc, appear in many applications in the field of engineering, biology, and sociology \cite{BG15,BPT12,BS12,CS07,KTIH11,LPL07,LLE10,MT11,TBL06}, we refer to \cite{CCP17,CHL17} and references therein for recent surveys. Mathematical modelling of such behaviors is based on Individual-Based Models (IBMs) which are microscopic descriptions, and it includes three basic effects, a short-range repulsion, a long-range attraction, and an alignment in certain spatial regions. These IBMs lead to continuum description by means of mean-field limit \cite{CCR11,CCH14,CCHS19,CS18,HL09,JW17}, and in particular a second-order $N$-particle system converges toward a kinetic equation as the number of particles $N$ goes to infinity. In this paper, we study a class of such kinetic-type models which are typically Vlasov-type equations with nonlocal forces. More precisely, let $f = f(x,v,t)$ be the one-particle distribution function at $(x,v) \in \Omega \times \R^d$ and at time $t > 0$, where $\Omega$ is either $\T^d$ or $\R^d$ with $d \geq 1$, then our main equation is given by 
\begin{align}\label{main_eq}
\begin{aligned}
&\pa_t f + v\cdot\nabla_x f -  \nabla_v \cdot \lt((\gamma v + \lambda\lt(\nabla_x V + \nabla_x W \star \rho\rt))f \rt)+ \alpha\nabla_v \cdot \lt(F[f]f\rt)  = \mathcal{N}_{FP}[f], 
\end{aligned}
\end{align}
with $(x,v, t) \in \om \times \R^d \times \R_+$ subject to the initial data:
\[
f(x,v,t)|_{t=0} =: f_0(x,v), \quad (x,v) \in \om \times \R^d,
\]
where $\rho = \rho(x,t)$ and $u = u(x,t)$ are the local particle density and velocity given by 
\[
\rho = \int_{\R^d} f\,dv \quad \mbox{and} \quad u = \frac{\int_{\R^d} vf\,dv}{\int_{\R^d} f\,dv},
\]
respectively, $V: \R^d \to \R$ and $W: \R^d \to \R$ are the confinement and the interaction potentials with a positive coefficient $\lambda$, respectively. Here $\mathcal{N}_{FP}$ denotes the nonlinear Fokker--Planck operator \cite{Vil02} given by 
\[
 \mathcal{N}_{FP}[f](x,v) := \nabla_v \cdot (\beta (v - u)f + \sigma \nabla_v f) = \sigma \nabla_v \cdot \lt(f \nabla_v \log \frac{f}{M_u} \rt) 
\]
with the local Maxwellian
\[
M_u := \frac{\beta^{d/2}}{(2\pi \sigma)^{d/2}} \exp\lt(-\frac{\beta |u-v|^2}{2\sigma} \rt) ,
\]
and positive constants $\beta$ and $\sigma$. $F$ represents the velocity alignment force fields, where the local average of relative velocities weighted by the function $\phi$, given by
\[
F[f](x,v) := \int_{\om \times \R^d} \phi(x-y) (w-v)f(y,w)\,dydw, 
\]
where $\phi: \R^d \to \R_+$ is called a communication weight. The confinement and interaction potentials are assumed to be symmetric in the sense $V(x) = V(-x)$ and $W(x) = W(-x)$ on $\R^d$ due to the action-reaction principle by  Newton's third law.  The weight function $\phi$ is usually assumed to be radially symmetric, i.e., $\phi(x) = \hat\phi(|x|)$ for some $\hat\phi : \R_+ \cup \{0\} \to \R_+$, and $\hat\phi$ is decreasing such that the closer particles have more stronger influence than the further ones. The right hand side of \eqref{main_eq} consists of the local alignment forces and the diffusion term in velocity. Throughout this paper, we also assume that $f$ is a probability density, i.e., $\|f(\cdot,\cdot,t)\|_{L^1} = 1$ for $t\geq 0$, since the total mass is preserved in time.

In the current work, we are interested in the asymptotic analysis of \eqref{main_eq} by considering singular parameters. More specifically, we deal with hydrodynamic limits to isothermal/pressureless Euler equations with nonlocal forces.

\subsection{Formal derivation from kinetic to isothermal/pressureless Euler equations} Taking into account the moments on the kinetic equation \eqref{main_eq}, we find that the local density $\rho $ and local velocity  $u$ satisfy
$$\begin{aligned}
&\pa_t \rho + \nabla_x \cdot (\rho u) = 0,\cr
&\pa_t (\rho u) + \nabla_x \cdot (\rho u \otimes u) + \nabla_x \cdot \lt( \int_{\R^d} (v-u)\otimes (v-u) f(x,v,t)\,dv\rt)\cr
&\hspace{1.5cm} = - \gamma \rho u - \lambda\rho(\nabla_x V + \nabla_x W \star \rho)- \alpha \rho \int_\om \phi(x-y)(u(x) - u(y))\rho(y)\,dy.
\end{aligned}$$
We notice that the above system is not closed in the sense that it cannot be written only in terms of $\rho$ and $u$. On the other hand, if we consider the singular parameters $\beta = \sigma = 1/\e$ in \eqref{main_eq}, i.e., the local alignment and diffusive forces are very strong and consider the limit $\e \to 0$, then at the formal level, we expect that $\mathcal{N}_{FP} \simeq 0$, and this leads that the particle density behaves like: 
\bq\label{max_as}
f^\e(x,v,t) \simeq \frac{\rho(x,t)}{(2\pi)^{d/2}}\exp\lt( - \frac{|v-u(x,t)|^2}{2} \rt) \quad \mbox{for} \quad \e \ll 1,
\eq
where $f^\e$ denotes the corresponding solution of \eqref{main_eq} with $\beta = \sigma = 1/\e$. This formal procedure gives the isothermal Euler equations with interaction forces:
\begin{align}\label{main_pE}
\begin{aligned}
&\pa_t \rho + \nabla_x \cdot (\rho u) = 0, \quad (x,t) \in \om \times \R_+,\cr
&\pa_t (\rho u) + \nabla_x \cdot (\rho u \otimes u) + \nabla_x \rho \cr
&\hspace{1.5cm} = - \gamma \rho u - \lambda\rho(\nabla_x V + \nabla_x W \star \rho) - \alpha \rho \int_\om \phi(x-y)(u(x) - u(y))\rho(y)\,dy.
\end{aligned}
\end{align}
Let us now take into account the hydrodynamic limit without diffusion, i.e., the equation \eqref{main_eq} with $\beta = 1/\e$ and $\sigma = 0$. Then, for the similar reason, we find that 
\bq\label{mono_as}
f^\e(x,v,t) \simeq \rho(x,t) \otimes \delta_{u(x,t)}(v) \quad \mbox{for} \quad \e \ll 1,
\eq
and this induces the following pressureless Euler equations with interaction forces:
\begin{align}\label{main_npE}
\begin{aligned}
&\pa_t \rho + \nabla_x \cdot (\rho u) = 0, \quad (x,t) \in \om \times \R_+,\cr
&\pa_t (\rho u) + \nabla_x \cdot (\rho u \otimes u) =  - \gamma \rho u - \lambda\rho(\nabla_x V + \nabla_x W \star \rho) - \alpha \rho \int_\om \phi(x-y)(u(x) - u(y))\rho(y)\,dy.
\end{aligned}
\end{align}

Some previous works closely related to the above asymptotic analysis can be summarized as follows. The asymptotic analysis for the kinetic Cucker--Smale model with a strong local alignment force and a strong diffusion, i.e., \eqref{main_eq} with $\gamma = 0$, $\lambda, \alpha > 0$, $V, W \equiv 0$, $\sigma = \beta = 1/\e$, is  investigated in \cite{KMT15}. In this regime, the isothermal Euler system  with the nonlocal velocity alignment forces, \eqref{main_pE} with $\gamma = \lambda = 0$ and $\alpha >0$ is rigorously derived, see also \cite{Choi19} for the global regularity of classical solutions of that system. In this work, the relative entropy method is employed, and the presence of the pressure term in the limiting system plays an important role in their strategy: it gives the convexity of the entropy with respect to the density $\rho$; see Section \ref{sec_pressure} for details. For the diffusionless case, in \cite{FK19}, the velocity alignment term $F[f]$ is taken into account in the hydrodynamic limit, i.e., the equation \eqref{main_eq} with $V, W \equiv 0$, $\gamma = \sigma = 0$, $\alpha >0$, and $\beta = 1/\e$ in the periodic spatial domain, and the pressureless Euler equations with the velocity alignment forces, \eqref{main_npE} with $\lambda=\gamma=0$ and $\alpha > 0$, which is also referred to {\it Euler alignment system} in \cite{CCTT16}, are rigorously derived. In that work, the modulated macroscopic energy combined with the second-order Wasserstein  distance is used. This strategy is improved in a recent work \cite{CCpre} where the whole space case is considered, see also \cite{Cpre} for the relation between modulated macroscopic kinetic energy and the $p$th order Wasserstein distance. It is worth noticing that the interaction potential $W$ is not taken into account in \cite{FK19, KMT15}, and it is not clear that the strategies used in that work can be applied to the case with the interaction potential $W$ when $W$ has a rather weak regularity, see \cite{CCpre} for the case with regular interaction potentials $W$. On the other hand, for the Coulombian interactions $W$, i.e., $-\Delta_x W \star \rho = \rho$, the hydrodynamic limit of Vlasov-Poisson equation with strong local alignment forces, which corresponds to \eqref{main_eq} with $\gamma = \alpha =\sigma = 0$, $V \equiv 0$, $\beta = 1/\e$, is discussed in \cite{Kang18}. 

The main purpose of this work is to consider the most general form of kinetic swarming models \eqref{main_eq} and identify regimes where the Euler-type equations \eqref{main_pE} or \eqref{main_npE} are well approximated by the kinetic equation \eqref{main_eq} in a quantifiable way. We first deal with the equation \eqref{main_eq} with strong local alignment and diffusive forces, that is, we consider a singular parameter in the nonlinear Fokker--Planck operator $\mathcal{N}_{FP}$. In this case, as mentioned above, the limiting system is expected to be the isothermal Euler-type system \eqref{main_pE}. We estimate the relative entropy functional together with the free energy to have the quantitative error estimate between solutions $f^\e$ of \eqref{main_eq} and $(\rho,u)$ of \eqref{main_pE}. In particular, we make the formal observation \eqref{max_as} completely rigorous with a quantitive bound in terms of $\e$, see Corollary \ref{cor_hydro1.5}. Due to the presence of pressure, $L^\infty$ bound assumptions for both the interaction potential $W$ and the communication weight function $\phi$ are sufficient to have that estimate of hydrodynamic limit. We are also able to deal with the Coulombian potential for $W$. 

In the case without diffusion, the limiting system is a pressureless Euler system \eqref{main_npE}, thus the corresponding macroscopic kinetic energy is not strictly convex with respect to $\rho$. In this respect, it is not clear to have the quantitative bound error estimate between solutions by means of the estimate of modulated kinetic energy only. For that reason, we combine the modulated kinetic energy estimate and the bounded Lipschitz distance between local particle density $\rho^\e$ and the fluid density $\rho$. Note that the bounded Lipschitz distance and the first order Wasserstein distance are equivalent in the set of probability measure with a bounded first moment. Thus our result improves the previous works  \cite{CCpre,FK19}, where the second-order Wasserstein distance is used as mentioned above. We show that the bounded Lipschitz distance between densities can be bounded from above by the modulated macroscopic kinetic energy, see Lemma \ref{prop_rho_wa}. Compared to the case with pressure, we need rather stronger assumptions for $W$ and $\phi$, bounded and Lipschitz continuity. Combining these observations, we close the modulated kinetic energy estimates and obtain the quantitative error estimates between solutions $f^\e$ of \eqref{main_eq} with $\beta=1/\e, \sigma = 0$ and $(\rho,u)$ of \eqref{main_npE}. As we expected from the formal derivation \eqref{mono_as}, the particle distribution function $f^\e$ converges to the monokinetic ansatz in the sense of distributions also quantified in terms of the bounded Lipschitz distance, see Corollary \ref{cor_hydro2} and the proofs in Subsection 2.3. Even in the pressureless case, we are also able to take into account the Coulombian interaction potential $W$ and establish the same convergence estimates with the regular interaction potential case. Our main mathematical tool is based on the weak-strong uniqueness principle \cite{Daf79}, and thus for the rigorous asymptotic analysis mentioned above, the existence of weak solutions of the kinetic equation \eqref{main_eq} and strong solutions to the limiting systems \eqref{main_pE} and \eqref{main_npE} should be obtained at least locally in time. We emphasize that it is important to have the global-in-time weak solutions of \eqref{main_eq} satisfying the free energy estimate. 

Here we introduce several notations used throughout the current work. For functions, $f(x,v)$ and $g(x)$, $\|f\|_{L^p}$ and $\|g\|_{L^p}$ represent the usual $L^p(\om \times \R^d)$- and $L^p(\om)$-norms, respectively. We denote by $C$ a generic positive constant which may differ from line to line. For simplicity, we often drop $x$-dependence of differential operators, that is, $\nabla f:= \nabla_x f$ and $\Delta f := \Delta_x f$. For any nonnegative integer $k$ and $p \in [1,\infty]$, $\mw^{k,p} = \mw^{k,p}(\om)$ stands for the $k$-th order $L^p$ Sobolev space. In particular, if $p=2$, we denote by $H^k=H^k(\om) = \mw^{k,2}(\om)$. $\mc^k([0,T];E)$ is the set of $k$-times continuously differentiable functions from an interval $[0,T] \subset \R$ into a Banach space $E$, and $L^p(0,T;E)$ is the set of functions from an interval $(0,T)$ to a Banach space $E$. $\nabla^k$ denotes any partial derivative $\pa^\alpha$ with multi-index $\alpha$, $|\alpha| = k$.

The rest of this paper is organized as follows. In the next section, we provide several a priori estimates of free energy inequalities. We also give precise statements of our main results on the asymptotic analysis of \eqref{main_eq}. In Section \ref{sec_pressure}, we consider our main equation \eqref{main_eq} in the regime of strong local alignment and diffusion, i.e., $\beta= \sigma = 1/\e$. We show that the weak solution to the kinetic equation \eqref{main_eq} strongly converges to the strong solution to the isothermal Euler equations with nonlocal interaction forces \eqref{main_pE}. Section \ref{sec_npE} is devoted to the asymptotic analysis for the diffusionless case, i.e., \eqref{main_eq} with $\sigma=0$. In this case, we consider the strong local alignment regime, $\beta = 1/\e$ and provide the rigorous convergence estimates of solutions $f^\e$ to the pressureless Euler system with nonlocal interactions forces \eqref{main_npE}.  Finally, in Sections \ref{sec_global_kin} and \ref{sec_local_pE} we provide the details on the global-in-time existence of weak solutions for the kinetic equation \eqref{main_eq} satisfying the free energy estimate and the local-in-time existence and uniqueness of classical solutions to the isothermal/pressureless Euler equations \eqref{main_pE} and \eqref{main_npE}. 

%
%
%
%
\section{Preliminaries and main results}

\subsection{Free energy estimates} In this part, we provide free energy estimates. For this, we introduce the free energy $\mf$ and the associated dissipations $\md_1$, $\md_2$, and $\md_3$ as follows: 
\[
\mf(f) := \int_{\om \times \R^d} \frac{\sigma}{\beta}f \log f\,dxdv + \frac12\intor |v|^2 f\,dxdv + \frac{\lambda}{2}\intoo W(x-y)\rho(x) \rho(y)\,dxdy + \lambda \into V \rho\,dx, 
\]
$$\begin{aligned}
\md_1(f) &:= \intor \frac{1}{f} \lt|\frac{\sigma}{\beta} \nabla_v f - f(u-v) \rt|^2 dxdv,\cr
\md_2(f) &:= \frac12\intorr \phi(x-y)|v-w|^2 f(x,v)f(y,w)\,dxdydvdw,
\end{aligned}$$
and
\[
\md_3(f):= \intor |v|^2 f\,dxdv,
\]
respectively.\newline
\vspace{0.3cm}

Then we have the following free energy estimate. 
\begin{lemma}\label{lem_energy}
Suppose that $f$ is a solution of \eqref{main_eq} with sufficient integrability. Then we have
\[
\frac{d}{dt}\mf(f) + \beta D_1(f) + \alpha D_2(f) +\gamma D_3(f) =  \frac{\sigma\gamma d}{\beta} + \frac{\sigma \alpha d}{\beta}\intoo \phi(x-y) \rho(x) \rho(y)\,dxdy.
\]
In particular, we have
\begin{align}\label{energy_zerosig}
\begin{aligned}
&\frac{d}{dt}\lt(\frac12\intor |v|^2 f\,dxdv + \frac{\lambda}{2}\intoo W(x-y)\rho(x) \rho(y)\,dxdy + \lambda \into V \rho\,dx \rt)\cr
&\qquad = - \beta \intor f|u-v|^2\,dxdv - \alpha D_2(f) - \gamma D_3(f),
\end{aligned}
\end{align}
when $\sigma = 0$. 
\end{lemma}
\begin{proof}A straightforward computation gives
$$\begin{aligned}
\frac{d}{dt}\intor \frac\sigma\beta f \log f\,dxdv &= \frac{\sigma}{\beta}\intor \nabla_v \cdot \lt( (\gamma v + \lambda(\nabla V + \nabla W \star \rho))f \rt)\log f\,dxdv \cr
&\quad - \frac{\sigma \alpha}{\beta}\intor \nabla_v \cdot (F[f]f) \log f\,dxdv\cr
&\quad + \frac\sigma\beta\intor \nabla_v \cdot (\beta (v - u)f + \sigma \nabla_v f) \log f\,dxdv\cr
&=: \sum_{i=1}^3 I_i,
\end{aligned}$$
where $I_i,i=1,2,3$, can be estimated as follows:
$$\begin{aligned}
I_1 &= \frac{\sigma}{\beta}\intor \nabla_v \cdot (\gamma v + \lambda(\nabla V + \nabla W \star \rho))f\,dxdv = \frac{\sigma\gamma d}{\beta},\cr
I_2 &= - \frac{\sigma \alpha}{\beta} \intor \nabla_v \cdot (F[f]) f \,dxdv = \frac{\sigma \alpha d}{\beta}\intoo \phi(x-y) \rho(x) \rho(y)\,dxdy,\cr
I_3 &= \frac\sigma\beta\intor \lt( \beta(v-u)f + \sigma \nabla_v f\rt) \cdot \frac{\nabla_v f}{f} \,dxdv.
\end{aligned}$$ 
We also estimate the kinetic energy as 
$$\begin{aligned}
\frac12\frac{d}{dt}\intor |v|^2 f\,dxdv &= -\lambda\intor |v|^2 f\,dxdv - \frac{\lambda}{2}\frac{d}{dt}\intoo W(x-y)\rho(x) \rho(y)\,dxdy - \lambda\frac{d}{dt}\into V \rho\,dx\cr
&\quad -\frac{\alpha}{2}\intorr \phi(x-y)|v-w|^2 f(x,v)f(y,w)\,dxdydvdw \cr
&\quad -\beta\intor |v-u|^2 f\,dxdv - \sigma \intor v \cdot \nabla_v f\,dxdv.
\end{aligned}$$
Combining the above estimates yields
\[
\frac{d}{dt}\mf(f) + \beta D_1(f) + \alpha D_2(f) +\gamma D_3(f) =  \frac{\sigma\lambda d}{\beta} + \frac{\sigma \alpha d}{\beta}\intoo \phi(x-y) \rho(x) \rho(y)\,dxdy.
\]
\end{proof}
Lemma \ref{lem_energy} shows that the linear damping in velocity and nonlocal velocity generate the free energy increase. In the proposition below, we show that they are controlled by the dissipations and the free energy.
\begin{proposition}\label{prop_energy} Suppose that $f$ is a solution of \eqref{main_eq} with sufficient integrability. Then we have
$$\begin{aligned}
&\mf(f) + \int_0^t \lt(\frac\beta2\md_1(f) + \gamma \into \rho|u|^2\,dx + \frac\alpha2\intoo\phi(x-y)|u(x) - u(y)|^2\rho(x)\rho(y)\,dxdy\rt)ds \cr
&\qquad \leq \mf(f_0)\exp\lt(\frac{C}{\beta}(1+ \gamma^2)T\rt).
\end{aligned}$$
Furthermore, we obtain
\begin{align}\label{entro_1}
\begin{aligned}
&\mf(f) + \int_0^t \lt(\frac\beta2\md_1(f) + \gamma \into \rho|u|^2\,dx + \frac\alpha2\intoo \phi(x-y)|u(x) - u(y)|^2\rho(x)\rho(y)\,dxdy\rt)ds \cr
&\qquad \leq \mf(f_0) + \lt(\frac{C}{\beta}(1+ \gamma^2)\rt),
\end{aligned}
\end{align}
where $C > 0$ depends only $T$, $f_0$ and $\|\phi\|_{L^\infty}$.
\end{proposition}
\begin{proof}It follows from \cite[Proposition 2.1]{KMT15} or \cite[Lemma 7.3]{KMT13} that 
$$\begin{aligned}
&\frac\alpha2\intoo \phi(x-y)|u(x) - u(y)|^2 \rho(x)\rho(y)\,dxdy - \frac\beta4 D_1(f)  \cr
&\quad \leq \frac{C}{\beta}\mathcal{F}(f)+ \alpha D_2(f) -\frac{\sigma \alpha d}{\beta}\intoo \phi(x-y) \rho(x) \rho(y)\,dxdy,
\end{aligned}$$
where $C$ depends only on $T$, $\|\phi\|_{L^\infty}$. On the other hand, a straightforward computation gives
$$\begin{aligned} 
\frac{\gamma\sigma d}{\beta} &= \gamma \intor v \cdot \lt( f(u-v) - \frac\sigma\beta \nabla_v f \rt)\,dxdv - \gamma \intor v \cdot f(u-v)\,dxdv\cr
&=: J_1 + J_2,
\end{aligned}$$
where $J_2$ can estimated as
\[
J_2 = \gamma \intor f|v|^2\,dxdv - \gamma\into \rho|u|^2\,dx.
\]
For the estimate of $J_1$, we use H\"older inequality to get
$$\begin{aligned}
J_1 &= \gamma \intor \sqrt{f} v \cdot \frac{1}{\sqrt{f}}\lt( f(u-v) - \frac\sigma\beta \nabla_v f  \rt)\,dxdv \leq \gamma\lt(\intor |v|^2 f\,dxdv \rt)^{1/2} D_1(f)^{1/2} \cr
&\leq \frac{\gamma^2}{\beta} \intor |v|^2 f\,dxdv + \frac\beta4D_1(f),
\end{aligned}$$
i.e.,
\[
J_1 \leq \frac{C\gamma^2}{\beta}\mathcal{F}(f) + \frac\beta4D_1(f).
\]
Thus we have
\[
\frac{\gamma\sigma d}{\beta} \leq \frac{C\gamma^2}{\beta}\mathcal{F}(f) + \frac\beta4D_1(f) + \gamma D_3(f)  - \gamma\into \rho|u|^2\,dx.
\]
Now we combine the above estimates together with Lemma \ref{lem_energy} to obtain
$$\begin{aligned}
&\frac{d}{dt}\mf(f) + \frac\beta2\md_1(f) + \gamma \into \rho|u|^2\,dx + \frac\alpha2\intoo \phi(x-y)|u(x) - u(y)|^2\rho(x)\rho(y)\,dxdy\cr
&\qquad \leq \frac{C}{\beta}(1 + \gamma^2) \mf(f).
\end{aligned}$$
Applying Gronwall's inequality to the above concludes the desired first result. The second inequality just follows from the first result and the above inequality.
\end{proof}

%
%
%
%
%
%
\subsection{Main results}

For the hydrodynamic limit to isothermal/pressureless Euler system with nonlocal forces, we use the relative entropy argument. For this, we need to establish the existence of weak solutions to the equation \eqref{main_eq} and the existence of the unique strong solution to the system \eqref{main_pE} and \eqref{main_npE} at least locally in time. Thus we first present a notion of weak solutions of the equation \eqref{main_eq}.
\begin{definition}\label{def_weak} For a given $T \in (0,\infty)$, we say that $f$ is a weak solution to the equation \eqref{main_eq} if the following conditions are satisfied:
\begin{itemize}
\item[(i)] $f \in L^\infty(0,T;(L^1_+ \cap L^\infty)(\om \times \R^d))$,
\item[(ii)] for any $\varphi \in \mc^\infty_c(\om \times \R^d \times [0,T])$, 
$$\begin{aligned}
&\int_0^t \intor f(\pa_s\varphi + v \cdot \nabla_x \varphi - \lt(\gamma v + \lambda(\nabla_x V + \nabla_x W \star \rho) \rt) \cdot \nabla_v \varphi)\,dxdvds\cr
&\quad + \int_0^t \intor f(\lt(\alpha F[f]  + \beta(u-v) \rt)\cdot \nabla_v \varphi + \sigma \Delta_v \varphi)\,dxdvds + \intor f_0 \varphi(x,v,0)\,dxdv= 0.
\end{aligned}$$
\end{itemize}
\end{definition}

We next state definitions of strong solutions to the systems \eqref{main_pE} and \eqref{main_npE} below.

\begin{definition}\label{def_strong1} For given $T\in(0,\infty)$, the pair $(\rho,u)$ is a strong solution of \eqref{main_pE} on the time interval $[0,T]$ if and only if the following conditions are satisfied:
\begin{itemize}
\item[(i)] $(\rho, u) \in \mc([0,T];L^1(\om)) \times \mc([0,T];\mathcal{W}^{1,\infty}(\om))$,
\item[(ii)] $(\rho, u)$ satisfies the following free energy estimate in the sense of distributions:
$$\begin{aligned}
&\frac{d}{dt}\lt(\frac12\into \rho|u|^2\,dx  + \into \rho \log \rho\,dx + \lambda \into \rho V\,dx + \frac\lambda2\into (W \star \rho)\rho\,dx \rt) \cr
&\quad = - \gamma \into \rho|u|^2\,dx -\frac\alpha2 \intoo \phi(x-y)|u(x) - u(y)|^2 \rho(x) \rho(y)\,dxdy,
\end{aligned}$$
\item[(iii)] $(\rho, u)$ satisfies the system \eqref{main_pE} in the sense of distributions.
\end{itemize}
\end{definition}

\begin{definition}\label{def_strong2}For given $T\in(0,\infty)$, the pair $(\rho,u)$ is a strong solution of \eqref{main_npE} on the time interval $[0,T]$ if and only if the following conditions are satisfied:
\begin{itemize}
\item[(i)] $(\rho, u) \in \mc([0,T];L^1(\om)) \times \mc([0,T];\mw^{1,\infty}(\om))$,
\item[(ii)] $(\rho, u)$ satisfies the following free energy estimate in the sense of distributions:
$$\begin{aligned}
&\frac{d}{dt}\lt(\frac12\into \rho|u|^2\,dx + \lambda \into \rho V\,dx + \frac\lambda2\into (W \star \rho)\rho\,dx \rt) \cr
&\quad = - \gamma \into \rho|u|^2\,dx -\frac\alpha2 \intoo \phi(x-y)|u(x) - u(y)|^2 \rho(x) \rho(y)\,dxdy,
\end{aligned}$$
\item[(iii)] $(\rho, u)$ satisfies the system \eqref{main_npE} in the sense of distributions.
\end{itemize}
\end{definition}

Before providing our results on the hydrodynamic limits, we list our main assumptions on the initial data below.
\begin{itemize}
\item[{\bf (H1)}] The initial data related to the entropy are well-prepared:
\[
\rho_0^\e\lt( \log \rho^\e_0 - \log \rho_0 \rt) + (\rho_0 - \rho_0^\e) = \mathcal{O}(\sqrt\e) \quad \mbox{and} \quad \into\lt(\intr f_0^\e \log f_0^\e\,dv - \rho_0 \log \rho_0 \rt)dx = \mathcal{O}(\sqrt\e).
\]
\item[{\bf (H2)}] The initial data related to the kinetic energy part in the entropy are well-prepared:
\[
\into \rho_0^\e|u_0 - u^\e_0|^2\,dx = \mathcal{O}(\sqrt\e) \quad \mbox{and} \quad \into \lt( \intr f_0^\e |v|^2\,dv  - \rho_0|u_0|^2\rt)dx  = \mathcal{O}(\sqrt\e).
\]
\item[{\bf (H3)}] The bounded Lipschitz distance between initial local densities satisfies
\[
d^2_{BL}(\rho^\e_0, \rho_0) = \mathcal{O}(\sqrt\e).
\]
\end{itemize}
\begin{remark} If we choose the initial data $f_0^\e$ as
\[
f_0^\e(x,v) = \frac{\rho_0(x)}{(2\pi)^{d/2}}\exp\lt(-\frac{|u_0(x) - v|^2}{2}\rt) \quad \mbox{for all} \quad \e > 0,
\]
then we obtain
\[
\rho^\e_0 = \intr f^\e_0\,dv = \rho_0 \quad \mbox{and} \quad \rho^\e_0 u^\e_0 = \intr v f^\e_0\,dv = \intr u_0 f^\e_0\,dv =\rho_0 u_0.
\]
\end{remark}

Let us define the classical relative entropy between two probability densities $\rho_1$ and $\rho_2$ as
\begin{equation}\label{relennew}
\mathcal{H}(\rho_1|\rho_2)=  \int_{\rho_2}^{\rho_1} \frac{\rho_1 - z}{z}\,dz =  \rho_1  \log \rho_1 - \rho_2  \log \rho_2 - (1+\log \rho_2) (\rho_1-\rho_2)\,,
\end{equation}
and analogously for two densities $f_1$ and $f_2$ in phase space as
$$
\mathcal{H}(f_1|f_2)=  \int_{f_2}^{f_1} \frac{f_1 - z}{z}\,dz =  f_1  \log f_1 - f_2  \log f_2 - (1+\log f_2) (f_1-f_2)\,,
$$
\begin{remark}The first assumptions in {\bf (H1)} and {\bf (H2)} imply that
\[
\into \frac{\rho_0^\e}{2}|u_0^\e - u_0|^2\,dx +  \into\mathcal{H}(\rho_0^\e|\rho_0)dx = \mathcal{O}(\sqrt\e).
\]
\end{remark}

\begin{theorem}\label{thm_hydro1} Let $f^\e$ be a weak solution to the equation \eqref{main_eq} with $\beta = \sigma = 1/\e$ in the sense of Definition \ref{def_weak} and $(\rho,u)$ be a strong solution to the system \eqref{main_pE} in the sense of Definition \ref{def_strong1} up to the time $T^* > 0$. Suppose that the assumptions {\bf (H1)}--{\bf (H2)} hold. Then we have the following inequalities for $0< \e \leq 1$ and $t \leq T^*$:
\begin{itemize}
\item[(i)] Coulombian case $\Delta W = -\delta_0$:
$$\begin{aligned}
&\frac12\into \rho^\e |u^\e - u|^2\,dx + \into\mathcal{H}(\rho^\e|\rho)dx + \frac\lambda 2 \into |\nabla W \star (\rho - \rho^\e)|^2\,dx  + \gamma\int_0^t \into \rho^\e(x)| u^\e(x) - u(x)|^2\,dxds\cr
&\quad + \frac\alpha2\int_0^t \intoo \rho^\e(x) \rho^\e(y)\phi(x-y)|( u^\e(x) - u(x)) - (u^\e(y) - u(y))|^2 dxdyds\cr
&\qquad \leq C\sqrt{\e} + C \into |\nabla W \star (\rho_0 - \rho^\e_0)|^2\,dx,
\end{aligned}$$
\item[(ii)] Weakly regular case $\nabla W \in L^\infty(\om)$: 
$$\begin{aligned}
&\frac12\into \rho^\e |u^\e - u|^2\,dx +  \into\mathcal{H}(\rho^\e|\rho)dx  + \gamma\int_0^t \into \rho^\e(x)| u^\e(x) - u(x)|^2\,dxds\cr
&\quad + \frac\alpha2\int_0^t \intoo \rho^\e(x) \rho^\e(y)\phi(x-y)|( u^\e(x) - u(x)) - (u^\e(y) - u(y))|^2 dxdyds \leq C\sqrt{\e}.
\end{aligned}$$
\end{itemize}
Here $C>0$ is a positive constant independent of $\e$. 
\end{theorem}

\begin{remark}Coulombian interaction potential on $\R^d$ is explicitly given by
\[
W(x) = \left\{ \begin{array}{ll}
 -\frac{|x|}{2} & \textrm{for $d=1$,}\\[2mm]
 -\frac{1}{2\pi} \log |x| & \textrm{for $d=2$,}\\[2mm]
\frac{1}{(d-2)|B(0,1)|}\frac{1}{|x|^{d-2}} & \textrm{for $d \geq 3$},
  \end{array} \right.
\]
where $|B(0,1)|$ denotes the volume of unit ball $B(0,1)$ in $\R^d$, i.e., $|B(0,1)| = \pi^{d/2}/\Gamma(d/2+1)$.
\end{remark}

\begin{corollary}\label{cor_hydro1} Suppose that all the assumptions in Theorem \ref{thm_hydro1} hold. Then we have the following convergences hold for the weakly regular case (ii):
\begin{align}\label{thm_h1_conv}
\begin{aligned}
\rho^\e &\to \rho \quad \mbox{a.e.} \quad \mbox{and} \quad L^\infty(0,T^*;L^1(\om)),\cr
\rho^\e u^\e &\to \rho u \quad \mbox{a.e.} \quad \mbox{and} \quad L^\infty(0,T^*;L^1(\om)), \cr
\rho^\e u^\e \otimes u^\e &\to \rho u\otimes u \quad \mbox{a.e.} \quad \mbox{and} \quad L^\infty(0,T^*;L^1(\om)), \quad \mbox{and}\cr
\intr f^\e v\otimes v\,dv &\to \rho u\otimes u + \rho \mathbb{I}_{d \times d} \quad \mbox{a.e.} \quad \mbox{and} \quad L^p(0,T^*;L^1(\om)) \quad \mbox{for} \quad 1 \leq p \leq 2
\end{aligned}
\end{align}
as $\e \to 0$. The same convergences for the Coulombian case (i) can be obtained if
\[
 \into |\nabla W \star (\rho_0 - \rho^\e_0)|^2\,dx \to 0 \quad \mbox{as} \quad \e \to 0.
\]
\end{corollary}

In the corollary below, under suitable assumptions we provide the convergence of $f^\e$ towards the local Maxwellian $M_{\rho,u}$ given by
\[
M_{\rho, u}:= \frac{\rho}{(2\pi)^{d/2}}e^{-\frac{|u-v|^2}{2}},
\]
where $(\rho,u)$ is the strong solution to the system \eqref{main_pE}.

\begin{corollary}\label{cor_hydro1.5} Suppose that all the assumptions in Theorem \ref{thm_hydro1} hold. Moreover, we assume that the confinement potential $V$ satisfies $|\nabla V(x)|^2 \leq C|V(x)|$ for some $C>0$ and the solution $\rho$ to the limiting system is regular such that  $\nabla W \star \rho \in L^\infty(\om \times (0,T^*))$. Then for $t \leq T^*$, we have 
\[
\|f^\e - M_{\rho,u}\|_{L^1} \leq C  \lt(\intor  \mathcal{H}(f_0^\e|M_{\rho_0, u_0})\,dxdv\rt)^{1/2} + C\e^{1/8}
\]
for the weakly regular potential case (ii), and 
$$\begin{aligned}
\|f^\e - M_{\rho,u}\|_{L^1} &\leq C  \lt(\intor \mathcal{H}(f_0^\e|M_{\rho_0,u_0})\,dxdv\rt)^{1/2} + C\e^{1/8} + C\left( \min\left\{1,\int_\Omega |\nabla W\star(\rho_0^\e - \rho_0)|^2\,dx\right\}\right)^{1/4}
\end{aligned}$$
for the Coulombian potential case (i), where $C>0$ is independent of $\e>0$. In particular, if the right hand side of the above inequality convergences to zero, then we have
\[
f^\e \to M_{\rho, u}:= \frac{\rho}{(2\pi)^{d/2}}e^{-\frac{|u-v|^2}{2}} \quad \mbox{in }\ L^\infty(0,T^*; L^1(\Omega))
\]
as $\e \to 0$.
\end{corollary}

\begin{proof}Since this proof is lengthy and technical, we postpone it to Appendix \ref{app_cor}.
\end{proof}

\begin{remark}Note that the assumption on $V$ in Corollary \ref{cor_hydro1.5} holds for the quadratic confinement potential $V(x) = |x|^2/2$.
\end{remark}

\begin{theorem}\label{thm_hydro2} Let $f^\e$ be a weak solution to the equation \eqref{main_eq} with $\beta = 1/\e$ and $\sigma = 0$ in the sense of Definition \ref{def_weak} and $(\rho,u)$ be a strong solution to the system \eqref{main_npE} in the sense of Definition \ref{def_strong2} up to the time $T^* > 0$. Suppose that the assumptions {\bf (H2)}--{\bf (H3)} hold. Then we have the following inequalities for $0 < \e \leq 1$ and $t \leq T^*$:
\begin{itemize}
\item[(i)] Coulombian case $\Delta W = -\delta_0$:
$$\begin{aligned}
&\into \frac{\rho^\e}{2}|u^\e - u|^2\,dx + \frac\lambda 2 \into |\nabla W \star (\rho - \rho^\e)|^2\,dx +d_{BL}^2(\rho^\e, \rho) + \gamma\int_0^t \into \rho^\e(x)| u^\e(x) - u(x)|^2\,dxds\cr
&\quad + \frac\alpha2\int_0^t \intoo \rho^\e(x) \rho^\e(y)\phi(x-y)|( u^\e(x) - u(x)) - (u^\e(y) - u(y))|^2 dxdyds\cr
&\qquad \leq C\sqrt{\e} + C \into |\nabla W \star (\rho_0 - \rho^\e_0)|^2\,dx,
\end{aligned}$$
\item[(ii)] Strongly regular case $\nabla W \in \mw^{1,\infty}(\om)$: 
$$\begin{aligned}
&\into \frac{\rho^\e}{2}|u^\e - u|^2 \,dx  +d_{BL}^2(\rho^\e, \rho) + \gamma\int_0^t \into \rho^\e(x)| u^\e(x) - u(x)|^2\,dxds\cr
&\quad + \frac\alpha2\int_0^t \intoo \rho^\e(x) \rho^\e(y)\phi(x-y)|( u^\e(x) - u(x)) - (u^\e(y) - u(y))|^2 \,dxdyds\cr
&\qquad \leq C\sqrt{\e}.
\end{aligned}$$
\end{itemize}
Here $C>0$ is a positive constant independent of $\e$. 
\end{theorem}
\begin{remark}Compared to Theorem \ref{thm_hydro1} (ii), the pressureless case requires higher regularity for $W$, like $\nabla W \in \mathcal{W}^{1,\infty}(\om)$ due to the lack of convexity of the entropy with respect to $\rho$.
\end{remark}

\begin{corollary}\label{cor_hydro2} Suppose that all the assumptions in Theorem \ref{thm_hydro2} hold. If 
\[
\into |\nabla W \star (\rho_0 - \rho^\e_0)|^2\,dx \to 0 \quad \mbox{as} \quad \e \to 0
\]
for Coulombian interaction case, then the following convergences hold:
$$\begin{aligned}
\rho^\e u^\e &\rightharpoonup \rho u \quad \mbox{weakly in } L^\infty(0,T^*;\mathcal{M}), \cr
\rho^\e u^\e \otimes u^\e &\rightharpoonup \rho u \otimes u \quad \mbox{weakly in } L^\infty(0,T^*;\mathcal{M}),\cr
\intr f^\e v \otimes v\,dv &\rightharpoonup \rho u \otimes u \quad \mbox{weakly in } L^1(0,T^*;\mathcal{M}), \quad \mbox{and} \cr
f^\e &\rightharpoonup \rho \otimes \delta_u  \quad \mbox{weakly in } L^p(0,T^*;\mathcal{M})
\end{aligned}$$
as $\e \to 0$, for $1 \leq p \leq 2$. Here $\mathcal{M}$ is the space of nonnegative Radon measures.
\end{corollary}

\begin{remark} The convergence of $d_{BL}(\rho^\e, \rho)$ directly gives 
\[
\rho^\e \rightharpoonup \rho \quad \mbox{weakly in } L^\infty(0,T^*;\mathcal{M}).
\]
\end{remark}

\begin{remark} Our results on the hydrodynamic limit also hold in a bounded domain with the specular reflection boundary condition. In this case, the limiting system has a kinematic boundary condition. Concerning this, we provide the existence theory in Section \ref{sec:ibv}. For the hydrodynamic limit estimate, we refer to \cite{CJpre2} where the hydrodynamic limit of nonlinear Vlasov--Fokker--Planck/Navier--Stokes equations in a bounded domain is discussed. 
\end{remark}

\subsection{Proofs of Corollaries \ref{cor_hydro1} and \ref{cor_hydro2}} Before proceeding, for the readers' convenience, we provide the details of proofs of convergences in Corollaries \ref{cor_hydro1} and \ref{cor_hydro2}. In fact, we provide quantitative bounds of convergences. 

\begin{lemma}\label{lem_conv} There exists a positive constant $C$ depending only on $\|u\|_{\mw^{1,\infty}}$ such that the following inequalities hold.
\begin{itemize}
\item[(i)] Error estimate between moments:
\[
\|\rho^\e u^\e - \rho u\|_{L^1} \leq \|\rho^\e\|_{L^1}^{1/2}\lt(\into \rho^\e |u^\e - u|^2\,dx \rt)^{1/2} + \|u\|_{L^\infty}\|\rho^\e - \rho\|_{L^1}
\]
and
\[
d_{BL}(\rho^\e u^\e, \rho u) \leq \|\rho^\e\|_{L^1}^{1/2}\lt(\into \rho^\e |u^\e - u|^2\,dx \rt)^{1/2} + Cd_{BL}(\rho^\e, \rho).
\]
\item[(ii)] Error estimate between convections:
$$\begin{aligned}
\|\rho^\e u^\e \otimes u^\e - \rho u \otimes u\|_{L^1} &\leq \into \rho^\e |u^\e - u|^2\,dx + 2\|u\|_{L^\infty}\|\rho^\e\|_{L^1}^{1/2}\lt(\into \rho^\e |u^\e - u|^2\,dx \rt)^{1/2}\cr
&\quad  + 3\|u\|_{L^\infty}^2\|\rho^\e - \rho\|_{L^1}
\end{aligned}$$
and
\[
d_{BL}(\rho^\e u^\e \otimes u^\e, \rho u \otimes u) \leq \into \rho^\e |u^\e - u|^2\,dx + C\|\rho^\e\|_{L^1}^{1/2}\lt(\into \rho^\e |u^\e - u|^2\,dx \rt)^{1/2} + Cd_{BL}(\rho^\e, \rho).
\]
\item[(iii)] Error estimate between particle distribution and mono-kinetic ansatz:
$$\begin{aligned}
d_{BL}(f^\e, \rho \otimes \delta_{u}) &\leq \|\rho^\e\|_{L^1}^{1/2}\lt( \lt(\intor |v - u^\e|^2 f^\e\,dxdv\rt)^{1/2} + \lt(\into  \rho^\e |u^\e - u|^2 \,dx\rt)^{1/2}\rt) \cr
&\quad +C  d_{BL}(\rho^\e, \rho).
\end{aligned}$$
\end{itemize}
\end{lemma}
\begin{proof} For any $\psi \in (L^\infty \cap Lip)(\om)$, we get
$$\begin{aligned}
&\lt| \into \psi(x) \lt((\rho^\e u^\e)(x) - (\rho u)(x) \rt)dx \rt|\cr
&\quad =\lt| \into \psi(x) \lt(\rho^\e(x)(u^\e - u)(x) - (\rho^\e - \rho)(x) u(x) \rt)dx \rt|\cr
&\quad \leq \|\psi\|_{L^\infty}\into \rho^\e |u^\e - u|\,dx + \lt|\into (\rho^\e - \rho)(x) (\psi u)(x)\,dx \rt|\cr
&\quad \leq \|\psi\|_{L^\infty}\|\rho^\e\|_{L^1}^{1/2}\lt(\into \rho^\e |u^\e - u|^2\,dx \rt)^{1/2} + \|\psi u\|_{L^\infty \cap Lip} \,d_{BL}(\rho^\e, \rho).
\end{aligned}$$
This asserts the inequality (i). For the estimate of (ii), we notice that
$$\begin{aligned}
\rho^\e u^\e \otimes u^\e - \rho u \otimes u &= \rho^\e (u^\e - u) \otimes (u^\e - u) + u \otimes \lt( \rho^\e u^\e - \rho u \rt)\cr
&\quad + \lt( \rho^\e u^\e - \rho u \rt) \otimes u - (\rho^\e - \rho) u\otimes u.
\end{aligned}$$
Using this identity, we obtain
$$\begin{aligned}
&\lt| \into \psi(x) \lt((\rho^\e u^\e \otimes u^\e)(x) - (\rho u \otimes u)(x) \rt)dx \rt|\cr
&\quad \leq \lt| \into \psi(x) \rho^\e(x) (u^\e - u)(x) \otimes (u^\e - u)(x)\,  dx \rt|+ \lt|\into \psi(x)u(x) \otimes \lt( \rho^\e u^\e - \rho u \rt)(x) \,dx \rt|\cr
&\qquad + \lt|\into \psi(x) \lt( \rho^\e u^\e - \rho u \rt)(x) \otimes u(x) \,dx \rt| +  \lt|\into \psi(x) (\rho^\e - \rho)(x) u(x)\otimes u(x) \,dx \rt|\cr
&\quad \leq \|\psi\|_{L^\infty}\into \rho^\e |u^\e - u|^2\,dx + 2\|\psi u\|_{L^\infty \cap Lip}\,d_{BL}(\rho^\e u^\e, \rho u) + \|\psi u\otimes u\|_{L^\infty \cap Lip} \,d_{BL}(\rho^\e, \rho).
\end{aligned}$$
This yields
\[
d_{BL}(\rho^\e u^\e \otimes u^\e, \rho u \otimes u) \leq \into \rho^\e |u^\e - u|^2\,dx + C\|\rho^\e\|_{L^1}^{1/2}\lt(\into \rho^\e |u^\e - u|^2\,dx \rt)^{1/2} + Cd_{BL}(\rho^\e, \rho),
\]
where $C>0$ depends only on $\|u\|_{\mw^{1,\infty}}$. For (iii), we find for any $\varphi \in (L^\infty \cap Lip)(\om \times \R^d)$ that
$$\begin{aligned}
&\lt| \intor \varphi(x,v) (f^\e(x,v) - \rho(x) \otimes \delta_{u(x)}(v))\,dxdv \rt|\cr
&\quad = \lt| \intor \varphi(x,v) f^\e(x,v)\,dxdv - \into \varphi(x,u(x))\rho(x) \,dx \rt|\cr
&\quad \leq \intor |\varphi(x,v) - \varphi(x,u^\e(x))|f^\e\,dxdv + \into |\varphi(x,u^\e) - \varphi(x,u)|\rho^\e\,dx + \lt|\into \varphi(x,u(x)) (\rho^\e - \rho)\,dx\rt|\cr
&\quad \leq \|\varphi\|_{Lip} \intor |v - u^\e| f^\e\,dxdv + \|\varphi\|_{Lip} \into \rho^\e |u^\e - u| \,dx + \|\varphi\|_{Lip}\|u\|_{Lip} \, d_{BL}(\rho^\e, \rho)\cr
&\quad \leq \|\varphi\|_{Lip}\|\rho^\e\|_{L^1}^{1/2}\lt( \lt(\intor |v - u^\e|^2 f^\e\,dxdv\rt)^{1/2} + \lt(\into \rho^\e |u^\e - u|^2 \,dx\rt)^{1/2}\rt) \cr
&\qquad + \|\varphi\|_{Lip}\|u\|_{Lip} \, d_{BL}(\rho^\e, \rho).
\end{aligned}$$
This concludes the inequality (iii).
\end{proof}

\begin{proof}[Proof of Corollary \ref{cor_hydro1}] We first obtain
\[
\|\rho^\e - \rho\|_{L^1} \leq C \lt(\into\mathcal{H}(\rho|\rho^\e)\,dx\rt)^{1/2},
\]
where $C>0$ depends only on $\|\rho^\e\|_{L^1}$ and $\|\rho\|_{L^1}$, see \eqref{est_l1} for details. This together with Lemma \ref{lem_conv} yields
\begin{align}\label{est_conv2}
\begin{aligned}
&\|\rho^\e - \rho\|_{L^1} + \|\rho^\e u^\e - \rho u\|_{L^1} + \|\rho^\e u^\e \otimes u^\e - \rho u \otimes u\|_{L^1}
\leq \into \rho^\e |u^\e - u|^2\,dx + C\lt(\into\mathcal{H}(\rho|\rho^\e)\,dx\rt)^{1/2}\cr
&\quad \leq C\e^{1/4} + C \lt(\into |\nabla W \star (\rho_0 - \rho^\e_0)|^2\,dx\rt)^{1/2} \to 0
\end{aligned}
\end{align}
as $\e \to 0$, where $C>0$ is independent of $\e$. Note that
$$\begin{aligned}
&\intr f^\e v\otimes v\,dv - (\rho u\otimes u + \rho \mathbb{I}_{d \times d}) \cr
&\quad = \intr f^\e v\otimes v\,dv - (\rho^\e u^\e\otimes u^\e + \rho^\e \mathbb{I}_{d \times d} ) + \rho^\e u^\e \otimes u^\e - \rho u\otimes u + (\rho^\e - \rho)\mathbb{I}_{d \times d}.
\end{aligned}$$
On the other hand, we find from \cite[Lemma 4.8]{KMT15} or \cite[Section 3]{CJpre2} that
\begin{align}\label{est_conv0}
\begin{aligned}
&\intr (u^\e \otimes u^\e - v \otimes v + \mathbb{I}_{d \times d})f^\e\,dv\cr
&\quad = \intr u^\e \sqrt{f^\e} \otimes \lt( (u^\e - v)\sqrt{f^\e} - 2\nabla_v \sqrt{f^\e} \rt) dv + \intr \lt( (u^\e - v)\sqrt{f^\e} - 2\nabla_v \sqrt{f^\e} \rt) \otimes v\sqrt{f^\e}\,dv.
\end{aligned}
\end{align}
This yields 
$$\begin{aligned}
&\lt\|\intr f^\e v\otimes v\,dv - (\rho^\e u^\e\otimes u^\e + \rho^\e \mathbb{I}_{d \times d} )\rt\|_{L^1}\cr
&\quad \leq \lt(\intor f^\e |u^\e|^2 + f^\e |v|^2\,dxdv \rt)^{1/2}\lt(\intor \frac{1}{f^\e}|\nabla_v f^\e - (u^\e - v)f^\e|^2\,dxdv \rt)^{1/2}\cr
&\quad \leq C\sqrt\e \sup_{0 \leq t \leq T} \lt(\intor  f^\e |v|^2\,dxdv \rt)^{1/2}\lt(\frac{1}{2\e} \intor \frac{1}{f^\e}|\nabla_v f^\e - (u^\e - v)f^\e|^2\,dxdv  \rt)^{1/2}\cr
&\quad \leq C\sqrt\e\lt(\frac{1}{2\e} \intor \frac{1}{f^\e}|\nabla_v f^\e - (u^\e - v)f^\e|^2\,dxdv  \rt)^{1/2}.
\end{aligned}$$
Combining this, \eqref{est_conv2}, and Proposition \ref{prop_energy} with $\beta = \sigma = 1/\e$, we have
$$\begin{aligned}
&\lt\|\intr f^\e v\otimes v\,dv - (\rho u\otimes u + \rho \mathbb{I}_{d \times d})\rt\|_{L^2(0,T^*;L^1(\om))}\cr
&\quad \leq \lt\|\intr f^\e v\otimes v\,dv - (\rho^\e u^\e\otimes u^\e + \rho^\e \mathbb{I}_{d \times d} )\rt\|_{L^2(0,T^;L^1(\om))} + \|\rho^\e u^\e \otimes u^\e - \rho u\otimes u\|_{L^2(0,T^;L^1(\om))}\cr
&\qquad  + \|\rho^\e - \rho\|_{L^2(0,T^;L^1(\om))}\cr
&\quad \leq C\e^{1/4} + C \lt(\into |\nabla W \star (\rho_0 - \rho^\e_0)|^2\,dx\rt)^{1/2} \to 0.
\end{aligned}$$
This completes the proof.
\end{proof}

\begin{proof}[Proof of Corollary \ref{cor_hydro2}] A simple combination of inequalities in Lemma \ref{lem_conv} together with Theorem \ref{thm_hydro2} gives
$$\begin{aligned}
&d_{BL}(\rho^\e u^\e, \rho u) + d_{BL}(\rho^\e u^\e \otimes u^\e, \rho u \otimes u)\cr
&\quad \leq C\into \rho^\e |u^\e - u|^2\,dx + C\lt(\into \rho^\e |u^\e - u|^2\,dx \rt)^{1/2} + Cd_{BL}(\rho^\e, \rho)\cr
&\quad \leq C\e^{1/4} + C \lt(\into |\nabla W \star (\rho_0 - \rho^\e_0)|^2\,dx\rt)^{1/2}.
\end{aligned}$$
This asserts the first two convergences. Note that 
\[
\intr f^\e v \otimes v\,dv - \rho^\e u^\e \otimes u^\e = \intr f^\e (u^\e - v) \otimes (u^\e - v)\,dv,
\]
thus we get
\[
\intr f^\e v \otimes v\,dv - \rho u\otimes u = \intr f^\e v \otimes v\,dv - \rho^\e u^\e \otimes u^\e + \rho^\e u^\e \otimes u^\e - \rho u\otimes u.
\]
This yields
\bq\label{est_conv1}
d_{BL} \lt(\intr f^\e v \otimes v\,dv, \rho u\otimes u \rt) \leq \intor f^\e |u^\e - v|^2\,dxdv + d_{BL}(\rho^\e u^\e \otimes u^\e, \rho u \otimes u).
\eq
On the other hand, it follows from  \eqref{energy_zerosig} with $\beta = 1/\e$ that 
\[
\int_0^t \intor f^\e |u^\e - v|^2\,dxdvds \leq C\e.
\]
This together with \eqref{est_conv1} implies the third assertion. We also use Lemma \ref{lem_conv} and \eqref{energy_zerosig} with $\beta = 1/\e$ to conclude that  for $1 \leq p \leq 2$
$$\begin{aligned}
&\int_0^t d^p_{BL}(f^\e(\cdot,\cdot,s), \rho(\cdot,s) \otimes \delta_{u(\cdot,s)})\,ds \cr
&\quad \leq C\e^{1/4} +  C\lt(\int_0^t\intor |v - u^\e|^2 f^\e\,dxdvds\rt)^{1/2} + C \lt(\into |\nabla W \star (\rho_0 - \rho^\e_0)|^2\,dx\rt)^{1/2} \to 0
\end{aligned}$$
as $\e \to 0$.
\end{proof}

%
%
%
%
\section{Hydrodynamic limit from kinetic to isothermal Euler equations}\label{sec_pressure}
In this section, we study the rigorous derivation of the isothermal Euler equations \eqref{main_pE} from the kinetic equation \eqref{main_eq} with $\beta = \sigma = 1/\e$ as $\e \to 0$. As mentioned before, we use the relative entropy argument based on the weak-strong uniqueness principle to have the quantitative error estimates between the kinetic equation and the limiting system. 
\subsection{Relative entropy inequality}
We rewrite the equations as a conservative form:
\bq\label{sys_cons}
\pa_t U + \nabla \cdot A(U) = F(U),
\eq
where 
\[
 U := \begin{pmatrix}
\rho \\
m 
\end{pmatrix} 
\quad \mbox{with} \quad m = \rho u, \quad
A(U) := \begin{pmatrix}
m  & 0 \\
\displaystyle \frac{m \otimes m}{\rho} & \rho \mathbb{I}_{d \times d}
\end{pmatrix},
\]
and
\[
F(U) := \begin{pmatrix}
0 \\
\displaystyle \alpha \rho \into \phi(x-y)(u(y) - u(x))\rho(y)\,dy  -\gamma \rho u - \lambda \rho\lt(\nabla V + \nabla W \star \rho \rt)
\end{pmatrix}.
\]
Here $\mathbb{I}_{d \times d}$ denotes the $d \times d$ identity matrix. The free energy of the above system is given by
\bq\label{ent_mac}
E(U) := \frac{|m|^2}{2\rho} + \rho \log \rho.
\eq
We now define the relative entropy functional $\me$ between two states of the system $U$ and $\bar U$ as follows.
\bq\label{def_rel}
\me(\bar U|U) := E(\bar U) - E(U) - DE(U)(\bar U-U) \quad \mbox{with} \quad \bar U := \begin{pmatrix}
        \bar\rho \\
        \bar m \\
    \end{pmatrix}, \quad \bar m = \bar\rho \bar u,
\eq
where $D E(U)$ denotes the derivation of $E$ with respect to $\rho, m$, i.e.,
$$\begin{aligned}
-DE(U)(\bar U - U) &= -\begin{pmatrix}
\displaystyle        -\frac{|m|^2}{2\rho^2} & \log \rho + 1\\[3mm]
\displaystyle        \frac{m}{\rho} & 0
    \end{pmatrix}
    \begin{pmatrix}
    \bar\rho - \rho \\
    \bar m - m
    \end{pmatrix}\\
    &= \frac{\bar\rho |u|^2}{2} - \frac{\rho|u|^2}{2} + (\rho - \bar\rho)(\log \rho + 1) + \rho |u|^2 - \bar\rho u \cdot \bar u.
\end{aligned}$$
This yields 
$$\begin{aligned}
\me(\bar U|U) &= \frac{\bar\rho|\bar u|^2}{2} - \frac{\rho|u|^2}{2} +\bar\rho \log \bar\rho - \rho \log \rho + \frac{\bar\rho |u|^2}{2} - \frac{\rho|u|^2}{2} + (\rho - \bar\rho)(\log \rho + 1) + \rho |u|^2 - \bar\rho u \cdot \bar u\cr
&= \frac{\bar\rho}{2}|\bar u - u|^2 + \mh(\bar\rho| \rho),
\end{aligned}$$
where $\mh(\bar\rho | \rho)$ is the relative entropy between densities given by \eqref{relennew}.
By Taylor's theorem, we readily see
\[
\mh(\bar\rho| \rho) \geq \frac12 \min\lt\{\frac{1}{\bar\rho}, \frac{1}{\rho} \rt\}(\rho - \bar\rho)^2,
\]
and moreover, we get
$$\begin{aligned}
\|\bar\rho - \rho\|_{L^1} &= \into \min\lt\{ (\sqrt{\bar\rho})^{-1}, (\sqrt{\rho})^{-1}  \rt\} \max\lt\{ \sqrt{\bar\rho}, \sqrt\rho \rt\} | \bar\rho - \rho|\,dx\cr
&\leq \lt(\frac12\into \min\lt\{ (\bar\rho)^{-1}, \rho^{-1}  \rt\}(\rho - \bar\rho)^2\,dx \rt)^{1/2}\lt(2\into\max\{ \bar\rho, \rho\}\,dx  \rt)^{1/2}\cr
&\leq \lt(\into \mh(\bar\rho| \rho)\,dx\rt)^{1/2}\lt(2(\|\bar\rho\|_{L^1} + \|\rho\|_{L^1}) \rt)^{1/2}.
\end{aligned}$$
Thus we obtain
\bq\label{est_l1}
\|\bar\rho - \rho\|_{L^1}^2 \leq C\into \mh(\bar\rho| \rho)\,dx \leq C\into \me(\bar U|U)\,dx,
\eq
where $C>0$ only depends on $\|\bar\rho\|_{L^1}$ and $\|\rho\|_{L^1}$.

\begin{remark}The free energy of the system \eqref{sys_cons} is given by
\[
\tilde E(U) = \frac{|m|^2}{2\rho} + \rho \log \rho + \lambda \rho V + \frac{\lambda}{2} \rho W \star \rho,
\]
and we can also define its modulated energy, also often called the relative entropy, as
\[
\tilde \me(\bar U|U) := \tilde E(\bar U) - \tilde E(U) - D\tilde E(U)(\bar U-U).
\]
A straightforward computation shows
\[
\tilde \me(\bar U|U) = \frac{\bar\rho}{2}|\bar u - u|^2 + \mh(\bar\rho| \rho) + \frac\lambda2 (\rho - \bar\rho) W\star\rho + \frac\lambda2 \bar\rho W \star(\bar\rho - \rho),
\]
and by symmetry of $W$, we obtain
\[
\into \tilde \me(\bar U|U)\,dx = \into \frac{\bar\rho}{2}|\bar u - u|^2\,dx + \into \mh(\bar\rho| \rho)\,dx + \frac\lambda2\into (\rho - \bar\rho) W\star(\rho - \bar\rho)\,dx.
\]
This functional $\tilde \me$ is used in the study of large friction limit of Euler equations with nonlocal forces \cite{CPWpre,LT13,LT17}, see also \cite{Cpre} for the pressureless case. However, we employ the form \eqref{ent_mac} to use the estimates in \cite{KMT15} providing the relation between $E(U)$ and the flux $A(U)$, see the estimate of $I_3$ in the proof of Lemma \ref{lem_rel} below.
\end{remark}

\begin{lemma}\label{lem_rel}The relative entropy $\me$ defined in \eqref{def_rel} satisfies the following equality:
$$\begin{aligned}
\begin{aligned}
&\frac{d}{dt}\into \me(\bar U|U)\,dx + \frac\alpha2\intoo \bar\rho(x) \bar\rho(y)\phi(x-y)|(\bar u(x) - u(x)) - (\bar u(y) - u(y))|^2 dxdy\cr
&\quad = \into \partial_t E(\bar U)\,dx - \into \nabla (DE(U)):A(\bar U|U)\,dx - \into DE(U)\left[ \pa_t \bar U + \nabla \cdot A(\bar U) - F(\bar U)\right]dx\cr
&\qquad +\frac\alpha2\intoo \bar \rho(x) \bar \rho(y)\phi(x-y)|\bar u(x) - \bar u(y)|^2\,dxdy\cr
&\qquad - \alpha\intoo \bar\rho(x) (\rho(y) - \bar \rho(y))\phi(x-y) (\bar u(x) - u(x))\cdot (u(y) - u(x))\,dxdy\cr
&\qquad -\gamma \into \bar \rho| \bar u - u|^2 - \bar \rho |\bar u|^2\,dx + \lambda \into \nabla V \cdot \bar \rho \bar u\,dx\cr
&\qquad + \lambda \into \bar \rho (\bar u - u) \cdot \nabla W \star (\rho - \bar\rho) + \bar \rho \bar u \cdot \nabla W \star \bar\rho\,dx,
\end{aligned}
\end{aligned}$$
where $A(\bar U|U)$ is the relative flux functional given by
\[
A(\bar U|U) := A(\bar U) - A(U) - DA(U)(\bar U-U).
\]
\end{lemma}
\begin{proof}It follows from \eqref{def_rel} that
$$\begin{aligned}
\begin{aligned}
\frac{d}{dt}\into \me(\bar U|U)\,dx & = \into \partial_t E(\bar U)\,dx - \into DE(U)(\pa_t \bar U + \nabla \cdot A(\bar U)- F(\bar U))\,dx \cr
&\quad +\into D^2E(U) \nabla \cdot A(U)(\bar U-U) + DE(U) \nabla \cdot A(\bar U)\,dx\cr
&\quad -\into D^2E(U)F(U)(\bar U-U) + DE(U)F(\bar U)\,dx\cr
&=: \sum_{i=1}^4 I_i.
\end{aligned}
\end{aligned}$$
We first use the integration by parts \cite[Lemma 4.1]{KMT15} to get
\[
\into D^2E(U) \nabla \cdot A(U)(\bar U-U)\,dx = \into \nabla DE(U): DA(U) (\bar U-U)\,dx.
\]
Furthermore, we use the following identity \cite[Proof of Proposition 4.2]{KMT15} 
\[
\into \nabla DE(U): A(U)\,dx = 0
\]
to yield
$$\begin{aligned}
\begin{aligned}
I_3 &= \into \left(\nabla DE(U)\right):\left( DA(U)(\bar U-U) - A(\bar U)\right)dx  \cr
&= -\into \left(\nabla DE(U)\right):\left(A(\bar U|U) + A(U)\right)dx\cr
&=-\into \left(\nabla DE(U)\right):A(\bar U|U)\,dx.
\end{aligned}
\end{aligned}$$
For the estimate $I_4$, we notice that
\begin{equation*}
    DE(U) = \begin{pmatrix}
\displaystyle       -\frac{|m|^2}{2\rho^2} + \log \rho + 1 \\[4mm]
\displaystyle        \frac{m}{\rho}
    \end{pmatrix}
    \quad \mbox{and} \quad
    D^2E(U) = \begin{pmatrix}
    * & \displaystyle - \frac{m}{\rho^2}  \\[4mm]
        * & \displaystyle \frac{1}{\rho} 
    \end{pmatrix}.
\end{equation*}
Then, by direct calculation, we find
$$\begin{aligned}
&D^2E(U)F(U)(\bar U-U) \cr
&\quad = \bar \rho(x) (\bar u(x) - u(x))\cdot\lt(\alpha \into \phi(x-y)(u(y) - u(x)) \rho(y)\,dy - \lambda\lt(u + \nabla V + \nabla W \star \rho \rt) \rt) 
\end{aligned}$$
and
\[
DE(U)F(\bar U) = \bar \rho u \cdot \lt( \alpha \into \phi(x-y)(\bar u(y) - \bar u(x)) \bar\rho(y)\,dy - \lambda\lt(\bar u + \nabla V + \nabla W \star \bar \rho \rt) \rt).
\]
Thus we obtain
$$\begin{aligned}
-I_4 &= \alpha\intoo \bar \rho(x)\phi(x-y) (\bar u(x) - u(x))\cdot (u(y) - u(x)) \rho(y)\,dydx\cr
&\quad + \alpha \intoo \bar \rho(x) \phi(x-y) u(x) \cdot (\bar u(y) - \bar u(x)) \bar\rho(y)\,dydx\cr
&\quad -\into  \bar \rho(x) (\bar u(x) - u(x))\cdot\lt(\gamma u(x) + \lambda(\nabla V(x) + (\nabla W \star \rho)(x)) \rt) dx \cr
&\quad -\lambda \into  \bar \rho(x) u(x)\cdot\lt(\bar u(x) + \nabla V(x) + (\nabla W \star \bar \rho)(x) \rt) dx \cr
&=: \sum_{i=1}^4 I_4^i.
\end{aligned}$$
Here we follow the same argument as in \cite{KMT15} to get
$$\begin{aligned}
I_4^1 + I_4^2 &= \frac\alpha2\intoo \bar\rho(x) \bar\rho(y)\phi(x-y)|(\bar u(x) - u(x)) - (\bar u(y) - u(y))|^2 \,dxdy\cr
&\quad -\frac\alpha2\intoo \bar \rho(x) \bar \rho(y)\phi(x-y)|\bar u(x) - \bar u(y)|^2\,dxdy \cr
&\quad + \alpha\intoo \bar\rho(x) (\rho(y) - \bar \rho(y))\phi(x-y) (\bar u(x) - u(x))\cdot (u(y) - u(x))\,dxdy.
\end{aligned}$$
We next estimate $I_4^3 + I_4^4$ as 
$$\begin{aligned}
I_4^3 + I_4^4 &= \gamma \into \bar \rho| \bar u - u|^2 - \bar \rho |\bar u|^2\,dx - \lambda \into \nabla V \cdot \bar \rho \bar u\,dx\cr
&\quad - \lambda \into \bar \rho (\bar u - u) \cdot \nabla W \star (\rho - \bar\rho) + \bar \rho \bar u \cdot \nabla W \star \bar\rho\,dx.
\end{aligned}$$
Combining the above estimates yields
$$\begin{aligned}
I_4 &= -\frac\alpha2\intoo \bar\rho(x) \bar\rho(y)\phi(x-y)|(\bar u(x) - u(x)) - (\bar u(y) - u(y))|^2 dxdy\cr
&\quad +\frac\alpha2\intoo \bar \rho(x) \bar \rho(y)\phi(x-y)|\bar u(x) - \bar u(y)|^2\,dxdy \cr
&\quad - \alpha\intoo \bar\rho(x) (\rho(y) - \bar \rho(y))\phi(x-y) (\bar u(x) - u(x))\cdot (u(y) - u(x))\,dxdy\cr
&\quad -\gamma \into \bar \rho| \bar u - u|^2 - \bar \rho |\bar u|^2\,dx + \lambda \into \nabla V \cdot \bar \rho \bar u\,dx\cr
&\quad + \lambda \into \bar \rho (\bar u - u) \cdot \nabla W \star (\rho - \bar\rho) + \bar \rho \bar u \cdot \nabla W \star \bar\rho\,dx.
\end{aligned}$$
This completes the proof.
\end{proof}
We now set 
\[
m^\e = \rho^\e u^\e \quad \mbox{and} \quad U^\e = \begin{pmatrix} \rho^\e \\ m^\e  \end{pmatrix} \quad \mbox{with} \quad \rho^\e = \intr f^\e\,dv, \quad m^\e = \intr v f^\e\,dv,
\]
where $f^\e$ is a weak solution to the equation \eqref{main_eq}.
\begin{proposition}\label{prop_re}Let $f^\e$ be a global weak solution to the equation \eqref{main_eq} and $(\rho,u)$ be a strong solution to the system \eqref{main_pE} on the time interval $[0,T]$. Then we have
$$\begin{aligned}
&\into \me(U^\e|U)\,dx + \gamma\int_0^t \into \rho^\e(x)| u^\e(x) - u(x)|^2\,dxds\cr
&\quad + \frac\alpha2\int_0^t \intoo \rho^\e(x) \rho^\e(y)\phi(x-y)|( u^\e(x) - u(x)) - (u^\e(y) - u(y))|^2 dxdyds\cr
&\qquad \leq C\sqrt{\e} + C(1 + \alpha)\int_0^t \into \me(U^\e|U)\,dxds  \cr
&\qquad \quad + \lambda\int_0^t \into \rho^\e(x) ( u^\e(x) - u(x)) \cdot (\nabla W \star (\rho - \rho^\e))(x)\,dxds
\end{aligned}$$
for $0 < \e \leq 1$, where $C>0$ is independent of $\e$.
\end{proposition}
\begin{proof}It follows from Lemma \ref{lem_rel} that
\begin{align*}
&\into \me(U^\e|U)\,dx + \gamma\int_0^t \into \rho^\e(x)| u^\e(x) - u(x)|^2\,dxds\cr 
&\quad + \frac\alpha2\int_0^t \intoo \rho^\e(x) \rho^\e(y)\phi(x-y)|( u^\e(x) - u(x)) - (u^\e(y) - u(y))|^2 dxdyds\cr
&\qquad = \into \me(U_0^\e|U_0)\,dx + \into E(U^\e) - E(U^\e_0)\,dx - \int_0^t \into \nabla (DE(U)):A(U^\e|U)\,dxds \cr
&\quad \qquad - \int_0^t \into DE(U)\left[ \pa_s U^\e + \nabla \cdot A(U^\e) - F(U^\e)\right]dxds \cr
&\quad \qquad +\frac\alpha2\int_0^t \intoo \rho^\e(x) \rho^\e(y)\phi(x-y)|u^\e(x) - u^\e(y)|^2\,dxdyds\cr
&\quad \qquad - \alpha\int_0^t \intoo \rho^\e(x) (\rho(y) - \rho^\e(y))\phi(x-y) (u^\e(x) - u(x))\cdot (u(y) - u(x))\,dxdyds\cr
&\quad \qquad + \gamma\int_0^t \into \rho^\e(x) |u^\e(x)|^2\,dxds + \lambda \int_0^t \into \nabla V(x) \cdot \rho^\e(x) u^\e(x)\,dxds\cr
&\quad \qquad + \lambda\int_0^t \into \rho^\e(x) ( u^\e(x) - u(x)) \cdot (\nabla W \star (\rho - \rho^\e))(x) +  \rho^\e(x) u^\e(x) \cdot (\nabla W \star \rho^\e)(x)\,dxds\cr
&\qquad =: \sum_{i=1}^8 J_i^\e.
\end{align*}
Here $J_i^\e, i =1,\cdots,8$ can be estimated as follows. \newline

\noindent {\bf Estimate of $J_1^\e$}: By the assumption, we get
\[
J_1^\e = \mathcal{O}(\sqrt\e).
\]
\noindent {\bf Estimate of $J_2^\e$}: Note that 
\[
E(U^\e) = \frac12\into \rho^\e|u^\e|^2\,dx + \into \rho \log \rho\,dx \leq \frac12\intor |v|^2 f^\e\,dxdv + \intor f^\e \log f^\e\,dxdv =: K(f^\e).
\]
Thus, by adding and subtracting the functional $K(f^\e)$, we find
$$\begin{aligned}
J_2^\e &= \into E(U^\e)\,dx - K(f^\e) + K(f^\e) - K(f^\e_0) + K(f^\e_0) - \into E(U^\e_0)\,dx\cr
&\leq 0 + K(f^\e) - K(f^\e_0) + \mathcal{O}(\sqrt\e).
\end{aligned}$$
\noindent {\bf Estimate of $J_3^\e$}: It follows from \cite[Lemma 4.3]{KMT15} that
\[
A(U^\e|U) = \begin{pmatrix}
    0 & 0  \\[4mm]
        \rho^\e(u^\e - u) \otimes (u^\e - u) & 0
    \end{pmatrix}.
\]
This yields
$$\begin{aligned}
|J_3^\e | &= \lt|\int_0^t \into \nabla u : \rho^\e(u^\e - u) \otimes (u^\e - u)\,dxds \rt|\cr
&\leq \|\nabla u\|_{L^\infty}\int_0^t \into \rho^\e |u^\e - u|^2\,dxds\cr
&\leq  \|\nabla u\|_{L^\infty}\int_0^t \into \me(U^\e|U)\,dxds.
\end{aligned}$$

\noindent {\bf Estimate of $J_4^\e$}: Note that $U^\e$ satisfies
\[
\pa_t U^\e + \nabla \cdot A(U^\e) - F(U^\e) = \begin{pmatrix} 0 \\ \nabla \cdot \lt( \displaystyle \intr (u^\e \otimes u^\e - v \otimes v + \mathbb{I}_{d \times d})f^\e\,dv \rt)  \end{pmatrix}.
\]
This implies
$$\begin{aligned}
J_4^\e &= -\int_0^t \into u \cdot \lt(\nabla \cdot \lt(  \intr (u^\e \otimes u^\e - v \otimes v + \mathbb{I}_{d \times d})f^\e\,dv \rt)\rt) dx ds\cr
&= \int_0^t \into \nabla u :  \lt(  \intr (u^\e \otimes u^\e - v \otimes v + \mathbb{I}_{d \times d})f^\e\,dv \rt) dx ds\cr
&\leq \|\nabla u\|_{L^\infty}\int_0^t \into \lt| \intr (u^\e \otimes u^\e - v \otimes v + \mathbb{I}_{d \times d})f^\e\,dv\rt| dx ds.
\end{aligned}$$
On the other hand, we recall \eqref{est_conv0} that
$$\begin{aligned}
&\intr (u^\e \otimes u^\e - v \otimes v + \mathbb{I}_{d \times d})f^\e\,dv\cr
&\quad = \intr u^\e \sqrt{f^\e} \otimes \lt( (u^\e - v)\sqrt{f^\e} - 2\nabla_v \sqrt{f^\e} \rt) dv + \intr \lt( (u^\e - v)\sqrt{f^\e} - 2\nabla_v \sqrt{f^\e} \rt) \otimes v\sqrt{f^\e}\,dv.
\end{aligned}$$
By this and Proposition \ref{prop_energy}, we obtain
$$\begin{aligned}
J_4^\e &\leq \|\nabla u\|_{L^\infty} \int_0^t \lt(\intor f^\e |u^\e|^2 + f^\e |v|^2\,dxdv \rt)^{1/2}\lt(\intor \frac{1}{f^\e}|\nabla_v f^\e - (u^\e - v)f^\e|^2\,dxdv \rt)^{1/2}\,ds\cr
&\leq 2\|\nabla u\|_{L^\infty}\sqrt\e\int_0^t \lt(\intor  f^\e |v|^2\,dxdv \rt)^{1/2}\lt(\frac{1}{2\e} D_1(f)(s) \rt)^{1/2}\,ds\cr
&\leq C\sqrt\e \sup_{0 \leq t \leq T} \lt(\intor  f^\e |v|^2\,dxdv \rt)^{1/2}\lt(\int_0^t \frac{1}{2\e} D_1(f)(s)\,ds \rt)^{1/2}\cr
&\leq C\sqrt\e,
\end{aligned}$$
where we used H\"older inequality to find
\bq\label{est_ef}
|u^\e|^2 =  \lt|\frac{\displaystyle \intr vf^\e\,dv}{\displaystyle\intr f^\e\,dv} \rt|^2 \leq \frac{\displaystyle \intr |v|^2 f^\e\,dv}{\rho^\e}, \quad \mbox{i.e.,} \quad \rho^\e|u^\e|^2 \leq \intr |v|^2 f^\e\,dv.
\eq
Here $C > 0$ depends on $T$, $\|\nabla u\|_{L^\infty}$. It is worth emphasizing that our estimate gives that the constant $C$ depends on $\|\nabla u\|_{L^\infty}$, while \cite[Lemma 4.8]{KMT15} provides that it also depends on $\|\nabla \log \rho\|_{L^\infty}$. \newline

\noindent {\bf Estimate of $J_5^\e$}: We again use \eqref{est_ef} to get
$$\begin{aligned}
&\intoo \phi(x-y)|u^\e(x) - u^\e(y)|^2 \rho^\e(x)\rho^\e(y)\,dxdy \cr
&\quad = \intoo \phi(x-y) \lt(|u^\e(x)|^2 - 2 u^\e(x)\cdot u^\e(y) + |u^\e(y)|^2 \rt)\rho^\e(x) \rho^\e(y)\,dxdy\cr
&\quad \leq \intorr \phi(x-y)|w - v|^2 f^\e(x,v)f^\e(y,w)\,dxdvdydw.
\end{aligned}$$
Thus we have
\[
J_5^\e \leq \frac\alpha2\int_0^t \intorr \phi(x-y)|w - v|^2 f^\e(x,v)f^\e(y,w)\,dxdvdydwds.
\]
\noindent {\bf Estimate of $J_6^\e$}: A straightforward computation gives
$$\begin{aligned} 
J_6^\e &\leq 2\alpha \|u\|_{L^\infty}\|\phi\|_{L^\infty} \int_0^t \intoo\rho^\e(x) |\rho(y) - \rho^\e(y)| |u^\e(x) - u(x)|\,dxdyds \cr
&=2\alpha \|u\|_{L^\infty}\|\phi\|_{L^\infty} \int_0^t \|\rho - \rho^\e\|_{L^1} \into\rho^\e(x)  |u^\e(x) - u(x)|\,dxds \cr
&\leq 2\alpha \|u\|_{L^\infty}\|\phi\|_{L^\infty} \lt(\int_0^t \|(\rho - \rho^\e)(\cdot,s)\|_{L^1}^2\,ds \rt)^{1/2} \lt(\int_0^t \into \rho^\e(x,s) |(u^\e - u)(x,s)|^2\,dxds \rt)^{1/2}.
\end{aligned}$$
We then use \eqref{est_l1} to have
\[
J_6^\e \leq C\alpha\int_0^t \into \me(U^\e|U)\,dxds,
\]
where $C > 0$ depends on $\|u\|_{L^\infty}$, $\|\rho^\e\|_{L^1}$, and $\|\rho\|_{L^1}$. \newline

\noindent {\bf Estimate of $J_7^\e$}: Integrating by parts gives
$$\begin{aligned}
\lambda \int_0^t \into \nabla V(x) \cdot \rho^\e(x) u^\e(x)\,dxds &= -\lambda\int_0^t \into V(x) \nabla \cdot (\rho^\e(x,s) u^\e(x,s))\,dxds \cr
&= \lambda\int_0^t \into V(x) \pa_s \rho^\e(x,s)\,dxds \cr
&= \lambda\into V(x)\rho^\e(x,t)\,dx - \lambda\into V(x)\rho^\e_0(x)\,dx.
\end{aligned}$$
Thus we get
\[
J_7^\e = \gamma\int_0^t \into \rho^\e(x) |u^\e(x)|^2\,dxds + \lambda\into V(x)\rho^\e(x,t)\,dx - \lambda\into V(x)\rho^\e_0(x)\,dx.
\]
\noindent {\bf Estimate of $J_8^\e$}: Note that
$$\begin{aligned}
 &\lambda\int_0^t \into \rho^\e(x,s) u^\e(x,s) \cdot (\nabla W \star \rho^\e)(x,s)\,dxds \cr
 &\quad =  \lambda\int_0^t \into \pa_s (\rho^\e(x,s) )  (W \star \rho^\e)(x,s)\,dxds\cr
 &\quad =\frac\lambda2 \int_0^t \frac{\pa}{\pa s}\lt(\intoo  W(x-y) \rho^\e(x,s)   \rho^\e(y,s)\,dxdy\rt)ds\cr
 &\quad =\frac\lambda2\lt(\intoo  W(x-y) \rho^\e(x,t)\rho^\e(y,t)\,dxdy - \intoo W(x-y) \rho^\e_0(x)   \rho^\e_0(y)\,dxdy \rt).
\end{aligned}$$
This yields
$$\begin{aligned}
J_8^\e &= \lambda\int_0^t \into \rho^\e(x) ( u^\e(x) - u(x)) \cdot (\nabla W \star (\rho - \rho^\e))(x)\,dxds \cr
&\quad + \frac\lambda2\lt(\intoo  W(x-y) \rho^\e(x,t)\rho^\e(y,t)\,dxdy - \intoo W(x-y) \rho^\e_0(x)   \rho^\e_0(y)\,dxdy \rt).
\end{aligned}$$
We now combine the estimates $J_i^\e, i=2, 5, 7, 8$ to get
$$\begin{aligned}
\sum_{i \in \{2,5,7,8\}} J_i^\e &= \mathcal{O}(\sqrt\e) + \mf(f) - \mf(f_0) \cr
&\quad + \frac\alpha2\int_0^t \intoo \phi(x-y)|w - v|^2 f^\e(x,v)f^\e(y,w)\,dxdvdydwds\cr
&\quad + \lambda\int_0^t \into \rho^\e(x) ( u^\e(x) - u(x)) \cdot (\nabla W \star (\rho - \rho^\e))(x)\,dxds\cr
&\quad + \gamma\int_0^t \into \rho^\e(x) |u^\e(x)|^2\,dxds.
\end{aligned}$$
We then use Proposition \ref{prop_energy} to find
$$\begin{aligned}
\sum_{i \in \{2,5,7,8\}} J_i^\e &\leq \mathcal{O}(\sqrt\e) + C(1 + \gamma^2)\e\cr
&\quad + \lambda\int_0^t \into \rho^\e(x) ( u^\e(x) - u(x)) \cdot (\nabla W \star (\rho - \rho^\e))(x)\,dxds.
\end{aligned}$$
We finally combine all the above estimates to conclude the proof.
\end{proof}
%
%
%
%
\subsection{Singular/weakly regular interactions: $\Delta W = - \delta_0$ \& $\phi \in L^\infty(\om)$}
In this part, we consider the Coulombian interactions $W$ satisfying $\Delta W = - \delta_0$. Motivated from \cite{CFGS17,LT13,LT17}, we use a particular structure of the Poisson equation. 
 
\begin{lemma}\label{lem_wd} Suppose that the interaction potential $W$ satisfies $\Delta W = - \delta_0$. Then we have
\[
\frac\lambda 2\frac{d}{dt}\into |\nabla W \star (\rho - \rho^\e)|^2\,dx = \lambda \into \nabla W\star(\rho - \rho^\e)  \cdot \lt((\rho u) - (\rho^\e u^\e)\rt) dx
\]
for $t \in [0,T]$.
\end{lemma}
\begin{proof}Using the continuity equations of $\rho$ and $\rho^\e$, we find
\begin{align}\label{est_w1}
\begin{aligned}
\frac\lambda 2\frac{d}{dt}\into |\nabla W \star (\rho - \rho^\e)|^2\,dx &= \lambda\into (\nabla W \star (\rho - \rho^\e))\cdot (\nabla W \star (\pa_t \rho - \pa_t \rho^\e))\,dx\cr
&=-\lambda \into (\Delta W \star (\rho - \rho^\e)) \lt(W\star (\pa_t \rho - \pa_t \rho^\e)\rt)dx\cr
&= \lambda \into   (\rho - \rho^\e) \lt(W\star (\pa_t \rho - \pa_t \rho^\e)\rt)dx.
\end{aligned}
\end{align}
We then use the symmetry of $W$ to get
\begin{align}\label{est_w2}
\begin{aligned}
&\into  (\rho - \rho^\e) \lt(W\star (\pa_t \rho - \pa_t \rho^\e)\rt)dx\cr
&\quad =-\into (\rho - \rho^\e)(x) W(x-y) \lt( \nabla_y \cdot (\rho u)(y) - \nabla_y \cdot (\rho^\e u^\e)(y) \rt) dxdy\cr
&\quad = \into (\rho - \rho^\e)(x) \nabla_y \lt( W(x-y)\rt) \cdot \lt( (\rho u)(y) - (\rho^\e u^\e)(y) \rt) dxdy\cr
&\quad =-\into (\rho - \rho^\e)(x) \nabla_x W(x-y) \cdot \lt( (\rho u)(y) - (\rho^\e u^\e)(y) \rt) dxdy\cr
&\quad =\into (\rho - \rho^\e)(y) \nabla_x W(x-y) \cdot \lt( (\rho u)(x) - (\rho^\e u^\e)(x) \rt) dxdy\cr
&\quad = \into \nabla W\star(\rho - \rho^\e)  \cdot \lt((\rho u) - (\rho^\e u^\e)\rt) dx.
\end{aligned}
\end{align}
We finally combine \eqref{est_w1} and \eqref{est_w2} to conclude the proof.
\end{proof}

Then we are now ready to provide the details of the proof of Theorem \ref{thm_hydro1} for the singular interactions case. Since the strong convergences \eqref{thm_h1_conv} can be obtained from the inequalities in Theorem \ref{thm_hydro1}  and it is also already discussed in \cite{KMT15}, we skip the details of that proof.

\begin{proof}[Proof of Theorem \ref{thm_hydro1} (i)] Using the integration by parts, we estimate 
$$\begin{aligned}
&\into \rho^\e ( u^\e - u) \cdot (\nabla W \star (\rho - \rho^\e))\,dx + \into \nabla W\star(\rho - \rho^\e)  \cdot \lt((\rho u) - (\rho^\e u^\e)\rt) dx \cr
& \quad = \into \nabla W \star (\rho - \rho^\e) \cdot u (\rho - \rho^\e)\,dx \cr
& \quad = - \into \nabla W \star (\rho - \rho^\e) \cdot u \lt(\Delta W \star (\rho - \rho^\e)\rt)\,dx\cr
& \quad = -\frac\lambda2\into |\nabla W \star (\rho - \rho^\e)|^2 \nabla \cdot u\,dx + \lambda\into \nabla W \star (\rho - \rho^\e)\otimes \nabla W \star (\rho - \rho^\e) : \nabla u\,dx,
\end{aligned}$$
i.e.,
$$\begin{aligned}
&\lt| \into \rho^\e ( u^\e - u) \cdot (\nabla W \star (\rho - \rho^\e))\,dx + \into \nabla W\star(\rho - \rho^\e)  \cdot \lt((\rho u) - (\rho^\e u^\e)\rt) dx\rt|\cr
&\quad \leq \frac{3\lambda}{2}\|\nabla u\|_{L^\infty}\into |\nabla W \star (\rho - \rho^\e)|^2\,dx.
\end{aligned}$$
This together with Lemma \ref{lem_wd} and Proposition \ref{prop_re} yields
$$\begin{aligned}
&\into \me(U^\e|U)\,dx + \frac\lambda 2 \into |\nabla W \star (\rho - \rho^\e)|^2\,dx + \gamma\int_0^t \into \rho^\e(x)| u^\e(x) - u(x)|^2\,dxds\cr
&\quad + \frac\alpha2\int_0^t \intoo \rho^\e(x) \rho^\e(y)\phi(x-y)|( u^\e(x) - u(x)) - (u^\e(y) - u(y))|^2 dxdyds\cr
&\qquad \leq C\sqrt{\e}  +  \frac\lambda 2 \into |\nabla W \star (\rho_0 - \rho^\e_0)|^2\,dx\cr
&\qquad \quad + C(1 + \alpha)\int_0^t \into \me(U^\e|U)\,dxds  + C\lambda \int_0^t\into |\nabla W \star (\rho - \rho^\e)|^2\,dxds.
\end{aligned}$$
We finally apply Gr\"onwall's lemma to the above to conclude the desired result. This completes the proof.
\end{proof}

\begin{remark} The convergence 
\[
\into |\nabla W \star (\rho - \rho^\e)|^2\,dx \to 0 \quad \mbox{as} \quad \e \to 0
\]
implies
\[
\rho^\e \to \rho \quad \mbox{in} \quad L^\infty(0,T;H^{-1}(\om)).
\]
Indeed, we can easily find 
\[
\|\rho^\e - \rho\|_{H^{-1}} \leq \|\nabla W \star (\rho - \rho^\e)\|_{L^2}.
\]
\end{remark}


%
%
%
%
\subsection{Weakly regular interactions: $\nabla W \in L^\infty(\om)$  \& $\phi \in L^\infty(\om)$}
In this part, we deal with the weakly singular interactions case.

\begin{lemma}\label{lem_wr} Suppose that the interaction potential $W$ satisfies $\nabla W \in L^\infty(\om)$. Then we have
\[
\lt|\into \rho^\e(x) ( u^\e(x) - u(x)) \cdot (\nabla W \star (\rho - \rho^\e))(x)\,dx \rt| \leq 2\|\nabla W\|_{L^\infty}\into \me(U^\e|U)\,dx.
\]
\end{lemma}
\begin{proof} We use H\"older inequality to get
$$\begin{aligned}
&\into \rho^\e(x) ( u^\e(x) - u(x)) \cdot (\nabla W \star (\rho - \rho^\e))(x)\,dx\cr
&\quad \leq \|\nabla W \star (\rho - \rho^\e)\|_{L^\infty}\lt(\into \rho^\e |u^\e - u|^2\,dx\rt)^{1/2}\cr
&\quad \leq \|\nabla W\|_{L^\infty}\|\rho - \rho^\e\|_{L^1}\lt(\into \me(U^\e|U)\,dx\rt)^{1/2}.
\end{aligned}$$
On the other hand, $L^1$-norm of $\rho - \rho^\e$ can be estimated as 
$$\begin{aligned}
\into |\rho - \rho^\e|\,dx &= \into \min\lt\{\frac1\rho, \frac{1}{\rho^\e} \rt\}^{1/2} \max\{ \rho,\, \rho^\e\}^{1/2}|\rho - \rho^\e|\,dx\cr
&\leq \lt(\into \min\lt\{\frac1\rho, \frac{1}{\rho^\e} \rt\} (\rho - \rho^\e)^2\,dx \rt)^{1/2} \lt(\into \max\{ \rho,\, \rho^\e\}\,dx \rt)^{1/2}\cr
&\leq 2\lt(\into \me(U^\e|U)\,dx\rt)^{1/2},
\end{aligned}$$
due to \eqref{est_l1}. Thus we have
\[
\lt|\into \rho^\e(x) ( u^\e(x) - u(x)) \cdot (\nabla W \star (\rho - \rho^\e))(x)\,dx \rt| \leq 2\|\nabla W\|_{L^\infty}\into \me(U^\e|U)\,dx.
\]
\end{proof}

\begin{proof}[Proof of Theorem \ref{thm_hydro1} (ii)] By combining Lemma \ref{lem_wr} and Proposition \ref{prop_re}, we find
$$\begin{aligned}
&\into \me(U^\e|U)\,dx + \gamma\int_0^t \into \rho^\e(x)| u^\e(x) - u(x)|^2\,dxds\cr
&\quad + \frac\alpha2\int_0^t \intoo \rho^\e(x) \rho^\e(y)\phi(x-y)|( u^\e(x) - u(x)) - (u^\e(y) - u(y))|^2 dxdyds\cr
&\qquad \leq  C\sqrt{\e}   + C(1 + \gamma + \alpha)\int_0^t \into \me(U^\e|U)\,dxds. 
\end{aligned}$$
We complete the proof by using the Gronwall inequality to the above.
\end{proof}

%
%
%
%
\section{Hydrodynamic limit from kinetic to pressureless Euler equations}\label{sec_npE}

In this section, we consider the hydrodynamic limit from \eqref{main_eq} with $\sigma = 0$ to the pressureless Euler equations with nonlocal interaction forces \eqref{main_npE}. Similarly as before, we rewrite the limiting system \eqref{main_npE} as the following conservative form:
\[
\pa_t U + \nabla \cdot \hat A(U) = F(U),
\]
where 
\[
m = \rho u, \quad U := \begin{pmatrix}
\rho \\
m 
\end{pmatrix},
\quad
\hat A(U) := \begin{pmatrix}
m  \\
\frac{m \otimes m}{\rho} 
\end{pmatrix},
\]
and
\[
F(U) := \begin{pmatrix}
0 \\
\displaystyle \alpha\rho\into \phi(x-y)(u(y) - u(x))\rho(y)\,dy -\gamma \rho u - \lambda \rho\lt(\nabla V + \nabla W \star \rho \rt)
\end{pmatrix}.
\]
We then consider the kinetic energy of the above system:
\[
\hat E(U) := \frac{|m|^2}{2\rho}.
\]
Note that the entropy defined above is not strictly convex with respect to $\rho$. We also define the modulated kinetic energy as
$$\begin{aligned}
\hat \me(\bar U|U) :=& \hat E(\bar U) - \hat E(U) - D\hat E(U)(\bar U-U)\cr
=& \frac{\bar\rho|\bar u|^2}{2} - \frac{\rho|u|^2}{2} - \frac{|u|^2}{2}(\rho - \bar \rho) - u\cdot (\bar \rho \bar u - \rho u) \cr
=& \frac{\bar\rho}{2}|\bar u - u|^2 \quad \mbox{with} \quad \bar U := \begin{pmatrix}
        \bar\rho \\
        \bar m \\
    \end{pmatrix}.
\end{aligned}$$
Compared to the previous diffusive case, our functional $\hat\me$ does not include the relative pressure, and as a consequence we cannot deal with the $L^1$-norm of the $\bar\rho - \rho$. Thus we need to estimate the nonlocal interaction forces in a different way. For this, we will use a bounded Lipschitz distance for local densities, and this requires a higher regularity for the communication weight $\phi$. 

\subsection{Modulated kinetic energy inequality}
In the proposition below, we provide the modulated kinetic energy estimate. 
\begin{proposition}\label{prop_re15}Let $T>0$, $f^\e$ be a global weak solution to the equation \eqref{main_eq} with $\sigma = 0$, and let $(\rho,u)$ be a strong solution to the system \eqref{main_pE} on the time interval $[0,T]$. Then we have
\begin{align}\label{re15}
\begin{aligned}
&\into \hat\me(U^\e|U)\,dx + \gamma\int_0^t\into \hat\me(U^\e|U)\,dxds\cr
&\quad + \frac\alpha2\int_0^t \intoo \rho^\e(x) \rho^\e(y)\phi(x-y)|( u^\e(x) - u(x)) - (u^\e(y) - u(y))|^2 dxdyds\cr
&\qquad \leq \into \hat\me(U_0^\e|U_0)\,dx + \hat K(f^\e_0) - \into \hat E(U^\e_0)\,dx \cr
&\quad \qquad + (\|\nabla u\|_{L^\infty} + C\alpha^2)\int_0^t \into \hat\me(U^\e|U)\,dxds + \|\nabla u\|_{L^\infty} \int_0^t \intor |u^\e - v|^2 f^\e\,dxdvds \cr
&\quad \qquad +  C\int_0^t d_{BL}^2(\rho^\e,\rho) \,ds + \lambda\int_0^t \into \rho^\e(x) ( u^\e(x) - u(x)) \cdot (\nabla W \star (\rho - \rho^\e))(x)\,dxds
\end{aligned}
\end{align}
for $t \in [0,T]$.
\end{proposition}
\begin{proof} Employing almost the same arguments as in Lemma \ref{lem_rel} and Proposition \ref{prop_re}, we find
$$\begin{aligned}
\begin{aligned}
&\into \hat\me(U^\e|U)\,dx + \gamma\int_0^t \into \hat\me(U^\e|U)\,dxds\cr 
&\quad + \frac\alpha2\int_0^t \intoo \rho^\e(x) \rho^\e(y)\phi(x-y)|( u^\e(x) - u(x)) - (u^\e(y) - u(y))|^2 dxdyds\cr
&\qquad = \into \hat\me(U_0^\e|U_0)\,dx + \into \hat E(U^\e) - \hat E(U_0)\,dx - \int_0^t \into \nabla (D\hat E(U)):A(U^\e|U)\,dxds \cr
&\quad \qquad - \int_0^t \into D\hat E(U)\left[ \pa_s U^\e + \nabla \cdot \hat A(U^\e) - F(U^\e)\right]dxds \cr
&\quad \qquad +\frac\alpha2\int_0^t \intoo \rho^\e(x) \rho^\e(y)\phi(x-y)|u^\e(x) - u^\e(y)|^2\,dxdyds\cr
&\quad \qquad - \alpha\int_0^t \intoo \rho^\e(x) (\rho(y) - \rho^\e(y))\phi(x-y) (u^\e(x) - u(x))\cdot (u(y) - u(x))\,dxdyds\cr
&\quad \qquad + \gamma\int_0^t \into \rho^\e(x) |u^\e(x)|^2\,dxds + \lambda \int_0^t \into \nabla V(x) \cdot \rho^\e(x) u^\e(x)\,dxds\cr
&\quad \qquad + \lambda\int_0^t \into \rho^\e(x) ( u^\e(x) - u(x)) \cdot (\nabla W \star (\rho - \rho^\e))(x) +  \rho^\e(x) u^\e(x) \cdot (\nabla W \star \rho^\e)(x)\,dxds\cr
&\qquad \leq \into \hat\me(U_0^\e|U_0)\,dx + \hat K(f^\e_0) - \into \hat E(U^\e_0)\,dx + \|\nabla u\|_{L^\infty}\int_0^t \into \hat\me(U^\e|U)\,dxds\cr
&\quad \qquad - \int_0^t \into D\hat E(U)\left[ \pa_s U^\e + \nabla \cdot \hat A(U^\e) - F(U^\e)\right]dxds \cr
&\quad \qquad - \alpha\int_0^t \intoo \rho^\e(x) (\rho(y) - \rho^\e(y))\phi(x-y) (u^\e(x) - u(x))\cdot (u(y) - u(x))\,dxdyds\cr
&\quad \qquad + \lambda\int_0^t \into \rho^\e(x) ( u^\e(x) - u(x)) \cdot (\nabla W \star (\rho - \rho^\e))(x)\,dxds,
\end{aligned}
\end{aligned}$$
where $\hat K(f)$ denotes the kinetic energy for the kinetic equation, i.e.,
\[
\hat K(f) := \frac12\intor |v|^2 f\,dxdv.
\]
On the other hand, we notice that
$$\begin{aligned}
\pa_t U^\e + \nabla \cdot \hat A(U^\e) - F(U^\e) &= \begin{pmatrix} 0 \\ \nabla \cdot \lt( \displaystyle \intr (u^\e \otimes u^\e - v \otimes v)f^\e\,dv \rt)  \end{pmatrix}\cr
&= \begin{pmatrix} 0 \\ \nabla \cdot \lt( \displaystyle \intr (u^\e - v) \otimes (v-u^\e)f^\e\,dv \rt)  \end{pmatrix},
\end{aligned}$$
and this yields
\[
\lt|\int_0^t \into D\hat E(U)\left[ \pa_s U^\e + \nabla \cdot \hat A(U^\e) - F(U^\e)\right]dxds\rt| \leq \|\nabla u\|_{L^\infty} \int_0^t \intor |u^\e - v|^2 f^\e\,dxdvds.
\]
For the term with the communication weight function $\phi$, we denoted it by $I^\e$ and split into two terms:
$$\begin{aligned}
I^\e &= - \alpha\int_0^t\intoo \rho^\e(x) (\rho(y) - \rho^\e(y))\phi(x-y) (u^\e(x) - u(x))\cdot u(y)\,dxdyds \cr
&\quad + \alpha\int_0^t\intoo \rho^\e(x) (\rho(y) - \rho^\e(y))\phi(x-y) (u^\e(x) - u(x))\cdot u(x)\,dxdyds\cr
&=: I_1^{\e} + I_2^{\e},
\end{aligned}$$
where $I_1^{\e}$ can be estimated as
$$\begin{aligned}
\lt|I_1^{\e}\rt| &= \alpha\lt|\int_0^t \into \lt(\into (\rho(y) - \rho^{\e}(y))\phi(x-y)u(y)\,dy \rt) \cdot \rho^\e(x)(u^\e(x) - u(x))\,dxds\rt|\cr
&\leq C \alpha\int_0^t d_{BL}(\rho^{\e},\rho)\into \rho^\e(x)|u^\e(x) - u(x)|\,dx dt\cr
&\leq C\int_0^t d_{BL}^2(\rho^\e,\rho) \,ds + C \alpha^2 \int_0^t \into \hat\me(U^\e|U)\,dxds.
\end{aligned}$$
Here we used the fact that $y \mapsto \phi(\cdot,y)u(y)$ is bounded and Lipschitz continuous. Similarly, we can also show that
\[
\lt|I_2^{\e}\rt| \leq C\int_0^t d_{BL}^2(\rho^\e,\rho) \,ds + C\alpha^2 \int_0^t\into \hat\me(U^\e|U)\,dxds,
\]
and this yields
\[
\lt|I^\e\rt| \leq C\int_0^t d_{BL}^2(\rho^\e,\rho) \,ds + C \alpha^2 \int_0^t\into \hat\me(U^\e|U)\,dxds,
\]
where $C>0$ is independent of $\e>0$. This completes the proof.
\end{proof}

\subsection{Singular/strongly regular interactions: $\Delta W = - \delta_0$  \& $\phi \in \mw^{1,\infty}(\om)$}
In this subsection, we consider the Coulombian interaction potential $W$, i.e., $W$ satisfies $\Delta W = - \delta_0$. 

We first notice from Lemma \ref{lem_wd}, see also proof of Theorem \ref{thm_hydro1} (i), that the last term on the right hand side of \eqref{re15} can be bounded from above by
\[
-\frac\lambda 2 \into |\nabla W \star (\rho - \rho^\e)|^2\,dx + \frac\lambda 2 \into |\nabla W \star (\rho_0 - \rho^\e_0)|^2\,dx + \frac{3\lambda}{2}\|\nabla u\|_{L^\infty}\int_0^t \into |\nabla W \star (\rho - \rho^\e)|^2\,dxds.
\]
We next use the free energy estimate \eqref{energy_zerosig} to show 
\[
\int_0^t \intor |u^\e - v|^2 f^\e\,dxdvds \leq C\e,
\]
where $C > 0$ is independent of $\e$. Combining those observations with Proposition \ref{prop_re15} yields
the following proposition.
\begin{proposition}\label{prop_re20}Let $T>0$, $f^\e$ be a global weak solution to the equation \eqref{main_eq} with $\sigma = 0$, and let $(\rho,u)$ be a strong solution to the system \eqref{main_pE} on the time interval $[0,T]$. Then we have
$$\begin{aligned}
&\into \hat\me(U^\e|U)\,dx + \frac\lambda 2 \into |\nabla W \star (\rho - \rho^\e)|^2\,dx + \gamma\int_0^t\into \hat\me(U^\e|U)\,dxds\cr
&\quad + \frac\alpha2\int_0^t \intoo \rho^\e(x) \rho^\e(y)\phi(x-y)|( u^\e(x) - u(x)) - (u^\e(y) - u(y))|^2 dxdyds\cr
&\qquad \leq \into \hat\me(U_0^\e|U_0)\,dx + \hat K(f^\e_0) - \into \hat E(U^\e_0)\,dx + \frac\lambda 2 \into |\nabla W \star (\rho_0 - \rho^\e_0)|^2\,dx + C\e \cr
&\quad \qquad + C \int_0^t \into \hat\me(U^\e|U)\,dxds  + C\int_0^t \into |\nabla W \star (\rho - \rho^\e)|^2\,dxds +  C\int_0^t d_{BL}^2(\rho^\e,\rho) \,ds
\end{aligned}$$
for $t \in [0,T]$, where $C>0$ is independent of $\e>0$.
\end{proposition}

In order to close the modulated kinetic energy inequality, we show that the bounded and Lipschitz distance $d_{BL}$ between local densities can be bounded from above by the modulated kinetic energy, which directly gives the quantitative error estimate between $\rho$ and $\rho^\e$. 

\begin{lemma}\label{prop_rho_wa} Let $f^\e$ be a global weak solution to the equation \eqref{main_eq} with $\sigma = 0$ and $(\rho,u)$ be a strong solution to the system \eqref{main_npE} on the time interval $[0,T]$. Then we have
\[
d_{BL}(\rho(t), \rho^\e(t)) \leq Cd_{BL}(\rho_0, \rho_0^\e) + C\lt(\int_0^t  \into \hat\me(U^\e|U)\,dx ds\rt)^{1/2}
\]
for $0 \leq t \leq T$, where $C > 0$ is independent of $\e>0$.
\end{lemma}

Although it has been already studied in \cite{Cpre}, see also \cite{CCpre, FK19}, we give the details of proof for the completeness of our work. Let us define forward characteristics $X(t) := X(t;0,x)$ which solves 
\bq\label{eq_char}
\pa_t X(t) = u(X(t),t) \quad \mbox{with} \quad X(0) = x \in \om.
\eq
Then $X(t)$ uniquely exists on the time interval $[0,T]$ since $u$ is bounded and Lipschitz continuous. Moreover, the solution $\rho$ can be determined as the push-forward of the its initial densities through the flow maps $X$, i.e.,  $\rho(t) = X(t;0,\cdot) \# \rho_0$. On the other hand, we cannot consider the characteristic for the continuity equation of $\rho^\e$ due to the lack of regularity of $u^\e$. Regarding this problem, we recall the following proposition from \cite[Theorem 8.2.1]{AGS08}, see also \cite[Proposition 3.3]{FK19}. 
\begin{proposition}\label{prop_am}Let $T>0$ and $\rho : [0,T] \to \mathcal{P}_p(\om)$ be a narrowly continuous solution of \eqref{eq_char}, that is, $\rho$ is continuous in the duality with continuous bounded functions, for a Borel vector field $u$ satisfying
\bq\label{est_p1}
\int_0^T\into |u(x,t)|^p\rho(x,t)\,dx dt < \infty
\eq
for some $p > 1$. Let $\Gamma_T: [0,T] \to \om$ denote the space of continuous curves. Then there exists a probability measure $\eta$ on $\Gamma_T \times \om$ satisfying the following properties:
\begin{itemize}
\item[(i)] $\eta$ is concentrated on the set of pairs $(\gamma,x)$ such that $\gamma$ is an absolutely continuous curve satisfying
\[
\dot\gamma(t) = u(\gamma(t),t)
\]
for almost everywhere $t \in (0,T)$ with $\gamma(0) = x \in \om$.
\item[(ii)] $\rho$ satisfies
\[
\int \varphi(x)\rho\,dx = \int_{\Gamma_T \times \om}\varphi(\gamma(t))\,d\eta(\gamma,x)
\]
for all $\varphi \in \mc_b(\om)$, $t \in [0,T]$.
\end{itemize}
\end{proposition}

\begin{proof}[Proof of Lemma \ref{prop_rho_wa}] It follows from Lemma \ref{lem_energy} that
$$\begin{aligned}
\into |u^\e|^2 \rho^\e\,dx &\leq \intor |v|^2 f^\e\,dxdv \cr
&\leq \intor |v|^2 f^\e_0\,dxdv + \lambda\intoo W(x-y)\rho^\e_0(x) \rho^\e(y)\,dxdy + 2\lambda \into V\rho^\e_0(x)\,dx < \infty,
\end{aligned}$$
thus, the integrability condition \eqref{est_p1} holds for $p=2$, and thus by Proposition \ref{prop_am}, we obtain a probability measure $\eta^\e$ in $\Gamma_T \times \R$ concentrated on the set of pairs $(\gamma,x)$ such that $\gamma$ is a solution of
\bq\label{eq_gam}
\dot{\gamma}(t) = u^\e(\gamma(t),t)
\eq
with $\gamma(0) = x \in \om$. Moreover, $\rho^\e$ satisfies Proposition \ref{prop_am} (ii), i.e.,
\bq\label{eq_gam2}
\into \varphi(x) \rho^\e(x,t)\,dx = \int_{\Gamma_T \times \om}\varphi(\gamma(t))\,d\eta^\e(\gamma,x)
\eq
for all $\varphi \in \mc_b(\om)$, $t \in [0,T]$. We now consider the push-forward of $\rho_0^\e$ through the flow map $X$ and denote it by $\bar\rho^\e$, i.e., $\bar\rho^\e = X \# \rho_0^\e$. 

We first estimate the error between $\bar\rho^\e$ and $\rho^\e$ in bounded Lipschitz distance. By the disintegration theorem of measures, see \cite{AGS08}, we can write
\[
d\eta^\e(\gamma,x) = \eta^\e_x(d\gamma) \otimes \rho^\e_0(x)\,dx,
\]
where $\{\eta^\e_x\}_{x \in \om}$ is a family of probability measures on $\Gamma_T$ concentrated on solutions of \eqref{eq_gam}. By using this newly introduced measure $\eta^\e$, we define a measure $\nu^\e$ on $\Gamma_T \times \Gamma_T \times \Omega$ by
\[
d\nu^\e(\gamma, \sigma, x) := \eta^\e_x(d\gamma) \otimes \delta_{X(\cdot;0,x)}(d\sigma) \otimes \rho^\e_0(x)\,dx.
\]
We further take into account an evaluation map $E_t : \Gamma_T  \times \Gamma_T \times \om \to \om \times \om$ defined as $E_t(\gamma, \sigma, x) = (\gamma(t), \sigma(t))$. Then we find that measure $\pi^\e_t:= (E_t)\# \nu^\e$ on $\om \times \om$ has marginals $\rho^\e(x,t)\,dx$ and $\bar\rho^\e(y,t)\,dy$ for $t \in [0,T]$, see \eqref{eq_gam2}. This implies
\begin{align}\label{est_rho2}
\begin{aligned}
d_{BL}(\rho^\e(t), \bar\rho^\e(t)) &\leq \int_{\om \times \om} |x-y|\,d\pi^\e_t(x,y)\cr
&=\int_{\Gamma_T \times \Gamma_T \times \om} |\sigma(t) - \gamma(t) | \,d\nu^\e(\gamma, \sigma, x) \cr
&= \int_{\Gamma_T \times \om} |X(t;0,x) - \gamma(t)| \,d\eta^\e(\gamma,x).
\end{aligned}
\end{align}
We notice from \eqref{eq_char} and \eqref{eq_gam} that
$$\begin{aligned}
\lt|X(t;0,x) -\gamma(t)\rt| &= \lt|\int_0^t u(X(s;0,x)) - u^\e(\gamma(s),s)\,ds\rt|\cr
&\leq \int_0^t \lt|u(X(s;0,x)) - u(\gamma(s),s)\rt|ds + \int_0^t \lt|u(\gamma(s),s) - u^\e(\gamma(s),s)\rt|ds\cr
&\leq \|\nabla_x u\|_{L^\infty}\int_0^t  \lt|X(s;0,x) - \gamma(s)\rt|ds + \int_0^t \lt|u(\gamma(s),s) - u^\e(\gamma(s),s)\rt|ds.
\end{aligned}$$
We then apply Gr\"onwall's lemma to the above to yield
\[
\lt|X(t;0,x) -\gamma(t)\rt| \leq C\int_0^t \lt|u(\gamma(s),s) - u^\e(\gamma(s),s)\rt|ds,
\]
where $C>0$ is independent of $\e>0$. Putting this into \eqref{est_rho2} entails
\begin{align}\label{est_rho22}
\begin{aligned}
d_{BL}(\rho^\e(t), \bar\rho^\e(t)) &\leq C\int_0^t \int_{\Gamma_T \times \om} \lt|u(\gamma(s),s) - u^\e(\gamma(s),s)\rt|d\eta^\e(\gamma,x)\,ds\cr
&\leq C\int_0^t \into |u(x,s) - u^\e(x,s)| \rho^\e(x,s)\,dxds\cr
&\leq C\sqrt T\lt( \int_0^t \into |u^\e(x,s) - u(x,s)|^2 \rho^\e(x,s) \,dx ds\rt)^{1/2}\cr
&=C\lt(\int_0^t  \into \hat\me(U^\e|U)\,dx ds\rt)^{1/2},
\end{aligned}
\end{align}
where $C>0$ is independent of $\e > 0$, and we used \eqref{eq_gam2}. 

We next estimate the bounded Lipschitz distance between $\bar\rho^\e$ and $\rho$. For bounded Lipschitz function $\phi$, we find
\bq\label{est_rho1}
\lt|\into \phi(x) (\rho(x) - \bar\rho^\e(x))\,dx\rt| = \lt|\into \phi(X(t))(\rho_0(x) - \rho_0^\e(x))\,dx  \rt| \leq Cd_{BL}(\rho_0, \rho_0^\e),
\eq
where $C>0$ is independent of $\e$, and we used the bounded Lipschitz continuity of $\phi(X(t;0,\cdot))$. More precisely, we have
$$\begin{aligned}
|X(t;0,x) - X(t;0,y)| &\leq |x-y| + \int_0^t |u(X(s;0,x)) - |u(X(s;0,y))|\,ds\cr
&\leq |x-y| + \|\nabla_x u\|_{L^\infty}\int_0^t |X(s;0,x) - X(s;0,y)|\,ds,
\end{aligned}$$
and applying Gr\"onwall's lemma to the above yields the Lipschitz continuity of the characteristic flow $X(t;0,x)$ in $x$. Furthermore, we have
\[
|\phi(X(t;0,x)) - \phi(X(t;0,y))| \leq \|\phi\|_{Lip}|X(t;0,x) - X(t;0,y)| \leq \|\phi\|_{Lip} \|X\|_{Lip}|x-y|,
\]
where $\|\cdot\|_{Lip}$ denotes the Lipschitz constant given by
\[
\|\phi\|_{Lip} := \sup_{x \neq y \in \R^d} \frac{|\phi(x) - \phi(y)|}{|x-y|}.
\]
This together with \eqref{est_rho1} implies
\[
d_{BL}(\rho(t), \bar\rho^\e(t)) \leq Cd_{BL}(\rho_0, \rho_0^\e)
\]
for $t \in [0,T]$, where $C > 0$ is independent of $\e > 0$.  Finally, we combine this with \eqref{est_rho22} to conclude
$$\begin{aligned}
d_{BL}(\rho(t), \rho^\e(t)) &\leq d_{BL}(\rho(t), \bar\rho^\e(t)) + d_{BL}(\rho^\e(t), \bar\rho^\e(t))\cr
&\leq Cd_{BL}(\rho_0, \rho_0^\e) + C\lt(\int_0^t  \into \hat\me(U^\e|U)\,dx ds\rt)^{1/2},
\end{aligned}$$
where $C>0$ is independent of $\e > 0$. 
\end{proof} 

\begin{proof}[Proof of Theorem \ref{thm_hydro2} (i)] Applying Lemma \ref{prop_rho_wa} to Proposition \ref{prop_re20} yields
$$\begin{aligned}
&\into \hat\me(U^\e|U)\,dx + \frac\lambda 2 \into |\nabla W \star (\rho - \rho^\e)|^2\,dx +d_{BL}^2(\rho^\e, \rho) + \gamma\int_0^t\into \hat\me(U^\e|U)\,dxds\cr
&\quad + \frac\alpha2\int_0^t \intoo \rho^\e(x) \rho^\e(y)\phi(x-y)|( u^\e(x) - u(x)) - (u^\e(y) - u(y))|^2 dxdyds\cr
&\qquad \leq \into \hat\me(U_0^\e|U_0)\,dx + \hat K(f^\e_0) - \into \hat E(U^\e_0)\,dx + \frac\lambda 2 \into |\nabla W \star (\rho_0 - \rho^\e_0)|^2\,dx + Cd_{BL}^2(\rho_0^\e, \rho_0) + C\sqrt\e \cr
&\quad \qquad + C \int_0^t \into \hat\me(U^\e|U)\,dxds  + C\int_0^t \into |\nabla W \star (\rho - \rho^\e)|^2\,dxds +  C\int_0^t d_{BL}^2(\rho^\e,\rho) \,ds.
\end{aligned}$$
We then use Gr\"onwall's lemma to the above and the assumptions {\bf (H2)}--{\bf (H3)} to conclude the desired result.
\end{proof}

\subsection{Strongly regular interactions: $\nabla W \in \mw^{1,\infty}(\om)$ \& $\phi \in \mw^{1,\infty}(\om)$}

As a direct consequence of Proposition \ref{prop_re15} together with the assumptions {\bf (H2)}--{\bf (H3)}, we have the following proposition.

\begin{proposition}\label{prop_re2}Let $f^\e$ be a global weak solution to the equation \eqref{main_eq} with $\sigma = 0$ and $(\rho,u)$ be a strong solution to the system \eqref{main_npE} on the time interval $[0,T]$. Then we have
$$\begin{aligned}
&\into \hat\me(U^\e|U)\,dx + \gamma\int_0^t \into \hat\me(U^\e|U)\,dxds\cr
&\quad + \frac\alpha2\int_0^t \intoo \rho^\e(x) \rho^\e(y)\phi(x-y)|( u^\e(x) - u(x)) - (u^\e(y) - u(y))|^2 dxdyds\cr
&\qquad \leq  C\sqrt{\e} + C\int_0^t d_{BL}^2(\rho^\e,\rho) \,ds + C\int_0^t \into \hat\me(U^\e|U)\,dxds  \cr
&\qquad \quad + \lambda\int_0^t \into \rho^\e(x) ( u^\e(x) - u(x)) \cdot (\nabla W \star (\rho - \rho^\e))(x)\,dxds,
\end{aligned}$$
where $C>0$ is independent of $\e$. 
\end{proposition}

\begin{proof}[Proof of Theorem \ref{thm_hydro2} (ii)] We first claim that  
\[
\lambda \lt|\into \rho^\e(x) ( u^\e(x) - u(x)) \cdot (\nabla W \star (\rho - \rho^\e))(x)\,dx \rt| \leq Cd_{BL}^2(\rho^\e,\rho) + C\lambda^2 \into \rho^\e|u^\e - u|^2\,dx.
\]
Indeed, since $\nabla W \in \mathcal{W}^{1,\infty}(\om)$, we get
\[
\lt|\into \nabla W(x-y)(\rho(y) - \rho^\e(y))\,dy \rt| \leq Cd_{BL}(\rho^\e,\rho).
\]
This yields
$$\begin{aligned}
&\lambda\lt|\into \rho^\e(x) ( u^\e(x) - u(x)) \cdot (\nabla W \star (\rho - \rho^\e))(x)\,dx \rt| \cr
&\quad \leq C\lambda d_{BL}(\rho^\e,\rho)\lt(\into \rho^\e|u^\e - u|^2\,dx \rt)^{1/2} \leq Cd_{BL}^2 (\rho^\e,\rho) + C\lambda^2 \into \rho^\e|u^\e - u|^2\,dx.
\end{aligned}$$
This together with Proposition \ref{prop_re2} provides 
$$\begin{aligned}
&\into \hat\me(U^\e|U)\,dx +d_{BL}^2(\rho^\e, \rho) + \gamma\int_0^t\into \hat\me(U^\e|U)\,dxds\cr
&\quad + \frac\alpha2\int_0^t \intoo \rho^\e(x) \rho^\e(y)\phi(x-y)|( u^\e(x) - u(x)) - (u^\e(y) - u(y))|^2 \,dxdyds\cr
&\qquad \leq C\sqrt{\e} + C\int_0^t d_{BL}^2(\rho^\e,\rho) \,ds + C\int_0^t \into \hat\me(U^\e|U)\,dxds. 
\end{aligned}$$
Hence, by applying Gronwall's lemma to the above, we complete the proof.
\end{proof}

%
%
%
%
\section{Global existence of weak solutions to the kinetic equation \eqref{main_eq}}\label{sec_global_kin}
In this section, we provide the global-in-time existence of weak solutions to the system \eqref{main_eq}. For notational simplicity, we set $\gamma = \lambda = \alpha = \beta = 1$. We also only consider the Coulombian case since the weakly singular case $\nabla_x W \in L^\infty(\om)$ can be easily obtained similarly as in \cite{KMT13}. Here, the domain of our interest is $\Omega = \T^d$ or $\R^d$ with $d\ge 3$. Since the analysis on $\T^d$ is almost similar to the $\R^d$ case, we mostly consider the case $\Omega = \R^d$. 

\begin{theorem}\label{T2.1} Let $T>0$. Suppose that $f_0$ satisfies 
\[
f_0 \in (L^1_+ \cap L^\infty)(\om \times \R^d) \quad \mbox{and} \quad (|v|^2 + V + W\star \rho_0)f_0 \in L^1(\om \times \R^d).
\]
Then there exists a weak solution of the equation \eqref{main_eq} in the sense of Definition \ref{def_weak} satisfying
\[
f \in \mc([0,T];L^1(\om \times \R^d)) \cap L^\infty(\om \times \R^d \times (0,T)) \quad \mbox{and} \quad (|v|^2 + V + W\star \rho)f \in L^\infty(0,T;L^1(\om \times \R^d)).
\]
Furthermore, $f$ satisfies the entropy inequality \eqref{entro_1}.
\end{theorem}

\subsection{Regularized kinetic equation}

For the existence of weak solutions to \eqref{main_eq}, we first regularize the system with respect to regularization parameters $\eta := (R, \zeta, \e)$ as follows: 
\begin{align}\label{E-1}
\begin{aligned}
&\pa_t f^\eta + v\cdot\nabla f^\eta -  \nabla_v \cdot \lt((v + \nabla V^R + \nabla W^\e \star \rho^\eta )f^\eta \rt)+ \nabla_v \cdot \lt(F[f^\eta]f^\eta\rt) \cr
&\hspace{2cm}=  \nabla_v \cdot ( (v - \chi_\zeta (u_\e^\eta))f^\eta + \sigma \nabla_v f^\eta),
\end{aligned}
\end{align}
subject to initial data:
\[
f_0^\eta = f^\eta(0,x,v) := f_0(x,v) \mathds{1}_{\{|v|\le \zeta\}},
\]
where 
\[
\rho^\eta := \int_{\R^d} f^\eta dv, \quad \rho^\eta u^\eta := \int_{\R^d} vf^\eta dv, \quad u_\e^\eta := \frac{\rho^\eta u^\eta}{\rho + \e},
\]
and $W^\e = W^\e(x)$ is given as
\[
W^\e(x) := c_d(\e + |x|^2)^{-(d-2)/2}
\]
with $d \geq 3$, where $c_d$ is a normalization constant. Moreover, $\chi_\zeta$ is given as
\[
\chi_\zeta(v) = v \mathds{1}_{\{ |v|\le \zeta \}},
\]
and $V^R$ is given as
\[
V^R(x) := V(x)M\left(\frac{x}{R}\right).
\]
Here $M(x) \in \mc^\infty_c(\R^d)$ is a smooth function given by
\[
M(x) = \left\{\begin{array}{cl}
1 & |x|\le 1, \\[1mm]
0 < M(x) <1 & 1 < |x| < 2,\\[1mm]
0 & |x|\ge 2.
\end{array}\right. 
\]
Now, we partially linearize \eqref{E-1} as follows:
\begin{align}\label{E-2}
\begin{aligned}
&\pa_t f^\eta + v\cdot\nabla f^\eta -  \nabla_v \cdot \lt((v + \nabla V^R + \nabla W^\e \star \rho^\eta )f^\eta \rt)+ \nabla_v \cdot \lt(F[f^\eta]f^\eta\rt) \cr
&\hspace{2cm}=  \nabla_v \cdot ( (v - \chi_\zeta (\tilde{u}))f^\eta + \sigma \nabla_v f^\eta),
\end{aligned}
\end{align}
where $ \tilde{u}$ is in $\mathcal{S} := L^2(\Omega \times (0,T))$. Once we note that $\nabla W^\e$ is bounded and Lipschitz continuous, the existence of weak solutions to \eqref{E-2} comes from almost the same argument in \cite[Theorem 6.3]{KMT13}. Moreover, we estimate
$$\begin{aligned}
\frac{d}{dt}\intor (f^\eta)^p\,dxdv &= (p-1) \intor (f^\eta)^p \nabla_v \cdot \lt( 2v + \nabla V^R + \nabla W^\e \star \rho^\eta - \chi_\zeta (\tilde{u}) - F[f^\eta]\rt) dxdv\cr
&\quad - \sigma p(p-1) \intor (f^\eta)^{p-2} |\nabla_v f^\eta|^2\,dxdv\cr
&=(p-1)\intor (f^\eta)^p (2d + d\phi \star\rho^\eta)\,dxdv - \frac{4\sigma (p-1)}{p}\intor |\nabla_v (f^\eta)^{p/2}|^2\,dxdv
\end{aligned}$$
for $p \in [1,\infty)$. This together with Gr\"onwall's lemma gives
\[
\begin{split}
\|f^\eta(\cdot,\cdot,t)\|_{L^p}^p &+ \frac{4\sigma(p-1)}{p}\int_0^t e^{d(p-1)(2 + \|\phi\|_{L^\infty}\|f^\eta_0\|_{L^1})(t-s)}\|\nabla_v (f^\eta)^{p/2}(\cdot,\cdot,s)\|_{L^2}^2\,ds \\
&\leq \|f^\eta_0\|_{L^p}^pe^{d(p-1)(2+\|\phi\|_{L^\infty}\|f_0^\eta\|_{L^1})t}.
\end{split}
\]
In particular, we have
\[
\|f^\eta(\cdot,\cdot,t)\|_{L^1} \leq \|f^\eta_0\|_{L^1} \le \|f_0\|_{L^1} = 1, \quad \|f^\eta(\cdot,\cdot,t)\|_{L^\infty} \le \|f_0^\eta\|_{L^\infty}e^{d(2+\|\phi\|_{L^\infty})t}
\]
for $t \in [0,T]$. \\

We next estimate higher-order velocity moments and entropy inequality of solutions to \eqref{E-2}.
\begin{lemma}\label{L5.1}
For a weak solution $f^\eta$ to \eqref{E-2}, its velocity moments satisfy the following boundedness condition:
\[
\sup_{t \in (0,T)}\int_{\Omega \times \R^d} |v|^N f dx dv \le C(d,\eta, N, T), \quad \forall \,N \ge 0. 
\]
\end{lemma}
\begin{proof}
For a weak solution $f^\eta$ to \eqref{E-2} and $N\ge2$, we let 
\[
m_N(f) := \int_{\Omega \times \R^d} |v|^N f\,dxdv.
\] 
Then, we estimate
\begin{align*}
\frac{d}{dt} m_N(f^\eta)&= -N \int_{\Omega \times \R^d} (v + \nabla V^R + \nabla W^\e\star\rho^\eta)\cdot v f^\eta |v|^{N-2}\,dxdv\\
&\quad + N\int_{\Omega \times \R^d} F[f^\eta]f^\eta \cdot v |v|^{N-2}dxdv - N\int_{\Omega\times\R^d} (v-\chi_\zeta(\tilde{u}))\cdot v f^\eta |v|^{N-2}\,dxdv\\
&\quad - \sigma N \int_{\Omega \times \R^d} \nabla_v f^\eta \cdot v |v|^{N-2}\,dxdv\\
&= -2N m_N(f^\eta) -N \int_{\Omega \times \R^d} (\nabla V^R + \nabla W^\e\star\rho^\eta)\cdot v f^\eta |v|^{N-2}\,dxdv\\
&\quad + N\intorr \phi(x-y)(w-v)\cdot v f^\eta(y,w)f^\eta(x,v)|v|^{N-2} \,dxdydvdw\\
&\quad + N\int_{\Omega \times \R^d}\chi_\zeta (\tilde{u}) \cdot v f^\eta |v|^{N-2} \,dxdv + \sigma N(N-2+d)m_{N-2}(f^\eta)\\
&\le -N m_N(f^\eta)+ \frac{N}{2} \int_{\Omega \times \R^d}(\nabla V^R + \nabla W^\e\star\rho^\eta)^2 f^\eta |v|^{N-2}\,dxdv\\
&\quad + N\phi_M \intorr |w||v| f^\eta(y,w)f^\eta(x,v)|v|^{N-2}\,dxdydvdw\\
&\quad + \frac{N}{2}\int_{\Omega \times \R^d}(\chi_\zeta (\tilde{u}))^2  f^\eta |v|^{N-2} \,dxdv + \sigma N(N-2+d)m_{N-2}(f^\eta)\\
&\le \frac{N}{2} \int_{\Omega \times \R^d}(\nabla_x V^R + \nabla W^\e\star\rho^\eta)^2 f^\eta |v|^{N-2}\,dxdv + \frac{N(\phi_M)^2}{2} m_2(f^\eta) m_{N-2}(f^\eta) \\
&\quad + \frac{N}{2}\int_{\Omega \times \R^d}(\chi_\zeta (\tilde{u}))^2  f^\eta |v|^{N-2} \,dxdv + \sigma N(N-2+d)m_{N-2}(f^\eta)\\
&\le C(m_N(f^\eta) + m_{N-2}(f^\eta)),
\end{align*}
where $C = C(d,\eta, N, T)$ is a positive constant and we used Young's inequality. Since $m_0(f^\eta)$ is just $\|f^\eta\|_{L^1} = \|f_0^\eta\|_{L^1}$, one uses Gr\"onwall's lemma and induction argument to conclude that
\[
\sup_{t \in (0,T)}\int_{\Omega \times \R^d} |v|^N f \,dx dv \le C(d,\eta, N, T), \quad \forall \, N=0,2,4, \cdots. 
\]
Moreover, for $N \in \R_+ \setminus\{0,2,4, \cdots\}$, we can find $l \in \bbn \cup \{ 0\}$ that satisfies $0 < N - 2l <2$, and this gives
\[
\int_{\Omega \times \R^d} |v|^{N-2l} f \,dxdv \le \left(\int_{\Omega \times \R^d} |v|^2 f \,dxdv \right)^{\frac{N-2l}{2}} \left( \int_{\Omega \times \R^d} f \, dxdv \right)^{\frac{2+2l-N}{2}} \le C(d,\eta, N, T).
\]
This asserts our desired result.
\end{proof}

\begin{proposition}\label{P5.1}
For a weak solution $f^\eta$ to \eqref{E-2}, it satisfies the following relation:
$$\begin{aligned}
\frac{d}{dt}&\left( \int_{\Omega \times \R^d} \left(\frac{|v|^2}{2} + V^R + \sigma \log f^\eta \right) f^\eta \,dxdv + \frac{1}{2}\intoo W^\e(x-y)\rho^\eta (x) \rho^\eta (y) \,dxdy\right)\\
&\qquad + \int_{\Omega\times\R^d}\frac{1}{f^\eta} \left|\sigma \nabla_v f^\eta - (v-\chi_\zeta(\tilde{u}))f^\eta\right|^2  \,dxdv + \int_{\Omega\times\R^d} |v|^2 f^\eta \,dxdv\\
&\qquad +\frac{1}{2}\intorr \phi(x-y)|w-v|^2 f^\eta (y,w)f^\eta(x,v)\,dxdydvdw\\
&= \int_{\Omega\times\R^d} (\chi_\zeta(\tilde{u})-v)\cdot \chi_\zeta(\tilde{u}) f^\eta \,dxdv  \\
&\quad + \sigma d \|f_0^\eta\|_{L^1} +\sigma d\intorr\phi(x-y)f^\eta(y,w)f^\eta(x,v)\,dxdydvdw.
\end{aligned}$$
\end{proposition}
\begin{proof}
First, it directly follows from Lemma \ref{L5.1} that
$$\begin{aligned}
\frac12\frac{d}{dt}\left(\int_{\Omega \times \R^d} |v|^2 f^\eta \,dxdv\right) &= -\int_{\Omega\times\R^d}(v+ \nabla V^R + \nabla W^\e \star \rho^\eta)\cdot v f^\eta \,dxdv\\
&\quad + \intorr \phi(x-y)(w-v)\cdot v f^\eta(y,w)f^\eta(x,v)\,dxdydvdw\\
&\quad - \int_{\Omega\times\R^d} (v-\chi_\zeta(\tilde{u}))\cdot v f^\eta \,dxdv  - \sigma\int_{\Omega \times \R^d} v \cdot \nabla_v f^\eta \,dxdv\\
&=  -\int_{\Omega\times\R^d}(v+ \nabla V^R + \nabla W^\e \star \rho^\eta)\cdot v f^\eta \,dxdv\\
&\quad -\frac{1}{2}\intorr \phi(x-y)|w-v|^2 f^\eta (y,w)f^\eta(x,v)\,dxdydvdw\\
&\quad - \int_{\Omega\times\R^d} |v-\chi_\zeta(\tilde{u})|^2 f^\eta \,dxdv  - \int_{\Omega\times\R^d} (v-\chi_\zeta(\tilde{u}))\cdot \chi_\zeta(\tilde{u}) f^\eta \,dxdv \\
&\quad - \sigma\int_{\Omega \times \R^d} (v-\chi_\zeta(\tilde{u})) \cdot \nabla_v f^\eta \,dxdv.
\end{aligned}$$
On the other hand, we get
\[
\frac{d}{dt}\left(\intor V^R f^\eta \,dxdv\right) = \intor v \cdot \nabla_x V^R f^\eta \,dxdv
\]
and
$$\begin{aligned}
\frac{d}{dt}\left(\frac{1}{2}\intoo W^\e(x-y)\rho^\eta(x) \rho^\eta(y) \,dxdy\right) &= \intoo W^\e(x-y)\partial_t\rho^\eta(x) \rho^\eta(y) \,dxdy\\
&= -\intoo W^\e(x-y)\nabla \cdot (\rho^\eta u^\eta)(x) \rho^\eta(y)\,dxdy\\
&= \intoo \nabla W^\e(x-y) \rho^\eta(y) \cdot (\rho^\eta u^\eta)(x) \,dxdy\\
&= \int_\Omega (\nabla W^\e \star \rho^\eta) \cdot  (\rho^\eta u^\eta) \,dx\\
&= \int_{\Omega\times\R^d}(\nabla W^\e \star \rho^\eta) \cdot v f^\eta \,dxdv.
\end{aligned}$$
This yields
$$\begin{aligned}
&\frac{d}{dt}\left(\frac12\int_{\Omega \times \R^d} |v|^2 f^\eta \,dxdv + \intor V^R f^\eta \,dxdv + \frac{1}{2}\intoo W^\e(x-y)\rho^\eta(x) \rho^\eta(y) \,dxdy\right)\cr
&\quad = - \int_{\Omega \times \R^d} |v|^2 f^\eta \,dxdv -\frac{1}{2}\intorr \phi(x-y)|w-v|^2 f^\eta (y,w)f^\eta(x,v)\,dxdydvdw\cr
&\qquad - \int_{\Omega\times\R^d} |v-\chi_\zeta(\tilde{u})|^2 f^\eta \,dxdv  - \int_{\Omega\times\R^d} (v-\chi_\zeta(\tilde{u}))\cdot \chi_\zeta(\tilde{u}) f^\eta \,dxdv \\
&\qquad - \sigma\int_{\Omega \times \R^d} (v-\chi_\zeta(\tilde{u})) \cdot \nabla_v f^\eta \,dxdv.
\end{aligned}$$
We then combine the previous estimates with the following entropy estimate
$$\begin{aligned}
\frac{d}{dt}\left(\int_{\Omega\times\R^d} \sigma f^\eta \log f^\eta \,dxdv\right)&= \int_{\Omega\times\R^d} \sigma(\partial_t f^\eta) \log f^\eta\,dxdv\\
&= -\sigma \int_{\Omega\times\R^d} ( v +\nabla V^R + \nabla W^\e \star \rho^\eta)\cdot \nabla_v f^\eta \,dxdv\\
&\quad + \sigma \int_{\Omega\times\R^d}F[f^\eta]\cdot \nabla_v f^\eta \,dxdv\\
&\quad - \sigma \int_{\Omega\times\R^d} (v-\chi_\zeta(\tilde{u}))\cdot \nabla_v f^\eta -\sigma^2 \int_{\Omega\times\R^d} \frac{|\nabla_v f^\eta|^2}{f^\eta}\,dxdv\\
&= \sigma d \|f_0^\eta\|_{L^1} + \sigma d \intorr \phi(x-y)f^\eta (y,w)f^\eta(x,v)\,dxdydvdw\\
&\quad - \sigma \int_{\Omega\times\R^d} (v-\chi_\zeta(\tilde{u}))\cdot \nabla_v f^\eta -\sigma^2 \int_{\Omega\times\R^d} \frac{|\nabla_v f^\eta|^2}{f^\eta}\,dxdv
\end{aligned}$$
to conclude the desired result.
\end{proof}

\subsection{Existence of the regularized kinetic equation}
Now, we provide the existence of weak solutions to \eqref{E-1} and their energy estimates. Similarly as in \cite{CCK16, MV07}, we define the mapping $\mathcal{T}: \mathcal{S} \to \mathcal{S}$, where $\mathcal{S}=L^2(\Omega \times (0,T))$ by 
\[
\tilde{u} \mapsto \mathcal{T}(\tilde{u}) := u_\e^\eta = \frac{\rho^\eta u^\eta}{\rho^\eta + \e}.
\]
First, we prove that the operator $\mathcal{T}$ is well-defined.
\begin{lemma}\label{L5.2}
For a weak solution $f^\eta$ to \eqref{E-2}, the averaged quantities $(\rho^\eta, \rho^\eta u^\eta)$ satisfy
\[
\rho^\eta \in L^p(\Omega), \quad \rho^\eta u^\eta \in [L^p(\Omega)]^d, \quad \forall \, p \in [1,\infty),
\]
and as a consequence, $\mathcal{T}$ is well-defined. 
\end{lemma}
\begin{proof}
Since the proof for the first assertion can be found in \cite{KMT13}, we omit its proof. Since $\rho^\eta$ is bounded, it suffices to show the boundedness of $u_\e^\eta$. Obviously,
\[
|u_{\e}^\eta| = \left|\frac{\rho^\eta u^\eta}{\rho^\eta + \e}\right| \le \frac{1}{\e}|\rho^\eta u^\eta|,
\]
and since $\rho^\eta u^\eta$ is bounded, $\mathcal{T}$ is well-defined.
\end{proof}
Next, we discuss the compactness of $\mathcal{T}$. Here, we state the velocity averaging lemma from \cite{PS98}.
\begin{lemma}\label{L5.3}
Let $\{f^m\}$ be bounded in $L_{loc}^p(\R^{2d+1})$ with $1 < p < \infty$ and $\{G^m\}$ be bounded in $L_{loc}^p(\R^{2d+1})$. If $f^m$ and $G^m$ satisfy
\[
f_t^m + v \cdot \nabla f^m = \nabla_v^\alpha G^m, \quad f^m|_{t=0} = f_0 \in L^p(\R^{2d}) 
\]
for some multi-index $\alpha$ and $\varphi \in \mathcal{C}_c^{|\alpha|}(\R^{2d})$, then
\[ \left\{ \int_{\R^d} f^m \varphi dv \right\} \]
is relatively compact in $L_{loc}^p(\R^{d+1})$.
\end{lemma}
We then use the previous lemma to obtain the following result, which is very similar to \cite[Lemma 2.7]{KMT13}.
\begin{lemma}\label{L5.4}
Let $\{f^m\}$  and $\{G^m\}$ be in Lemma \ref{L5.3} and assume that for $r \ge 2$,
\bq\label{av_est}
\sup_{m \in \bbn} \|f^m\|_{L^\infty(\R^{2d+1})} + \sup_{m \in \bbn}\| (|v|^{r} + |x|^2) f^m\|_{L^\infty(0,T;L^1(\R^{2d}))}  <\infty.
\eq
Then, for any $\varphi(v)$ satisfying $|\varphi(v)| \le c|v|$, the sequence
\[ \left\{ \int_{\R^d} f^m \varphi dv \right\} \]
is relatively compact in $L^q(\R^{d+1})$ for any $q \in \left(1, \frac{d+r}{d+1}\right)$.
\end{lemma}
Thanks to Lemma \ref{L5.4}, we can prove the compactness of $\mathcal{T}$.
\begin{corollary}\label{C5.1}
For a uniformly bounded sequence $\tilde{u}^m$ in $\mathcal{S}$, the sequence $\mathcal{T}(\tilde{u}^m) =  (u_\e^\eta)^m$ converges strongly in $\mathcal{S}$, up to a subsequence.
\end{corollary}
\begin{proof}
For the convergence of $\{(u_\e^\eta)^m\}$, we set
\[ 
f^m := (f^\eta)^m , \quad G^m :=\lt(\sigma \nabla_v (f^\eta)^m + (2v + \nabla V^R + \nabla W^\e \star (\rho^\eta)^m - F[(f^\eta)^m] - \chi_\zeta(\tilde{u}))(f^\eta)^m \rt),
\]
then it is easy to see $G^m \in L^p_{loc}(\R^{2d+1})$. Let us choose $r$ appeared in \eqref{av_est} sufficiently large. Then, we set $\varphi(v) =1$ and $\varphi(v) = v$ in Lemma \ref{L5.4}, respectively, and obtain the following strong convergence up to a subsequence:
$$\begin{aligned}
&(\rho^\eta)^m \to \rho^\eta \quad \quad \mbox{in} \quad L^2(\Omega \times (0,T)) \quad \mbox{and a.e.},\\
&(\rho^\eta)^m (u^\eta)^m \to \rho^\eta u^\eta \quad \mbox{in} \quad L^2(\Omega \times (0,T)),
\end{aligned}$$
and consequently, it gives the convergence of $\{(u_\e^\eta)^m\}$ up to a subsequence.
\end{proof}

\begin{remark}
Lemma \ref{L5.2} and Corollary \ref{C5.1} imply that the operator $\mathcal{T}$ is well-defined, continuous and compact. Thus, we can use the Schauder's fixed point theorem to obtain a weak solution $f^\eta$ to \eqref{E-1}.
\end{remark}
From the previous fixed point argument, the following energy inequality associated with \eqref{E-1} is obvious.
\begin{corollary}\label{C5.2}
Let $f^\eta$ be a weak solution to \eqref{E-1}. Then, it satisfies the following energy inequality:
\[
\begin{aligned}
\frac{d}{dt}&\left( \int_{\Omega \times \R^d} \left(\frac{|v|^2}{2} + V^R + \sigma \log f^\eta \right) f^\eta \,dxdv + \frac{1}{2}\intoo W^\e(x-y)\rho^\eta (x) \rho^\eta (y) \,dxdy\right)\\
&\qquad + \int_{\Omega\times\R^d}\frac{1}{f^\eta} \left|\sigma \nabla_v f^\eta - (v-\chi_\zeta(u_\e^\eta))f^\eta\right|^2  \,dxdv + \int_{\Omega\times\R^d}|v|^2 f^\eta \,dxdv\\
&\qquad +\frac{1}{2}\intorr \phi(x-y)|w-v|^2 f^\eta (y,w)f^\eta(x,v)\,dxdydvdw\\
&\qquad \quad\le \sigma d \|f_0^\eta\|_{L^1} +\sigma d\int_{\Omega}(\phi\star\rho^\eta)\rho^\eta\,dx.
\end{aligned}
\]
\end{corollary}
\begin{proof}
From the existence of the fixed point of $\mathcal{T}$ and Proposition \ref{P5.1}, it is obvious that
$$\begin{aligned}
\frac{d}{dt}&\left( \int_{\Omega \times \R^d} \left(\frac{|v|^2}{2} + V^R + \sigma \log f^\eta \right) f^\eta \,dxdv + \frac{1}{2}\intoo W^\e(x-y)\rho^\eta (x) \rho^\eta (y) \,dxdy\right)\\
&\qquad + \int_{\Omega\times\R^d}\frac{1}{f^\eta} \left|\sigma \nabla_v f^\eta - (v-\chi_\zeta(u_\e^\eta))f^\eta\right|^2  \,dxdv + \int_{\Omega\times\R^d}|v|^2 f^\eta \,dxdv\\
&\qquad +\frac{1}{2}\intorr \phi(x-y)|w-v|^2 f^\eta (y,w)f^\eta(x,v)\,dxdydvdw\\
&\qquad \quad =\int_{\Omega\times\R^d} (\chi_\zeta(u_\e^\eta)-v)\cdot \chi_\zeta(u_\e^\eta) f^\eta \,dxdv\\
&\qquad \qquad+ \sigma d \|f_0^\eta\|_{L^1} +\sigma d\intorr \phi(x-y)f^\eta(y,w)f^\eta(x,v)\,dxdydvdw.
\end{aligned}$$
Here, we note that
\[
\begin{split}
\int_{\R^d}(\chi_\zeta(u_\e^\eta)-v)\cdot \chi_\zeta(u_\e^\eta)f^\eta \,dv &= \rho^\eta(\chi_\zeta(u_\e^\eta)-u^\eta)\cdot \chi_\zeta(u_\e^\eta)\\
&= \rho^\eta \left(\frac{\rho^\eta u^\eta}{\rho^\eta +\e} - u^\eta\right)\cdot \frac{\rho^\eta u^\eta}{\rho^\eta +\e} \mathds{1}_{\{ |u_\e^\eta|\le \zeta\}}\\
&= -\e \left|\frac{\rho^\eta u^\eta}{\rho^\eta + \e}\right|\mathds{1}_{\{ |u_\e^\eta|\le \zeta\}} \le 0,
\end{split}
\]
which implies our desired estimate.
\end{proof}

\subsection{Proof of Theorem \ref{T2.1}}
Now, we provide the existence of a weak solution to \eqref{main_eq} based on the energy inequality, compactness argument and velocity averaging lemma.

\subsubsection{Convergences $R \to \infty$ and $\zeta \to \infty$}
First, we set $R=\zeta$, and we will tend $R$ to infinity. \\

\noindent $\bullet$ (Step A: Uniform bound estimates) As an initial step, we derive a upper-bound estimate which is uniform in $R$ and $\zeta$ from Corollary \ref{C5.2}. For technical reason, we also estimate
\[
\frac{d}{dt}\left(\int_{\Omega\times\R^d} \frac{|x|^2}{2} f^\eta \,dxdv\right) = \int_{\Omega \times \R^d} x \cdot v f^\eta \,dxdv \le \int_{\Omega \times \R^d} \left(\frac{|x|^2}{2} + \frac{|v|^2}{2}\right) f^\eta \,dxdv,
\]
and combine this with Corollary \ref{C5.2} to get
\begin{align}
\begin{aligned}\label{E-4}
&\int_{\Omega\times\R^d} \left(\frac{|v|^2}{2} + \frac{|x|^2}{2} + V^R+\sigma \log f^\eta \right) f^\eta \,dxdv + \frac{1}{2}\intoo W^\e(x-y)\rho^\eta (x) \rho^\eta (y) \,dxdy\\
& \quad+ \int_0^t\int_{\Omega\times\R^d}\frac{1}{f^\eta} \left|\sigma \nabla_v f^\eta - (v-\chi_\zeta(u_\e^\eta)) f^\eta\right|^2 \,dxdvds \\
& \quad+\frac{1}{2}\int_0^t\intorr \phi(x-y)|w-v|^2 f^\eta (y,w)f^\eta(x,v)\,dxdydvdwds\\
&\qquad \le \int_{\Omega\times\R^d} \left(\frac{|v|^2}{2} + \frac{|x|^2}{2} + V^R+\sigma \log f_0^\eta \right) f_0^\eta \,dxdv + \frac{1}{2}\intoo W^\e(x-y)\rho_0^\eta (x) \rho_0^\eta (y) \,dxdy\\
&\qquad \quad +\sigma d t\|f_0^\eta\|_{L^1} +\sigma d\int_0^t\int_{\Omega}(\phi\star\rho^\eta)\rho^\eta \,dxds +\int_0^t\intor \frac{|x|^2}{2}f^\eta\,dxdvds.
\end{aligned}
\end{align}
Here, we recall the following inequality from the classical result \cite{C-I-P}:
\[
2\int_{\om \times \R^d} f^\varepsilon \log_{-}f^\varepsilon \,dx dv  \le \int_{\om \times \R^d} f^\varepsilon \left( \frac{|x|^2}{2} + \frac{|v|^2}{2} \right) dx dv  + \frac{1}{e}\int_{\om \times \R^d} e^{-\frac{|v|^2}{4} - \frac{|x|^2}{4}} dx dv,
\]
where $\log_{-}g(x) := \max\{0, -\log g(x)\}$. We apply the aforementioned inequality to \eqref{E-4} and obtain
$$\begin{aligned}
&\int_{\Omega\times\R^d} \left(\frac{|v|^2}{2} + \frac{|x|^2}{2} + V^R+\sigma |\log f^\eta| \right) f^\eta \,dxdv + \frac{1}{2}\int_{\Omega \times \Omega} W^\e(x-y)\rho^\eta (x) \rho^\eta (y) \,dxdy\\
& \quad+ \int_0^t\int_{\Omega\times\R^d}\frac{1}{f^\eta} \left|\sigma \nabla_v f^\eta - (v-\chi_\zeta(u_\e^\eta))f^\eta\right|^2  \,dxdvds \\
& \quad+\frac{1}{2}\int_0^t\intorr \phi(x-y)|w-v|^2 f^\eta (y,w)f^\eta(x,v)\,dxdydvdwds\\
&\qquad \le \int_{\Omega\times\R^d}\left(\frac{|v|^2}{2} + \frac{|x|^2}{2} + V^R+\sigma |\log f_0^\eta| \right) f_0^\eta \,dxdv + \frac{1}{2}\int_{\Omega \times \Omega} W^\e(x-y)\rho_0^\eta (x) \rho_0^\eta (y) \,dxdy\\
&\qquad \quad +\sigma d t\|f_0^\eta\|_{L^1} +\sigma d\|\phi\|_{L^\infty}t \|f_0^\eta\|_{L^1}^2 +\int_0^t\int_{\Omega\times\R^d} \left(|v|^2 +|x|^2\right)f^\eta\,dxdvds + C,
\end{aligned}$$
where $C=C(T)$ is a positive constant independent of $\eta$. Thus, we use $\|f_0^\eta\|_{L^1} \le \|f_0\|_{L^1}$ and $W^\e(x) \le W(x)$, and apply Gr\"onwall's lemma to yield, for $t \in [0,T]$,
$$\begin{aligned}
\int_{\Omega\times\R^d}& \left(\frac{|v|^2}{2} + \frac{|x|^2}{2} \right) f^\eta \,dxdv \le C,
\end{aligned}$$
where $C=C(T)$ is a positive constant independent of $\eta$. Since $f^\eta$ satisfies
\[
\begin{split}
\|f^\eta(\cdot,\cdot,t)\|_{L^p}^p &+ \frac{4\sigma(p-1)}{p}\int_0^t e^{d(p-1)(2 + \|\phi\|_{L^\infty}\|f^\eta_0\|_{L^1})(t-s)}\|\nabla_v (f^\eta)^{p/2}(\cdot,\cdot,s)\|_{L^2}^2\,ds \\
&\leq \|f^\eta_0\|_{L^p}^pe^{d(p-1)(2+\|\phi\|_{L^\infty}\|f_0^\eta\|_{L^1})t}\\
&\leq \|f_0\|_{L^p}^pe^{d(p-1)(2+\|\phi\|_{L^\infty}\|f_0^\eta\|_{L^1})t}
\end{split}
\]
for every $p \in [1,\infty]$, these uniform bounds yield the following estimates:
\[
\|f^\eta\|_{L^\infty(0,T;L^p(\Omega\times\R^d))} + \|\rho^\eta\|_{L^\infty(0,T;L^{q_1}(\Omega))} + \|\rho^\eta u^\eta\|_{L^\infty(0,T;L^{q_2}(\Omega))} \le C(T),
\]
where $p \in [1,\infty]$, $q_1 \in [1,(d+2)/d)$, $q_2 \in [1,(d+2)/(d+1))$ and $C=C(T)$ is a positive constant independent of $\eta$. This uniform estimate implies the following weak convergence as $R\to\infty$ up to a subsequence:
\[
\begin{array}{lcll}
f^\eta \rightharpoonup f^\e  &\mbox{ in }&  L^\infty(0,T;L^p(\Omega\times\R^d)), & p\in[1,\infty],\\
\displaystyle \rho^\eta \rightharpoonup \rho^\e & \mbox{ in } & L^\infty(0,T;L^p(\Omega)), & p\in[1,(d+2)/d),\\
\displaystyle \rho^\eta u^\eta \rightharpoonup \rho^\e u^\e & \mbox{ in } & L^\infty(0,T;L^p(\Omega)), & p\in [1,(d+2)/(d+1)).
\end{array}
\]
Moreover, once we choose $p \in (1,(d+2)/(d+1))$ and write $G^\eta$ as
$$
G^\eta := \sigma \nabla_v f^\eta + (2v+\nabla V^R + \nabla W^\e \star \rho^\eta - F[f^\eta]-\chi_\zeta(u_\e^\eta))f^\eta,
$$
then we can see that $G^\eta \in L_{loc}^p(\Omega\times\R^d \times (0,T))$. Indeed, if we consider a bounded region $\mathcal{D} \subset \Omega\times\R^d$, the boundedness of $\nabla V^R f^\eta$ in $L^p(\mathcal{D})$ follows since
$$\begin{aligned}
|\nabla V^R f^\eta| &= \lt|(\nabla V)(x) M\left(\frac{x}{R}\right) f^\eta+\frac{1}{R} V(x) (\nabla M)\left(\frac{x}{R}\right)f^\eta\rt|\\
& \le |\nabla V (x) | f^\eta + \frac{\|\nabla M\|_{L^\infty}}{R} V(x) f^\eta,
\end{aligned}$$
and we will consider the regime $R \to \infty$. The boundedness of the others naturally follows. Thus, we let $G^\eta$ as above, set $r=2$ and apply them to Lemma \ref{L5.4} to obtain, for $p \in (1, (d+2)/(d+1))$,
\[
\begin{array}{lcl}
\displaystyle \rho^\eta \to \rho^\e & \mbox{ in } & L^p(\Omega \times (0,T)) \ \mbox{ and a.e.},\\
\displaystyle \rho^\eta u^\eta \to \rho^\e u^\e & \mbox{ in } & L^p(\Omega \times (0,T))
\end{array}
\]
as $R\to\infty$, up to a subsequence. \\

\noindent $\bullet$ (Step B: Existence of weak solutions and entropy inequality) Now, it remains to show that the limit $f^\e$ satisfies the following equation in distributional sense:
\begin{equation}\label{E-7}
\begin{split}
&\pa_t f^\e + v\cdot\nabla_x f^\e -  \nabla_v \cdot \lt((v + \lt(\nabla V + \nabla W^\e \star \rho^\e \rt))f^\e \rt)+ \nabla_v \cdot \lt(F[f^\e]f^\e\rt) \cr
&\hspace{2cm}=  \nabla_v \cdot ( (v -  u_\e^\e)f^\e + \sigma \nabla_v f^\e),
\end{split}
\end{equation}
and also the entropy inequality:
\begin{align}\label{E-8}
\begin{aligned}
&\int_{\Omega \times \R^d} \left(\frac{|v|^2}{2} + V + \sigma \log f^\e \right) f^\e \,dxdv + \frac{1}{2}\intoo W^\e(x-y)\rho^\e (x) \rho^\e (y) \,dxdy\\
&\quad + \int_0^t\int_{\Omega\times\R^d}\frac{1}{f^\e} \left|\sigma \nabla_v f^\e - (v-u_\e^\e)f^\e\right|^2  \,dxdvds + \int_0^t\int_{\Omega\times\R^d}|v|^2 f^\e \,dxdvds\\
&\quad +\frac{1}{2}\int_0^t\int_{\Omega^2\times\R^{2d}}\phi(x-y)|w-v|^2 f^\e (y,w)f^\e(x,v)\,dxdydvdwds\\
&\qquad \le \int_{\Omega \times \R^d} \left(\frac{|v|^2}{2} + V + \sigma \log f_0^\e \right) f_0^\e \,dxdv + \frac{1}{2}\intoo W^\e(x-y)\rho_0^\e (x) \rho_0^\e (y) \,dxdy \\
&\qquad \quad+\sigma dt \|f_0\|_{L^1} +\sigma d\int_0^t \int_{\Omega}(\phi \star \rho^\e)\rho^\e\,dxds.
\end{aligned}
\end{align}
For $f^\e$ to be a weak solution to \eqref{E-7}, it suffices to show the following convergence in distribution sense since the others are obvious:
\begin{equation}\label{E-9}
\begin{cases}
\mbox{(i)}~~\nabla W^\e \star \rho^\eta \to \nabla W^\e \star \rho^\e,\\[2mm]
\mbox{(ii)}~~F[f^\eta]f^\eta \to F[f^\e]f^\e,\\[2mm]
\mbox{(iii)}~~ \chi_R(u_\e^\eta) f^\eta \to u_\e^\e f^\e.\\
\end{cases}
\end{equation}
$\diamond$ (Step B-1: Convergence of \eqref{E-9} (i)) We choose $\Psi\in \mathcal{C}_c^\infty(\Omega \times\R^d\times [0,T]) $ and write 
\[
\rho^\eta_\Psi := \int_{\R^d} f^\eta(v)\Psi(v)\,dv.
\]
Then, we have
$$\begin{aligned}
\int_0^t& \int_{\Omega\times\R^d} \left[(\nabla W^\e \star \rho^\eta)f^\eta - (\nabla W^\e \star \rho^\e)f^\e\right]\Psi\,dxdvds\\
&=\int_0^t \int_{\Omega\times\R^d} \nabla W^\e \star (\rho^\eta -\rho^\e)\rho_\Psi^\eta\,dxds + \int_0^t \int_{\Omega\times\R^d} (\nabla W^\e \star\rho^\e )(f^\eta-f^\e) \Psi\,dxdvds\\
&=: K_1^1 +K_1^2.
\end{aligned}$$
For $K_1^1$, since $f^\eta$ is uniformly bounded in $L^\infty(\Omega \times\R^d\times (0,T))$ and $\Psi$ is compactly supported, it is obvious that
\[
\rho_\Psi^\eta \in L^p(0,T; L^q(\Omega)) \quad \mbox{for any } \ p,q \in [1,\infty], \quad \mbox{uniformly in } \eta.
\]
Note that $\rho_{|\Psi|}^\eta := \int_{\R^d} f^\eta(v)|\Psi|(v)\,dv$ also satisfies the above estimate. Thus, we have
$$\begin{aligned}
K_1^1 &= \int_0^t \int_\Omega (\nabla W^\e \mathds{1}_{\{ |\cdot | \le 1\}})\star(\rho^\eta - \rho^\e) \rho_\Psi^\eta \,dxds + \int_0^t \int_\Omega (\nabla W^\e \mathds{1}_{\{ |\cdot | > 1\}})\star(\rho^\eta - \rho^\e) \rho_\Psi^\eta \,dxds\\
&\le \big\||\nabla W(\cdot)|\mathds{1}_{\{|\cdot|\le 1\}} \big\|_{L^1(\Omega)} \|\rho^\eta - \rho^\e\|_{L^p(\Omega \times (0,T))} \|\rho_\Psi^\eta\|_{L^{p'}(\Omega \times (0,T))}\\
&\quad + \big\||\nabla W(\cdot)|\mathds{1}_{\{|\cdot|>1\}} \cdot \mathds{1}_{\mbox{supp}_x \Psi}\big\|_{L^{p'}(\Omega)} \|\rho^\eta - \rho^\e\|_{L^p(\Omega \times (0,T))}\|\rho_\Psi^\eta\|_{L^{p'}(0,T;L^1(\Omega))}, 
\end{aligned}$$
where $p \in (1,(d+2)/(d+1))$, $p'$ is the H\"older conjugate of $p$ (thus $p' > d+2$), $\mbox{supp}_x \Psi$ denotes the support of $\Psi$ in $x$ which is compact:
\[
\mbox{supp}_x \Psi := \overline{\{ x\in\Omega \ | \ \Psi(x,v) \neq 0 \ \mbox{ for some } \ v \in \R^d\}},
\]
and we used Young's convolution inequality:
\[
\into (f\star g)h \,dx \le \|f\|_{L^p}\|g\|_{L^q}\|h\|_{L^r}, \quad \frac{1}{p} + \frac{1}{q} + \frac{1}{r} = 2.
\]
Thus, thanks to the strong convergence of $\rho^\eta$, we have $K_1^1 \to 0$ as $R \to \infty$.\\

\noindent For $K_1^2$, it naturally follows from the compact support of $\Psi$ that
\[
(\nabla W^\e \star \rho^\e) \Psi \in L^1(0,T;L^p(\Omega)) \ \mbox{ for some } \ p \in (1,\infty). 
\]
Thus, due to the weak convergence of $f^\eta$, we get $K_1^2 \to 0$ as $R \to \infty$.\\

\noindent $\diamond$ (Step B-2: Convergence of \eqref{E-9} (ii)) We note that
\[
F[f^\eta](x,v)= \int_{\Omega\times\R^d}\phi(x-y)(w-v)f^\eta(y,w)\,dydw = \int_\Omega \phi(x-y)\lt((\rho^\eta u^\eta)(y) - v \rho^\eta(y)\rt)dy.
\]
Thus, for $\Psi\in \mathcal{C}_c^\infty([0,T]\times\Omega \times \R^d)$, we get
$$\begin{aligned}
&\int_0^t \int_{\Omega\times\R^d} (F[f^\eta]f^\eta - F[f^\e]f^\e)\Psi \,dxdvds\\
& =\int_0^t \int_{\Omega^2 \times \R^d} \phi(x-y)\left[ ((\rho^\eta u^\eta)(y) - v\rho^\eta(y)) - ((\rho^\e u^\e)(y) - v\rho^\e(y)) \right](f^\eta \Psi)(x,v)\,dxdydvds\\
&\quad + \int_0^t \int_{\Omega^2 \times \R^d}\phi(x-y)((\rho^\e u^\e)(y) - v\rho^\e(y))(f^\eta - f^\e)(x,v)\Psi(x,v)\,dxdydv\\
&=: K_2^1 + K_2^2. 
\end{aligned}$$
For $K_2^1$, we use Young's convolution inequality, the uniform boundedness of $\rho_\Psi^\eta$ and the compact support of $\Psi$ to obtain
$$\begin{aligned}
K_2^1&= \int_0^t \int_\Omega \phi\star(\rho^\eta u^\eta - \rho^\e u^\e) \rho_\Psi^\eta \,dxds + \int_0^t \int_{\Omega\times\R^d} \phi\star(\rho^\e - \rho^\eta)(vf^\eta)(x,v) \cdot \Psi(x,v)\,dxdvds\\
&\le \|\phi\|_{L^\infty} \|\rho^\eta u^\eta - \rho^\e u^\e\|_{L^p((0,T)\times\Omega)} \|\rho_\Psi^\eta\|_{L^{p'}((0,T)\times\Omega)}\\
&\quad +  |\mbox{supp}_v \Psi| \|\phi\|_{L^\infty} \|\rho^\eta - \rho^\e \|_{L^p((0,T)\times\Omega)} \|\rho_{|\Psi|}^\eta\|_{L^{p'}((0,T)\times\Omega)},
\end{aligned}$$
where $p \in (1,(d+2)/(d+1))$ and hence, $K_2^1 \to 0$ as $R\to \infty$.\\

\noindent For $K_2^2$, we use the compact support of $\Psi$ to get
\[
(\phi\star(\rho^\e u^\e) - v\phi\star \rho^\e) \Psi \in L^1(0,T;L^p(\Omega\times\R^d))
\]
for some $p \in (1,\infty)$, and we combine this with the weak convergence of $f^\eta$ to yield $K_2^2 \to 0$ as $R \to \infty$.\\

\noindent $\diamond$ (Step B-3: Convergence of \eqref{E-9} (iii)) Again, for $\Psi \in \mathcal{C}_c^\infty([0,T]\times\Omega \times \R^d)$, we estimate
$$\begin{aligned}
\int_0^t& \int_{\Omega\times\R^d} (\chi_R(u_\e^\eta)f^\eta - u_\e^\e f^\e)\Psi \,dxdvds\\
&=\int_0^t \int_{\Omega\times\R^d} (u_\e^\eta f^\eta - u_\e^\e f^\e)\Psi \,dxdvds + \int_0^t \int_{\Omega\times\R^d} u_\e^\eta f^\eta\mathds{1}_{\{|u_\e^\eta|>R\}} \Psi \,dxdvds\\
&=\int_0^t \int_\Omega (u_\e^\eta - u_\e^\e)\rho_\Psi^\eta \,dxds + \int_0^t \int_{\Omega\times\R^d} u_\e^\e (f^\eta-f^\e) \Psi \,dxdvds \\
&\quad + \int_0^t \int_{\Omega\times\R^d} u_\e^\eta f^\eta\mathds{1}_{\{|u_\e^\eta|>R\}} \Psi \,dxdvds\\
&=: K_3^1 + K_3^2 + K_3^3.
\end{aligned}$$
For $K_3^1$, we find
\[
K_3^1 = \int_0^t \int_\Omega \left[\left(\frac{1}{\rho^\eta + \e} - \frac{1}{\rho^\e + \e}\right)(\rho^\e u^\e) + \frac{1}{\rho^\eta +\e} (\rho^\eta u^\eta - \rho^\e u^\e) \right]\rho_\Psi^\eta\,dxds.
\]
Here, since $\rho^\eta \to \rho^\e$ a.e. and $\rho_\Psi^\eta \in L^\infty((0,T)\times\Omega)$, we have
\[
\left(\frac{1}{\rho^\eta + \e} - \frac{1}{\rho^\e + \e}\right)(\rho^\e u^\e)\rho_\Psi^\eta \to 0, \quad \mbox{a.e.}
\]
as $R \to \infty$, and since
\[
\left|\left(\frac{1}{\rho^\eta + \e} - \frac{1}{\rho^\e + \e}\right)(\rho^\e u^\e)\rho_\Psi^\eta\right| \le \frac{2\|\rho_\Psi^\eta\|_{L^\infty}}{\e}|\rho^\e u^\e| \le C(\e) |\rho^\e u^\e|,
\]
where $C = C(\e)$ is a positive constant indepedent of $\eta$, we use the dominated convergence theorem to get
\[
\int_0^t \int_\Omega \left(\frac{1}{\rho^\eta + \e} - \frac{1}{\rho^\e + \e}\right)(\rho^\e u^\e)\rho_\Psi^\eta \,dxds \to 0
\] 
as $R \to \infty$. Moreover, we estimate
\[
\int_0^t \int_\Omega  \frac{1}{\rho^\eta +\e} (\rho^\eta u^\eta - \rho^\e u^\e) \rho_\Psi^\eta\,dxds\le \frac{1}{\e}\|\rho^\eta u^\eta -\rho^\e u^\e\|_{L^p} \|\rho_\Psi^\eta\|_{L^{p'}},
\]
where $p \in (1,(d+2)/(d+1))$, and hence we can get $K_3^1 \to 0$ as $R \to \infty$. For the estimate of $K_3^2$, it is obvious that $u_\e^\e \Psi \in L^1(0,T;L^p(\Omega\times\R^d))$ for some $p\in(1,\infty)$. Thus, the weak convergence implies $K_3^2 \to 0$  as $R \to \infty$. Finally we use 
\[
|u^\eta| \le \left(\frac{\int_{\R^d} |v|^2 f^\eta \,dv}{\int_{\R^d} f^\eta \,dv}\right)^{1/2}
\] 
to get
\[
K_3^3 \le \frac{1}{R} \int_0^t \int_{\Omega\times\R^d}|u_\e^\eta|^2 f^\eta \Psi\,dxdvds\le \frac{\|\Psi\|_{L^\infty}}{R} \int_0^t\int_{\Omega\times\R^d} |v|^2 f^\eta \,dxdvds \to 0
\]
as $R \to \infty$. Hence we can find out that $f^\e$ becomes a weak solution to \eqref{E-7}.\\

\noindent $\diamond$ (Step B-4: Entropy inequality) For \eqref{E-8}, we first take the liminf on the left hand side of Corollary \ref{C5.2}, convexity of the entropy and use $\|f_0^\eta \|_{L^1} \le \|f_0\|_{L^1}$ to get
\[
\begin{split}
&\int_{\Omega \times \R^d} \left(\frac{|v|^2}{2} + V + \sigma \log f^\e \right) f^\e \,dxdv + \frac{1}{2}\int_{\Omega^2} W^\e(x-y)\rho^\e (x) \rho^\e (y) \,dxdy\\
&\quad + \int_0^t\int_{\Omega\times\R^d}\frac{1}{f^\e} \left|\sigma \nabla_v f^\e - (v-u_\e^\e)f^\e\right|^2  \,dxdvds + \int_0^t\int_{\Omega\times\R^d}|v|^2 f^\e \,dxdvds\\
&\quad +\frac{1}{2}\int_0^t\int_{\Omega^2\times\R^{2d}}\phi(x-y)|w-v|^2 f^\e (y,w)f^\e(x,v)\,dxdydvdwds\\
&\qquad \le \liminf_{R \to 0}\left(\int_{\Omega \times \R^d} \left(\frac{|v|^2}{2} + V + \sigma \log f_0^\eta \right) f_0^\eta \,dxdv + \frac{1}{2}\int_{\Omega^2} W^\e(x-y)\rho_0^\eta (x) \rho_0^\eta (y) \,dxdy\right)\\
&\qquad \quad +\sigma dt \|f_0\|_{L^1} +\sigma d\liminf_{R \to 0}\int_0^t \int_{\Omega}(\phi \star \rho^\eta)\rho^\eta\,dxds.
\end{split}
\] 
Here, we use the reverse Fatou's lemma and the pointwise convergences $\rho_0^\eta \to \rho_0^\e=\rho_0$ and $f_0^\eta \to f_0^\e=f_0$ to get
$$\begin{aligned}
&\liminf_{R \to 0}\left(\int_{\Omega \times \R^d} \left(\frac{|v|^2}{2} + V + \sigma \log f_0^\eta \right) f_0^\eta \,dxdv + \frac{1}{2}\int_{\Omega \times \Omega} W^\e(x-y)\rho_0^\eta (x) \rho_0^\eta (y) \,dxdy\right)\\
&\quad\le \limsup_{R \to 0}\left(\int_{\Omega \times \R^d} \left(\frac{|v|^2}{2} + V + \sigma \log f_0^\eta \right) f_0^\eta \,dxdv + \frac{1}{2}\int_{\Omega \times \Omega} W^\e(x-y)\rho_0^\eta (x) \rho_0^\eta (y) \,dxdy\right)\\
&\quad\le \int_{\Omega \times \R^d} \left(\frac{|v|^2}{2} + V + \sigma \log f_0 \right) f_0 \,dxdv + \frac{1}{2}\int_{\Omega \times \Omega} W^\e(x-y)\rho_0 (x) \rho_0 (y) \,dxdy,
\end{aligned}$$
and we claim that the following convergence holds:
\[
\lim_{R \to 0}\int_0^t \int_{\Omega}(\phi \star \rho^\eta)\rho^\eta\,dxds = \int_0^t \int_{\Omega}(\phi \star \rho^\e)\rho^\e\,dxds.
\]
For this, we present a theorem similar to Vitali convergence theorem whose proof is presented in Appendix B. Note that when $\Omega=\T^d$, the condition (ii) is unnecessary.
\begin{theorem}\label{thm_v}
A sequence $\{h_n\}$ in $L^1(\Omega)$ converges to $h\in L^1(\Omega)$ in $L^1(\Omega)$ if the following three conditions hold:
\begin{enumerate}
\item[(i)] $h_n$ converges to $h$ almost everywhere.

\item[(ii)] for every $\e>0$, there exists $L>0$ such that
\[
\sup_{n \in \N} \int_{|x|> L} |h_n|\,dx < \e.
\]
\item[(iii)] for every $\e>0$, there exists $\delta = \delta(\e)>0$ such that whenever $m(E)<\delta$,
\[
\sup_{n \in \N} \int_E |h_n| \,dx <\e.
\]
\end{enumerate}
\end{theorem}
In our case, since $f^\eta \rightharpoonup f^\e $ in $L^1(\om \times \R^d \times (0,T))$ and $\phi \in L^\infty(\om)$, we obtain
$$\begin{aligned}
\int_0^t \into \phi(x-y) \rho^\eta(y,t)\,dydt &=\int_0^t\intor \phi(x-y)f^\eta(y,w,t)\,dydwdt \cr
&\to \int_0^t\intor \phi(x-y)f^\e(y,w,t)\,dydwdt\cr
&=\int_0^t \into \phi(x-y) \rho^\e(y,t)\,dydt
\end{aligned}$$
for each $x \in \om$ and $t \in (0,T)$. This implies the convergence of $\phi\star\rho^\eta$ to $\phi\star \rho^\e$ almost everywhere in $\om \times (0,t)$. On the other hand, we also know that $\rho^\eta \to \rho^\e$ almost everywhere, and thus we have the convergence of $(\phi\star\rho^\eta)\rho^\eta$ to $(\phi\star \rho^\e)\rho^\e$ almost everywhere in $\om \times (0,t)$. 

For the condition (ii) in Theorem \ref{thm_v}, we let $L>0$ and
\[
\int_{|x|>L} (\phi\star\rho^\eta)\rho^\eta \,dx \le \frac{1}{L}\int_{|x|>L} |x| (\phi\star\rho^\eta)\rho^\eta\,dx \le \frac{\|\phi\star\rho^\eta\|_{L^\infty}}{L}\int_\Omega |x|\rho^\eta\,dx \le \frac{\|\phi\|_{L^\infty}}{L}\left(\int_\Omega |x|^2\rho^\eta \,dx\right)^{1/2},
 \]
and we can deduce the condition (ii) from the above. For the third condition (uniform integrability condition), we choose a measurable set $E \subset \Omega$ with $m(E)<\delta$. Then, we have
\[
\int_E (\phi\star\rho^\eta)\rho^\eta \,dx \le \|\phi\star\rho^\eta\|_{L^\infty} \|\rho^\eta\|_{L^p} m(E)^{1/p'} \le Cm(E)^{1/p'},
\]
where $p \in (1,(d+2)/d)$ and $C$ is a constant independent of $\eta$, and this implies the uniform integrability condition. This concludes our desired result.

\subsubsection{Convergence $\e \to 0$}Finally, it remains to prove the convergence as $\e \to 0$. We note that the weak convergence of $f^\eta$ to $f^\e$ implies the following uniform upper bound estimate:
\begin{equation}\label{E-10}
\|f^\e\|_{L^\infty(0,T;L^p(\Omega\times\R^d))}^p + \frac{4\sigma(p-1)}{p}\int_0^T \|\nabla_v (f^\e)^{p/2}(\cdot,\cdot,s)\|_{L^2}^2\,ds \leq \|f_0\|_{L^p}^pe^{d(p-1)(2+\|\phi\|_{L^\infty})t}
\end{equation}
for $p\in[1,\infty]$. Thus, we combine \eqref{E-10} with the entropy inequality \eqref{E-8} to get the following weak convergence up to a subsequence as before:
\[
\begin{array}{lcll}
f^\e \rightharpoonup f  &\mbox{ in }&  L^\infty(0,T;L^p(\Omega\times\R^d)), & p\in[1,\infty],\\
\displaystyle \rho^\e \rightharpoonup \rho & \mbox{ in } & L^\infty(0,T;L^p(\Omega)_, & p\in[1,(d+2)/d),\\
\displaystyle \rho^\e u^\e \rightharpoonup \rho u & \mbox{ in } & L^\infty(0,T;L^p(\Omega)), & p\in [1,(d+2)/(d+1)).
\end{array}
\]
Moreover, applying the velocity averaging lemma, Lemma \ref{L5.4}, asserts the strong convergence up to a subsequence for $p \in (1,(d+2)/(d+1))$:
\bq\label{st_conv}
\begin{array}{lcl}
\displaystyle \rho^\e \to \rho & \mbox{ in } & L^p((0,T)\times\Omega) \ \mbox{ and a.e.},\\
\displaystyle \rho^\e u^\e \to \rho u & \mbox{ in } & L^p((0,T)\times\Omega)
\end{array}
\eq
as $\e \to 0$. Now, to show that $f$ is a weak solution to \eqref{main_eq}, it suffices to show the following convergence in distribution sense, since the others are obvious or can be obtained in the same way from the previous argument:
\begin{equation}\label{C-11}
\begin{cases}
\mbox{(i)}~~(\nabla W^\e \star \rho^\e)f^\e \to (\nabla W \star \rho)f,\\[2mm]
\mbox{(ii)}~~ u_\e^\e f^\e  \to u f.
\end{cases}
\end{equation}

\noindent $\diamond$ (Convergence $\eqref{C-11}$ (i)) We choose again  $\Psi\in\mathcal{C}_c^\infty(\Omega \times \R^d \times [0,T])$ and get
$$\begin{aligned}
\int_0^t& \int_{\Omega\times\R^d} \left[(\nabla W^\e \star \rho^\e)f^\e - (\nabla W \star \rho)f\right]\Psi\,dxdvds\\
&= \int_0^t \int_{\Omega\times\R^d} (\nabla (W^\e - W)\star \rho^\e )\rho_\Psi^\e \,dxds\\
&\quad+\int_0^t \int_{\Omega\times\R^d} \nabla W \star (\rho^\e -\rho)\rho_\Psi^\e\,dxds + \int_0^t \int_{\Omega\times\R^d} (\nabla W \star\rho )(f^\e-f) \Psi\,dxdvds\\
&=: K_4^1 +K_4^2+K_4^3.
\end{aligned}$$
Since the estimates for $K_4^2$ and $K_4^3$ are similar to those for $K_1^1$ and $K_1^2$, respectively, we only need to show $K_4^1 \to 0$ as $\e \to 0$. Still, thanks to the uniform-in-$\e$ estimate for $f^\e$ in $L^\infty(\Omega\times\R^d\times(0,T))$ and the compact support of $\Psi$, we find
\[
\rho_\Psi^\e \in L^p(0,T; L^q(\Omega))
\]
for any $p,q \in [1,\infty]$ uniformly in $\e$. This gives
$$\begin{aligned}
K_4^1&= \int_0^t \int_{\Omega} (\nabla (W^\e - W)(\cdot)(\mathds{1}_{\{|\cdot|\le 1\}} +\mathds{1}_{\{|\cdot|> 1\}} ) \star \rho^\e )\rho_\Psi^\e \,dxds\\
&\le \|\nabla(W^\e - W)(\cdot) \mathds{1}_{\{|\cdot|\le 1\}}\|_{L^1(\Omega)} \|\rho^\e\|_{L^p(\Omega \times (0,T))}\|\rho_\Psi^\e\|_{L^{p'}(\Omega \times (0,T))}\\
&\quad + \|\nabla(W^\e-W)(\cdot)\mathds{1}_{\{|\cdot|>1\}}\mathds{1}_{\mbox{supp}_x \Psi}\|_{L^{p'}(\Omega)}\|\rho^\e\|_{L^p(\Omega \times (0,T))} \|\rho_\Psi^\e\|_{L^{p'}(0,T;L^1(\Omega))},
\end{aligned}$$
where $p \in (1,(d+2)/(d+1))$ and we used the uniform bound for $\rho^\e$, $\rho_\Psi^\e$ and the dominated convergence theorem for the convergence of the interaction potential term.\\

\noindent $\diamond$ (Convergence of $\eqref{C-11}$ (ii))  Although the proof is almost the same as that of \cite[Lemma 4.4]{KMT13}, we present here for readers' convenience. First, consider a test function $\Psi$ of the form $\Psi(x,v,t) := \psi(x,t)\varphi(v)$. Thus, $\varphi \in C_c^\infty (\R^d)$ and we similarly write $\rho_\varphi^\e := \int_{\R^d} f^\e \varphi(v)\,dv$. Then, we write
\[
\int_0^t \int_{\Omega\times\R^d} f^\e u_\e^\e \Psi\,dxdvds = \int_0^t \int_\Omega u_\e^\e \rho_\varphi^\e \psi\,dxds.
\]
Here, for $p \in (1,(d+2)/(d+1))$, we get
\[
\begin{split}
\|u_\e^\e \rho_\varphi^\e\|_{L^p} &\le \|\varphi\|_{L^\infty} \|\rho^\e\|_{L^{p/(2-p)}}^{1/2} \|\sqrt{\rho^\e} u_\e^\e\|_{L^2}\\
&\le \|\varphi\|_{L^\infty} \|\rho^\e\|_{L^{p/(2-p)}}^{1/2} \left(\int_{\Omega\times\R^d} |v|^2 f^\e \,dxdv\right)^{1/2},
\end{split}
\]
which gives the uniform bound for $u_\e^\e \rho_\varphi^\e$ in $L^p(\Omega)$, since $p/(2-p)\in (1,(d+2)/d)$, and we already have the uniform bound for $\rho^\e$ in $L^q$ with $q \in (1, (d+2)/d)$ and $|v|^2 f^\e$ in $L^1$. Thus, we can find $m$ such that, up to a subsequence,
\[
u_\e^\e \rho_\varphi^\e \rightharpoonup m \quad \mbox{ in } \ L^\infty(0,T;L^p(\Omega)),\quad  \forall \, p \in (1,(d+2)/(d+1)).
\]
It remains to show that $m = u \rho_\varphi$. For this, we let $h_1, h_2>0$ and define
\[
\mathcal{A}_{h_1}^{h_2} := \{(x,t) \in (B(0,h_1) \cap \Omega)\times (0,T) \ : \ \rho(x,t) > h_2\}.
\]
For each $h_1$ and $h_2$, we combine the pointwise convergence of $\rho^\e$ to $\rho$ with Egorov's theorem to deduce that for every $\delta>0$, we may choose $A_\delta \subset A_{h_1}^{h_2}$ satisfying
\[
|A_{h_1}^{h_2} \setminus A_\delta| < \delta \quad \mbox{and} \quad \rho^\e \to \rho \quad \mbox{as $\e \to 0$ uniformly on } \ A_\delta. 
\]
Then, for a sufficiently small $\e$, we have $\rho^\e >h_2/2$ on $A_\delta$ and thus we get
\[
 u_\e^\e \rho_\varphi^\e = \frac{\rho^\e u^\e}{\rho^\e +\e} \rho_\varphi^\e \to m = u\rho_\varphi, \quad \mbox{on }\ A_\delta.
\]
Since the choices of $h_1$, $h_2$ and $\delta$ were arbitrary, we now obtain
\[
m=u\rho_\varphi \quad \mbox{on} \quad \{\rho>0\},
\]
and therefore, we have
\[
\int_0^t \int_{\Omega\times\R^d} f^\e u_\e^\e \Psi\,dxdvds = \int_0^t \int_\Omega u_\e^\e \rho_\varphi^\e \psi\,dxds \to \int_0^t u \rho_\varphi \psi \,dxds = \int_0^t \int_{\Omega\times\R^d} uf \Psi \,dxdvds
\]
for all test functions $\Psi$ of the form $\Psi(x,v,t) = \psi(x,t)\varphi(v)$. Thus, we conclude that $f$ is a weak solution to \eqref{main_eq}. It remains to show that the weak solution obtained above satisfies the entropy inequality \eqref{entro_1}. Note that the regularized solutions $f^\e$ satisfies the entropy inequality \eqref{E-8} and the strong compactness of macroscopic fields $\rho^\e, \rho^\e u^\e$ are obtained in \eqref{st_conv} via the velocity averaging lemma.  Thus we can use a similar argument as in the previous step together with Fatou's lemma to have the following entropy inequality:
\[
\begin{split}
&\int_{\Omega \times \R^d} \left(\frac{|v|^2}{2} + V + \sigma \log f \right) f \,dxdv + \frac{1}{2}\int_{\Omega} (W\star\rho) \rho \,dx\\
&\qquad + \int_0^t\int_{\Omega\times\R^d}\frac{1}{f} \left|\sigma \nabla_v f - (v-u)f \right|^2  \,dxdvds + \int_0^t\int_{\Omega\times\R^d}|v|^2 f \,dxdvds\\
&\qquad +\frac{1}{2}\int_0^t\intorr \phi(x-y)|w-v|^2 f (y,w)f(x,v)\,dxdydvdwds\\
&\qquad \quad \le \int_{\Omega \times \R^d} \left(\frac{|v|^2}{2} + V + \sigma \log f_0 \right) f_0 \,dxdv + \frac{1}{2}\int_{\Omega} (W\star\rho_0) \rho_0 \,dx\\
& \qquad \qquad +\sigma dt \|f_0\|_{L^1} +\sigma d\int_0^t \int_{\Omega}(\phi \star \rho)\rho\,dxds.
\end{split}
\]

\subsection{Global-in-time existence of weak solutions for dimensions $d=1,2$} In this subsection, we investigate the global existence of weak solutions to \eqref{main_eq} for the case $d=1,2$. In these cases, the positivity of the interaction energy is not guaranteed, and as a result, it makes some problems in obtaining the uniform upper bound estimates for solutions. Once we can find a way to get the uniform bound, then other parts of the proof for the existence of weak solutions can follow from almost the same analysis as in the previous subsections. Since the results for $\Omega = \T^d$ is analogous to the case $\Omega = \R^d$, we only consider the $\Omega = \R^d$. Let us introduce the regularized fundamental solutions of the Laplace's equation in $d=1,2$:
\[
W^\e(x) := \left\{\begin{array}{lcl}
-\frac{1}{2}\sqrt{\e + |x|^2} & \mbox{ if } & d=1,\\[2mm]
-\frac{1}{4\pi}\log(\e + |x|^2) & \mbox{ if } & d=2.
\end{array}\right.
\]
Then, since we also have $\nabla W^\e$ is bounded and smooth, global-in-time existence of weak solutions to the equation \eqref{E-2} is clear. Moreover, we can also deduce that the entropy inequality \eqref{E-4} and the following upper bound estimate hold:
$$\begin{aligned}
&\int_{\R^d \times \R^d} \left(\frac{|v|^2}{2} + \frac{|x|^2}{2} +\sigma |\log f^\eta| \right) f^\eta \,dxdv + \frac{1}{2}\int_{\R^d \times \R^d} W^\e(x-y)\rho^\eta (x) \rho^\eta (y) \,dxdy\\
& \quad+ \int_0^t\int_{\R^d \times \R^d}\frac{1}{f^\eta} \left|\sigma \nabla_v f^\eta - (v-\chi_\zeta(u_\e^\eta))f^\eta\right|^2  \,dxdvds \\
& \quad+\frac{1}{2}\int_0^t\int_{\R^{2d} \times \R^{2d}}\phi(x-y)|w-v|^2 f^\eta (y,w)f^\eta(x,v)\,dxdydvdwds\\
&\qquad \le \int_{\R^d \times \R^d}\left(\frac{|v|^2}{2} + \frac{|x|^2}{2} +\sigma |\log f_0^\eta| \right) f_0^\eta \,dxdv + \frac{1}{2}\int_{\R^d \times \R^d} W^\e(x-y)\rho_0^\eta (x) \rho_0^\eta (y) \,dxdy\\
&\qquad \qquad +\sigma d t\|f_0^\eta\|_{L^1} +\sigma d\|\phi\|_{L^\infty}t \|f_0^\eta\|_{L^1}^2 +\int_0^t\int_{\R^d \times \R^d} \left(|v|^2 +|x|^2\right)f^\eta\,dxdvds + C,
\end{aligned}$$
where $C=C(T)$ is a positive constant independent of $\eta$. When $d=1$, one uses Young's inequality to get
$$\begin{aligned}
-\frac{1}{2}\int_{\R \times \R}W^\e(x-y)\rho^\eta(x)\rho^\eta(y)\,dxdy
&= \frac{1}{4}\int_{\R \times \R}\sqrt{\e + |x-y|^2} \rho^\eta(x)\rho^\eta(y)\,dxdy\\
&\le \frac{1}{4}\int_{\R \times \R}\left(1+ \frac{1}{4}(\e + |x-y|^2)\right)\rho^\eta(x)\rho^\eta(y)\,dxdy\\
&\le \frac{1}{16}(4 + \e) + \frac{1}{16}\int_{\R \times \R}(|x|^2 + |y|^2)\rho^\eta(x) \rho^\eta(y)\,dxdy\\
&\le C + \frac{1}{8}\int_{\R}|x|^2 \rho^\eta \,dx = C+\frac{1}{8}\int_{\R\times\R} |x|^2 f^\eta \,dxdv,  
\end{aligned}$$
and this gives
$$\begin{aligned}
&\frac{1}{2}\int_{\R\times\R} \left(\frac{|v|^2}{2} + \frac{|x|^2}{2} +\sigma |\log f^\eta| \right) f^\eta \,dxdv\\
& \quad+ \int_0^t\int_{\R\times\R}\frac{1}{f^\eta} \left|\sigma \nabla_v f^\eta - (v-\chi_\zeta(u_\e^\eta)) f^\eta\right|^2 \,dxdvds \\
& \quad+\frac{1}{2}\int_0^t\int_{\R^2\times\R^2}\phi(x-y)|w-v|^2 f^\eta (y,w)f^\eta(x,v)\,dxdydvdwds\\
&\qquad \le \int_{\R\times\R}\left(\frac{|v|^2}{2} + \frac{|x|^2}{2} +\sigma |\log f_0^\eta| \right) f_0^\eta \,dxdv +\sigma d t\|f_0^\eta\|_{L^1} +\sigma d\|\phi\|_{L^\infty}t \|f_0^\eta\|_{L^1}^2 \\
&\qquad \qquad +\int_0^t\int_{\R\times\R} \left(|v|^2 +|x|^2\right)f^\eta\,dxdvds + C,
\end{aligned}$$
where $C= C(T)$ is a constant independent of $\eta$, which implies the desired uniform upper bound estimate and this can be also used when $\e \to 0$. For $d=2$, we note that the following inequality holds:
$$\begin{aligned}
\e + |x-y|^2 &\le (1+\e)(1+|x-y|^2)\\
&\le (1+\e)(1+2|x|^2 + 2|y|^2)\\
&\le 2(1+\e)(1+|x|^2)(1+|y|^2), 
\end{aligned}$$
which subsequently gives
\[
\log(\e + |x-y|^2)\le \log 2(1+\e) + \log(1+|x|^2) + \log(1+|y|^2).
\]
We use the above inequality and $\log(1+x) \le x$ on $x\ge 0$  to get
$$\begin{aligned}
&\frac{1}{2}\int_{\R^2 \times \R^2}W^\e(x-y)\rho^\eta(x)\rho^\eta(y)\,dxdy\\
&\quad =-\frac{1}{8\pi} \int_{\R^2\times\R^2} \log(\e+|x-y|^2)\rho^\eta(x)\rho^\eta(y)\,dxdy\\
&\quad \ge -\frac{1}{8\pi}\int_{\R^2\times\R^2} \left[\log 2(1+\e) + \log(1+|x|^2) + \log(1+|y|^2)\right]\rho^\eta(x)\rho^\eta(y)\,dxdy\\
&\quad \ge -\frac{1}{8\pi}\log 2(1+\e)-\frac{1}{8\pi}\int_{\R^2\times\R^2}\left(\log(1+|x|^2) + \log(1+|y|^2)\right)\rho^\eta(x)\rho^\eta(y)\,dxdy\\
&\quad \ge -\frac{1}{8\pi}\log2(1+\e)- \frac{1}{4\pi} \int_{\R^2}\rho^\eta \log(1+|x|^2)\,dx\\
&\quad \ge -\frac{1}{8\pi}\log2(1+\e)- \frac{1}{4\pi} \int_{\R^2}|x|^2\rho^\eta \,dx.
\end{aligned}$$
Moreover, the integral of $|\log(\e+|x|)|^p$ on $|x|\le 1$ can be bounded uniformly in $\e$ for every $p\in[1,\infty)$ and thus we have
\begin{align}\label{d2_est}
\begin{aligned}
\left|\int_{\R^2 \times \R^2} W^\e(x-y)\rho_0^\eta(x)\rho_0^\eta(y)\,dxdy\right|
&\le C\left(\|W^\e(\cdot) \mathds{1}_{\{|\cdot| \le 1\}}\|_{L^q}\|\rho_0^\eta\|_{L^p}^2 +  \|W^\e(\cdot) \mathds{1}_{\{|\cdot|>1\}}\|_{L^\infty}\|\rho_0^\eta\|_{L^1}^2\right)\cr
&\le C,
\end{aligned}
\end{align}
where $C$ is a constant independent of $\eta$ and $1/q + 2/p = 2$ with $p \in (1,(d+2)/d)$. Hence, for $d=2$, we can obtain
$$\begin{aligned}
&\frac{1}{2}\int_{\R^2\times\R^2} \left(\frac{|v|^2}{2} + \frac{|x|^2}{2} +\sigma |\log f^\eta| \right) f^\eta \,dxdv\\
& \quad+ \int_0^t\int_{\R^2\times\R^2}\frac{1}{f^\eta} \left|\sigma \nabla_v f^\eta - (v-\chi_\zeta(u_\e^\eta))f^\eta\right|^2  \,dxdvds \\
& \quad+\frac{1}{2}\int_0^t\int_{\R^4\times\R^4}\phi(x-y)|w-v|^2 f^\eta (y,w)f^\eta(x,v)\,dxdydvdwds\\
&\qquad \le \int_{\R^2\times\R^2}\left(\frac{|v|^2}{2} + \frac{|x|^2}{2} +\sigma |\log f_0^\eta| \right) f_0^\eta \,dxdv +\sigma d t\|f_0^\eta\|_{L^1} +\sigma d\|\phi\|_{L^\infty}t \|f_0^\eta\|_{L^1}^2 \\
&\qquad \qquad +\int_0^t\int_{\R^2\times\R^2} \left(|v|^2 +|x|^2\right)f^\eta\,dxdvds + C,
\end{aligned}$$
which gives the desired uniform upper bound estimate.

For the free energy inequality, we first notice that the following inequality still holds for $d=1,2$:
$$\begin{aligned}
&\int_{\R^d \times \R^d} \left(\frac{|v|^2}{2} + V + \sigma \log f^\e \right) f^\e \,dxdv + \frac{1}{2}\liminf_{R \to \infty}\int_{\R^d \times \R^d} W^\e(x-y)\rho^\eta (x) \rho^\eta (y) \,dxdy\\
&\quad + \int_0^t\int_{\R^d\times\R^d}\frac{1}{f^\e} \left|\sigma \nabla_v f^\e - (v-u_\e^\e)f^\e\right|^2  \,dxdvds + \int_0^t\int_{\R^d\times\R^d}|v|^2 f^\e \,dxdvds\\
&\quad +\frac{1}{2}\int_0^t\int_{\R^{2d}\times\R^{2d}}\phi(x-y)|w-v|^2 f^\e (y,w)f^\e(x,v)\,dxdydvdwds\\
&\qquad \le \int_{\R^d \times \R^d} \left(\frac{|v|^2}{2} + V + \sigma \log f_0^\e \right) f_0^\e \,dxdv + \frac{1}{2}\limsup_{R \to \infty}\int_{\R^d \times \R^d} W^\e(x-y)\rho_0^\eta (x) \rho_0^\eta (y) \,dxdy \\
&\qquad \quad+\sigma dt \|f_0\|_{L^1} +\sigma d\int_0^t \int_{\R^d}(\phi \star \rho^\e)\rho^\e\,dxds.
\end{aligned}
$$
When $d\ge 3$, the interaction potential $W$ is positive, thus we used Fatou's Lemma to obtain the desired inequality. Although it is no longer possible to use Fatou's Lemma when $d=1,2$, we use Theorem \ref{thm_v} instead to show that
\[
\lim_{R \to \infty}\int_{\R^d} (W^\e\star\rho^\eta) \rho^\eta \,dx = \int_{\R^d} (W^\e\star\rho^\e) \rho^\e \,dx
\]  
and
\[
\lim_{\e \to 0}\int_{\R^d} (W^\e\star\rho^\e) \rho^\e \,dx = \int_{\R^d} (W\star\rho) \rho \,dx
\]  
for each $t \in [0,T]$. 
Since the proof for the $\e \to 0$ case is similar, we only consider the case $R \to \infty$.

First, we show that $W^\e \star \rho^\eta$ converges to $W^\e \star \rho^\e$ pointwise. Indeed, the pointwise convergence $\rho^\eta \to \rho^\e$ implies $W^\e(x-\cdot)\rho^\eta(\cdot)$ converges to $W^\e(x-\cdot)\rho^\e(\cdot)$ for each $x$. If $d=1$, for each $x$. 
\[
\begin{split}
\int_{|y|\ge L} W^\e(x-y)\rho^\eta(y)\, dy &\le \int_{|y|\ge L} \sqrt{\e + |x-y|^2}\rho^\eta(y)\, dy\\
& \le (|x|+ \sqrt{\e})\int_{|y|\ge L} \rho^\eta (y)\,dy + \int_{|y|\ge L} |y|  \rho^\eta(y)\, dy\\
&\le \left( \frac{|x|}{L^2} + \frac{1}{L} \right) \int_{\R^d} |y|^2 \rho^\eta \,dy \to 0, \quad \mbox{as} \quad L \to \infty.
\end{split}
\]
When $d=2$, one uses $|\log x| \le \max\{x, x^{-1}\}$, $x>0$, and chooses $L$ sufficiently large so that $L \gg |x|$ to get
\[
\begin{split}
\int_{|y|\ge L} W^\e(x-y)\rho^\eta(y)\,dy &\le \frac{1}{2\pi}\int_{|y|\ge L} \log\sqrt{\e + |x-y|^2}\rho^\eta(y)\,dy\\
&\le   \frac{1}{2\pi}\int_{|y|\ge L} \max\{\sqrt{\e + |x-y|^2}, (\e + |x-y|^2)^{-1/2}\} \rho^\eta(y)\,dy\\
&= \frac{1}{2\pi}\int_{|y|\ge L}\sqrt{\e + |x-y|^2} \rho^\eta(y)\,dy,
\end{split}
\]
which also gives the desired estimate Theorem \ref{thm_v} (ii). For the last condition (iii) in Theorem \ref{thm_v}, we choose a measurable set $E$. Then for $d=1$, we use H\"older's inequality and the uniform bounds for $\rho^\eta$ in $L^p$ with $p \in (1,2)$ to get 
\[
\begin{split}
\int_E W^\e(x-y)\rho^\eta(y)\,dy &\le   (|x|+ \sqrt{\e})\int_E \rho^\eta (y)\,dy + \int_E |y|  \rho^\eta(y)\, dy\\
&\le (|x|+\sqrt{\e}) \|\rho^\eta\|_{L^p} m(E)^{1/p'} + \left(\int_E |y|^2\rho^\eta(y)\,dy\right)^{1/2}\left(\int_E \rho^\eta(y)\,dy\right)^{1/2}\\
&\le C(|x|+\sqrt{\e})m(E)^{1/(2p')} \to 0 \quad \mbox{as } \ m(E) \to 0,
\end{split}
\]
where $p' = p/(p-1)$ is the H\"older conjugate of $p$. When $d=2$, we obtain
\[\begin{split}
\int_E W^\e(x-y)\rho^\eta(y)\,dy &\le \frac{1}{\pi}\int_E \max\{(\e + |x-y|^2)^{1/4}, (\e + |x-y|^2)^{-1/4}\} \rho^\eta(y)\,dy\\
&\le \int_{E\cap \{|x-y|\ge 1-\e\}} \sqrt{\e + |x-y|^2} \rho^\eta(y)\,dy + \int_{E\cap \{|x-y|< 1-\e\}}(\e + |x-y|^2)^{-1/4} \rho^\eta(y)\,dy\\
&\le  C(|x|+\sqrt{\e})m(E)^{1/(2p')} + \|(\e + |x|^2)^{-1/4}\mathds{1}_{\{|x|\le 1\}}\|_{L^{7/2}}\|\rho^\eta\|_{L^{7/4}} m(E)^{1/7},
\end{split}\]
which guarantees the condition (iii). Thus, we have the pointwise convergence  of $W^\e \star \rho^\eta$ to $W^\e \star \rho^\e$  and hence $(W^\e \star \rho^\eta)\rho^\eta$ converges to $(W^\e \star \rho^\e)\rho^\e$ almost everywhere.  To prove the desired convergence, we notice that
\[
|W^\e \star \rho^\eta | \le C(1+|x|),
\]
where $C$ is a constant independent of $\eta$. More precisely, we find
\[
|W^\e\star\rho^\eta | \le (1+\e)\int_{\R^d}(1 + |x| + |y|)\rho^\eta(y) \,dy \le C(1+|x|)
\]
for $d=1$ and 
\[\begin{split}
\int_{\R^d} W^\e(x-y)\rho^\eta(y)\,dy &\le \frac{1}{\pi}\int_{\R^d} \max\{(\e + |x-y|^2)^{1/4}, (\e + |x-y|^2)^{-1/4}\} \rho^\eta(y)\,dy\\
&\le \int_{\{|x-y|\ge 1-\e\}} \sqrt{\e + |x-y|^2} \rho^\eta(y)\,dy + \int_{ \{|x-y|< 1-\e\}}(\e + |x-y|^2)^{-1/4} \rho^\eta(y)\,dy\\
&\le  C(1+|x|) + \int_{ \{|x-y|< 1\}}|x-y|^{-1/2} \rho^\eta(y)\,dy\\
&\le C(1+|x|) + \|| \cdot|^{-1/2} \mathds{1}_{\{|\cdot | <1\}} \|_{L^{7/2}} \|\rho^\eta\|_{L^{7/5}} \le C(1+|x|)
\end{split}\]
for $d=2$. Thus, we use the above estimate to validate the conditions in Theorem \ref{thm_v}. For the condition (ii) in Theorem \ref{thm_v}, we choose $L>0$ to get
\[\begin{split}
\int_{|x|\ge L} (W^\e \star \rho^\eta) \rho^\eta\,dx &\le C\int_{|x|\ge L} (1+|x|)\rho^\eta\,dx\\
&\le C\left(\frac{1}{L} + \frac{1}{L^2}\right)\int_{|x|\ge L} |x|^2 \rho^\eta \,dx \to 0 \quad \mbox{as} \quad L \to \infty.
\end{split}\]
For the condition (iii) in Theorem \ref{thm_v}, we have
\[\begin{split}
\int_E (W \star\rho^\eta)\rho^\eta \,dx &\le C\int_E (1+|x|)\rho^\eta \,dx\\
&\le  C \|\rho^\eta\|_{L^p} m(E)^{1/p'} + C\left(\int_E |x|^2 \rho^\eta \,dx\right)^{1/2}\left(\int_E \rho^\eta \,dx\right)^{1/2}\\
&\le C (\|\rho^\eta\|_{L^p} m(E)^{1/p'})^{1/2} ( 1+  (\|\rho^\eta\|_{L^p} m(E)^{1/p'})^{1/2}).
\end{split}\]
Hence, we can obtain the entropy inequality similarly as in the case $d\ge 3$.

\section{Local-in-time existence of strong solutions to the systems \eqref{main_pE} and \eqref{main_npE}}\label{sec_local_pE}
In this section, we study the local-in-time existence and uniqueness of strong solutions to \eqref{main_pE} and \eqref{main_npE} in the periodic domain $\Omega = \T^d$. Since the proof for the system \eqref{main_npE} is similar to that for \eqref{main_pE}, we only provide the details of the proof for the system \eqref{main_pE}, see Section \ref{sec:npE} for the brief idea of the proof for the pressureless case. For the case with smooth interaction potential $W$, we briefly mention the existence result in Remark \ref{rmk_reg} below.

To be more specific, we are mainly interested in the local-in-time solvability of the following isothermal Euler--Poisson system with nonlocal forces:
\begin{align}\label{F-10}
\begin{aligned}
&\pa_t \rho + \nabla \cdot (\rho u) = 0, \quad (x,t) \in \T^d \times \R_+,\cr
&\pa_t (\rho u) + \nabla \cdot (\rho u \otimes u) + \nabla \rho = - \rho u -  \rho(\nabla V + \nabla W \star \rho) -   \rho \int _{\T^d}\phi(x-y)(u(x) - u(y))\rho(y)\,dy.
\end{aligned}
\end{align}
Here we set $\gamma = \lambda = \alpha =1$ for the sake of simplicity. We then reformulate the above system by setting $g := \log \rho$ and rewrite it as follows:
\begin{align}\label{F-1}
\begin{aligned}
&\pa_t g + \nabla g \cdot  u + \nabla \cdot u = 0, \quad (x,t) \in \T^d \times \R_+,\cr
&\pa_t u +  (u \cdot \nabla ) u + \nabla g = -  u - (\nabla V + \nabla W\star e^g) - (\phi\star e^g)u + \phi\star(e^g u),\\
\end{aligned}
\end{align}
subject to initial data:
\bq\label{ini_F-1}
(g(x,0), u(x,0)) = (g_0(x), u_0(x)), \quad x \in \T^d.
\eq

We now state the result on the well-posedness of the system \eqref{F-1}--\eqref{ini_F-1}.

\begin{theorem}\label{local_strong} Let $s > d/2+1$. Suppose that the confinement potential and communication weight satisfy $(\nabla V, \phi) \in H^s(\T^d) \times \mw^{s,\infty}(\T^d)$, and the initial data $(g_0, u_0) \in H^s(\T^d) \times H^s(\T^d)$ with $e^{g_0} > 0$. Then for any positive constants $\epsilon_0 < M_0$, there exists a positive constant $T^*$ such that if $\|g_0\|_{H^s} + \|u_0\|_{H^s} < \epsilon_0$, then the system \eqref{F-1}--\eqref{ini_F-1} admits a unique solution $(g,u) \in \mc([0,T^*]; H^s(\T^d)) \times \mc([0,T^*]; H^s(\T^d))$ satisfying
\[
\sup_{0 \leq t \leq T^*} \lt(\|g(\cdot,t)\|_{H^s} + \|u(\cdot,t)\|_{H^s} \rt) \leq M_0.
\]
\end{theorem}

Note that the solution $(g,u)$ obtained above has $\mathcal{C}^1$-regularity and in particular $g$ is bounded, we can easily show that $(\rho, u)$ with $\rho := e^g$ is a strong solution to \eqref{F-10}. More precisely, we can have the equivalence relation between the classical solutions to the systems \eqref{F-10} and \eqref{F-1}.

\begin{proposition} For any fixed $T>0$, $(\rho, u) \in \mc^1(\T^d \times [0,T]) \times \mc^1(\T^d \times [0,T])$ solves the system \eqref{F-10} with $\rho > 0$ if and only if $(g,u) \in \mc^1(\T^d \times [0,T]) \times \mc^1(\T^d \times [0,T])$ solves the system \eqref{F-1} with $e^g > 0$.
\end{proposition}

The well-posedness theory for the equation \eqref{F-10} has not been developed so far to the best of our knowledge. On the other hand, if the velocity alignment forces, the last term on the right hand side of the momentum equations in \eqref{F-10} and the confinement forces are ignored, the system \eqref{F-10} reduces to the damped isothermal Euler--Poisson system. For that system, the global existence of weak/strong solutions is studied in \cite{CW96, Cor90,Guo98,GP11,PRV95}. We refer to \cite{Chen05} for a general survey on the Euler equations and related conservation laws. Critical thresholds phenomena leading to a finite-time blow-up or a global regularity of strong solutions for the Euler-Poisson system are also investigated in \cite{CCZ16, ELT01}.

\subsection{Solvability for the linearized system} In this subsection, we linearize the system \eqref{F-1} and discuss the local-in-time estimates of solutions to that system. More precisely, for a given
\[
(\tilde{g}, \tilde u) \in \mathcal{C}([0,T];H^s(\T^d)) \times \mathcal{C}([0,T];H^s(\T^d)),
\]
we consider the associated linear system:
\begin{align}\label{F-2}
\begin{aligned}
&\pa_t g + \tilde{u} \cdot \nabla g   + \nabla \cdot u = 0, \quad (x,t) \in \T^d \times \R_+,\cr
&\pa_t u +  (\tilde{u} \cdot \nabla ) u + \nabla g = - u - (\nabla V + \nabla W \star e^{\tilde{g}}) - (\phi\star e^{\tilde{g}})\tilde{u} + \phi\star(e^{\tilde{g}} \tilde{u}),\\
\end{aligned}
\end{align}
with the initial data $(g_0, u_0) \in H^s(\T^d) \times H^s(\T^d)$.

\begin{lemma}\label{L6.1} Let $T>0$ and $s > d/2 +1$. For any positive constants $N<M$, if 
\bq\label{ini_gu}
\|g_0\|_{H^s}^2 + \|u_0\|_{H^s}^2 < N
\eq
and
\[
\sup_{0 \le t \le T} \left( \|\tilde{g}(\cdot,t)\|_{H^s}^2 + \|\tilde{u}(\cdot,t)\|_{H^{s}}^2\right) < M,
\]
then the Cauchy problem \eqref{F-2} has a unique classical solution $(g,u) \in \mc([0,T];H^s(\T^d)) \times \mc([0,T];H^s(\T^d))$ satisfying
\[
 \sup_{0 \le t \le T^*} \left( \|g(\cdot,t)\|_{H^s}^2 + \|u(\cdot,t)\|_{H^{s}}^2\right) < M
\]
for some $T^* \leq T$.
\end{lemma}
\begin{proof} We first easily obtain the existence and uniqueness of solutions to \eqref{F-2} by a standard linear theory of PDEs. Thus, it suffices to provide bound estimates for $g$ and $u$. A straightforward computation gives
\[
\frac12\frac{d}{dt}\|g\|_{L^2}^2 \le \frac{\|\nabla \cdot \tilde{u}\|_{L^\infty}}{2}\|g\|_{L^2}^2 - \int_{\T^d} g \nabla \cdot u \,dx
\]
and
$$\begin{aligned}
&\frac12\frac{d}{dt}\|u\|_{L^2}^2 \le \frac{\|\nabla \cdot \tilde{u}\|_{L^\infty}}{2}\|u\|_{L^2}^2  - \int_{\T^d} \nabla g \cdot u\,dx + \|\nabla V\|_{L^2}\|u\|_{L^2} + \|\nabla W \star e^{\tilde{g}}\|_{L^2}\|u\|_{L^2}\\
&\hspace{2cm} + C\|\phi\|_{L^\infty}(1+e^{\|\tilde{g}\|_{L^\infty}})\|u\|_{L^2}\|\tilde{u}\|_{L^2}\\
&\hspace{1.3cm}\le  \frac{\|\nabla \cdot \tilde{u}\|_{L^\infty}}{2}\|u\|_{L^2}^2  - \int_{\T^d} \nabla g \cdot u\,dx + \|\nabla V\|_{L^2}\|u\|_{L^2} + \|\nabla W\|_{L^1} e^{\|\tilde{g}\|_{L^\infty}}\|u\|_{L^2}\\
&\hspace{2cm} + C\|\phi\|_{L^\infty}(1+e^{\|\tilde{g}\|_{L^\infty}})\|u\|_{L^2}\|\tilde{u}\|_{L^2}.
\end{aligned}$$
Then we use Sobolev inequality and Young's inequality to get
\begin{equation}\label{F-3}
\frac{d}{dt}\left( \|g\|_{L^2}^2 +\|u\|_{L^2}^2 \right) \le Ce^{CM}(1+M)\left( \|g\|_{L^2}^2 +\|u\|_{L^2}^2 \right) +  Ce^{CM}(1+M),
\end{equation}
where $C > 0$ only depends on $s$, $d$, $\nabla V$ and $\phi$.\\

\noindent For higher-order estimates, we first recall the Moser-type inequality:
\[
\|\nabla^k (fg) - f\nabla^k g\|_{L^2} \le C \lt(\|\nabla f\|_{L^\infty} \|\nabla^{k-1} g\|_{L^2} + \|\nabla^k f\|_{L^2}\|g\|_{L^\infty}\rt).
\]
For $1 \le k \le s$, we estimate $\nabla^k g$ as
$$\begin{aligned}
\frac12\frac{d}{dt}\|\nabla^k g\|_{L^2}^2 &= -\int_{\T^d} \nabla (\nabla^k g) \cdot \tilde{u} \nabla^k g \,dx-\int_{\T^d} \left( \nabla^k (\nabla g \cdot \tilde{u}) - \nabla (\nabla^k g)\cdot \tilde{u}\right) \nabla^k g \,dx\\
&\quad -\int_{\T^d}(\nabla \cdot (\nabla^k u)) \nabla^k g \,dx\\
&\le \frac{\|\nabla \cdot \tilde{u}\|_{L^\infty}}{2}\|\nabla^k g\|_{L^2}^2 +C\|\nabla^k g\|_{L^2}\Big( \|\nabla \tilde{u}\|_{L^\infty} \|\nabla(\nabla^{k-1} g)\|_{L^2} + \|\nabla g\|_{L^\infty}\|\nabla^k \tilde{u}\|_{L^2}\Big)\\
&\quad -\int_{\T^d}(\nabla \cdot (\nabla^k u)) \nabla^k g \,dx\\
&\le CM\|\nabla^k g\|_{L^2} + CM\|\nabla^k g\|_{L^2}\|g\|_{H^s} -\int_{\T^d}(\nabla \cdot (\nabla^k u)) \nabla^k g \,dx,
\end{aligned}$$
where $C > 0$ depends only on $d$ and $s$. Similarly, we estimate $\nabla^k u$ as
$$\begin{aligned}
\frac12\frac{d}{dt}\|\nabla^k u\|_{L^2}^2 &= -\int_{\T^d}(\tilde{u} \cdot \nabla (\nabla^k u)) \cdot \nabla^k u \,dx - \int_{\T^d} \lt( \nabla^k (\tilde{u} \cdot \nabla u) - \tilde{u} \cdot \nabla (\nabla^k u)\rt) \cdot \nabla^k u \,dx\\
&\quad -\|\nabla^k u\|_{L^2}^2 - \int_{\T^d}(\nabla (\nabla^k V) + \nabla^k (\nabla W \star e^{\tilde{g}})) \cdot \nabla^k u \,dx -\int_{\T^d}\nabla(\nabla^k g)\cdot \nabla^k u \,dx\\
&\quad + \int_{\T^d \times \T^d} \nabla_x^k \phi(x-y)\tilde{u}(y) e^{\tilde{g}(y)} \nabla^k u(x) \,dydx\\
&\quad -\int_{\T^d \times \T^d}\nabla_x^k (\phi(x-y)\tilde{u}(x)) e^{\tilde{g}(y)}\nabla^k u(x) \,dydx\\
&\le \frac{\|\nabla \cdot \tilde{u}\|_{L^\infty}}{2}\|\nabla^k u\|_{L^2} + C\|\nabla^k u\|_{L^2}\Big( \|\nabla \tilde{u}\|_{L^\infty}\|\nabla (\nabla^{k-1} u)\|_{L^2} + \|\nabla u\|_{L^\infty} \|\nabla^k \tilde{u}\|_{L^2}\Big)\\
&\quad + \|\nabla V\|_{H^k}\|\nabla^k u\|_{L^2} + \|\nabla W \star (\nabla^k e^{\tilde{g}})\|_{L^2}\|\nabla^k u\|_{L^2} -\int_{\T^d}\nabla(\nabla^k g)\cdot \nabla^k u \,dx\\
&\quad + C\|\phi\|_{\mw^{k,\infty}}\|\tilde{u}\|_{H^k} e^{\|\tilde{g}\|_{L^\infty}}\|\nabla^k u\|_{L^2}\\
&\le CM\|\nabla^k u\|_{L^2}^2 + CM\|\nabla^k u\|_{L^2}\|u\|_{H^s}+ Ce^{CM}(1+M)\|\nabla^k u\|_{L^2} \\
&\quad + C\|\nabla^k (e^{\tilde{g}}) \|_{L^2} \|\nabla^k u\|_{L^2}-\int_{\T^d}\nabla(\nabla^k g)\cdot \nabla^k u \,dx,
\end{aligned}$$
where $C > 0$ only depends on $d$, $s$, $\nabla V$, $\|\nabla W\|_{L^1}$ and $\phi$. To estimate the Poisson interaction term, we let $a_k := \|\nabla^k (e^{\tilde{g}}) \|_{L^2}$. As shown previously, we have $a_0 \le e^{\|\tilde{g}\|_{L^\infty}} \le Ce^{CM}$. Then, we use the Moser-type inequality and Sobolev inequality to obtain
$$\begin{aligned}
a_k &=\|\nabla^{k-1} (e^{\tilde{g}} \nabla \tilde{g})\|_{L^2}\\
&\le \|e^{\tilde{g}} \nabla^k \tilde{g} \|_{L^2} + \|\nabla^{k-1}(e^{\tilde{g}}\nabla \tilde{g} ) - e^{\tilde{g}} \nabla^k \tilde{g}\|_{L^2}\\
&\le M e^{CM} + C(\|\nabla e^{\tilde{g}}\|_{L^\infty}\|\nabla^{k-1}\tilde{g}\|_{L^2} + \|\nabla^{k-1}(e^{\tilde{g}})\|_{L^2}\|\nabla \tilde{g}\|_{L^\infty})\\
&\le CMa_{k-1} + CMe^{CM}(1+M),
\end{aligned}$$
where $C>0$ only depends on $d$ and $k$, and inductively, we get
\[
a_k \le CM^k a_0 + CM^{k-1}e^{CM}(1+M) \le Ce^{CM}.
\]
Here $C>0$ only depends on $d$ and $k$. Now, we combine the estimates for $\nabla^k g$ and $\nabla^k u$ to yield
\begin{align}
\begin{aligned}\label{F-4}
\frac{d}{dt}&\left( \|\nabla^k g\|_{L^2}^2 + \|\nabla^k u\|_{L^2}^2 \right)\\
&\le CM\left( \|\nabla^k g\|_{L^2}^2 + \|\nabla^k u\|_{L^2}^2\right) +CM\Big( \|\nabla^k g\|_{L^2}\|g\|_{H^s} + \|\nabla^k u\|_{L^2}\|u\|_{H^s}\Big)\\
&\quad + Ce^{CM}(1+M)\|\nabla^k u\|_{L^2}.
\end{aligned}
\end{align}
We sum the relation \eqref{F-4} over $1 \le k \le s$ and combine this with \eqref{F-3} to get
\[
\frac{d}{dt}\left(\|g\|_{H^s}^2 + \|u\|_{H^s}^2\right)\le Ce^{CM}(1+M)\Big( \|g\|_{H^s}^2 + \|u\|_{H^s}^2\Big) +  Ce^{CM}(1+M).
\]
We write $h(M) := Ce^{CM}(1+M)$ and use Gr\"onwall's lemma to obtain
$$\begin{aligned}
\|g\|_{H^s}^2 + \|u\|_{H^s}^2 &\le (\|g_0\|_{H^s}^2 + \|u_0\|_{H^s}^2)e^{h(M)t} + e^{h(M)t}\lt(1-e^{-h(M)t}\rt)\\
&\le Ne^{h(M)t} +e^{h(M)t}\lt(1-e^{-h(M)t}\rt)\cr
&= N + (N+1)\lt(e^{h(M)t} - 1\rt).
\end{aligned}$$
Note that $N<M$ and $e^{h(M)t} - 1$ can be arbitrary small if $t \ll 1$. This allows us to find $T^*>0$ such that
\[
N + (N+1)\lt(e^{h(M)T^*} - 1\rt) < M.
\]
This concludes the desired result.
\end{proof}

\subsection{Construction of approximate solutions} Now, we construct a sequence that approximates a (unique) solution to \eqref{main_pE}. More precisely, we consider a sequence $(g^n, u^n)$ which is a solution to the following system:
\begin{align}\label{F-5}
\begin{aligned}
&\pa_t g^{n+1} + \nabla g^{n+1} \cdot  u^n + \nabla \cdot u^{n+1} = 0, \quad \quad (x,t) \in \T^d \times \R_+,\cr
&\pa_t u^{n+1} +  (u^n \cdot \nabla ) u^{n+1} + \nabla g^{n+1} \cr
&\hspace{1.5cm} = -  u^{n+1} - (\nabla V + \nabla W \star e^{g^n}) - \phi\star(e^{g^n})u^n + \phi\star(e^{g^n}u^n),
\end{aligned}
\end{align}
with the initial step and initial data defined by
\[
(g^0(x,t), u^0(x,t)) = (g_0(x), u_0(x)) \quad (x,t) \in \T^d \times \R_+
\]
and
\[
(g^n(x,0), u^n(x,0)) = (g_0(x), u_0(x)) \quad \forall \, n \in \N, \quad x \in \T^d,
\]
respectively. We notice that the approximation sequence $(g^n, u^n)$ is well-defined due to Lemma \ref{L6.1}. Moreover, by Lemma \ref{L6.1}, we have the following uniform-in-$n$ bound estimates for the approximation sequence.

\begin{corollary}\label{C6.1} Let $s > d/2 + 1$. For any $M>N$, there exists $T^*>0$ such that if the initial data $(g_0, u_0)$ satisfy \eqref{ini_gu}, then for each $n \in \N$ 
\[
(g^n, u^n) \in \mc([0,T^*]; H^s(\T^d)) \times \mc([0,T^*]; H^s(\T^d))
\] 
and
\[
\sup_{0 \le t \le T^*}\left( \|g^n(\cdot,t)\|_{H^s}^2 + \|u^n(\cdot,t)\|_{H^s}^2\right) < M, \quad \forall  \, n \in \N \cup \{0\}.
\]
\end{corollary}
\begin{proof} For the proof, we use the inductive argument. Since the initial step $(n=0)$ is obvious, it suffices to consider the induction step. We recall from Lemma \ref{L6.1} that 
\[
Ne^{h(M)T^*} + e^{h(M)T^*}(1-e^{-h(M)T^*}) < M
\]
for some $T^*>0$. Then, by the induction hypothesis, we get
\[
\sup_{0 \le t \le T^*}\left( \|g^n(\cdot,t)\|_{H^s}^2 + \|u^n(\cdot,t)\|_{H^s}^2 \right) < M.
\]
This together with the same analysis in Lemma \ref{L6.1}, we have
$$\begin{aligned}
\|g^{n+1}\|_{H^s}^2 + \|u^{n+1}\|_{H^s}^2 &\le \lt(\|g_0\|_{H^s}^2 + \|u_0\|_{H^s}^2\rt)e^{h(M)t} + e^{h(M)t}\lt(1-e^{-h(M)t}\rt)\\
&\le Ne^{h(M)t} +e^{h(M)t}\lt(1-e^{-h(M)t}\rt) < M
\end{aligned}$$
for $0 \leq t \le T^*$. This completes the proof.
\end{proof}
In the lemma below, we show that the approximation sequence $(g^n, u^n)$ is a Cauchy sequence in $\mc([0,T^*];L^2(\T^d)) \times \mc([0,T^*];L^2(\T^d))$.
\begin{lemma}\label{L6.2} Let $(g^n, u^n)$ be a sequence of the approximated solutions with the initial data $(g_0, u_0)$ satisfying \eqref{ini_gu}. Then we have
$$\begin{aligned}
&\|(g^{n+1} -g^n)(\cdot,t)\|_{L^2}^2 + \|(u^{n+1} - u^n)(\cdot,t)\|_{L^2}^2\cr
&\quad \leq  C\int_0^t \lt(\|(g^n -g^{n-1})(\cdot,s)\|_{L^2}^2 + \|(u^n - u^{n-1})(\cdot,s)\|_{L^2}^2\rt) ds
\end{aligned}$$
for $0\leq t \le T^*$ and $n \in \N$, where $C>0$ is independent of $n$.
\end{lemma}
\begin{proof}
First, it follows from \eqref{F-5} that
$$\begin{aligned}
\frac12&\frac{d}{dt}\|g^{n+1} - g^n\|_{L^2}^2 \cr
&\quad = -\int_{\T^d} u^n \cdot \nabla( g^{n+1} - g^n) (g^{n+1} - g^n) \,dx - \int_{\T^d} (u^n - u^{n-1})\cdot \nabla g^n (g^{n+1} - g^n) \,dx\\
&\qquad -\int_{\T^d} \nabla \cdot (u^{n+1} - u^n) (g^{n+1} - g^n) \,dx\\
&\quad \le \frac{\|\nabla \cdot u^n\|_{L^\infty}}{2}\|g^{n+1} - g^n\|_{L^2}^2 + \|\nabla g^n \|_{L^\infty} \|u^n - u^{n-1}\|_{L^2}\|g^{n+1} - g^n\|_{L^2}\\
&\qquad - \int_{\T^d} \nabla \cdot (u^{n+1} - u^n) (g^{n+1} - g^n) \,dx\\
&\quad \le C\left(\|g^{n+1}-g^n\|_{L^2}^2 + \|u^n - u^{n-1}\|_{L^2}^2\right) - \int_{\T^d} \nabla \cdot (u^{n+1} - u^n) (g^{n+1} - g^n) \,dx.
\end{aligned}$$
Next, we estimate
$$\begin{aligned}
\frac12&\frac{d}{dt}\|u^{n+1} - u^n\|_{L^2}^2\\
&= -\int_{\T^d} u^n \cdot \nabla(u^{n+1} - u^n) \cdot (u^{n+1} - u^n)\,dx - \int_{\T^d}(u^n - u^{n-1}) \cdot \nabla u^n \cdot (u^{n+1} - u^n)\,dx\\
&\quad -\int_{\T^d}\nabla(g^{n+1} - g^n) \cdot (u^{n+1} - u^n)\,dx -\|u^{n+1}-u^n\|_{L^2}^2 -\int_{\T^d} \nabla W \star(e^{g^n} - e^{g^{n-1}})\cdot (u^{n+1} - u^n)\,dx\\
&\quad -\int_{\T^d \times \T^d}\phi(x-y)\left( (u^n(x)-u^n(y)) - (u^{n-1}(x)-u^{n-1}(y)) \right) e^{g^n(y)} \cdot (u^{n+1} - u^n)(x)\,dydx\\
&\quad - \int_{\T^d \times \T^d}\phi(x-y)(u^{n-1}(x) - u^{n-1}(y))\left( e^{g^n(y)} - e^{g^{n-1}(y)}\right)\cdot (u^{n+1} - u^n)(x)\,dydx\\
&\le \frac{\|\nabla \cdot u^n\|_{L^\infty}}{2}\|u^{n+1} - u^n\|_{L^2}^2 + \|\nabla u^n\|_{L^\infty}\|u^{n+1} - u^n\|_{L^2}\|u^n - u^{n-1}\|_{L^2}\\
&\quad -\int_{\T^d}\nabla(g^{n+1} - g^n) \cdot (u^{n+1} - u^n)\,dx + \|\nabla W\|_{L^1}\|e^{g^n} - e^{g^{n-1}}\|_{L^2}\|u^{n+1} - u^n\|_{L^2}\\
&\quad +2\|\phi\|_{L^\infty}e^{\|g^n\|_{L^\infty}}\|u^n - u^{n-1}\|_{L^2}\|u^{n+1} - u^n\|_{L^2}\\
&\quad +2\|\phi\|_{L^\infty}\|u^{n-1}\|_{L^2} \left\|e^{g^n} - e^{g^{n-1}}\right\|_{L^2}\|u^{n+1} - u^n\|_{L^2}\\
&\le C\left(\|u^{n+1}-u^n\|_{L^2}^2 + \|u^n- u^{n-1}\|_{L^2}^2 +\|g^n - g^{n-1}\|_{L^2}^2 \right) -\int_{\T^d}\nabla(g^{n+1} - g^n) \cdot (u^{n+1} - u^n)\,dx,
\end{aligned}$$
where we used the mean value theorem to get
\[
\left\|e^{g^n} - e^{g^{n-1}}\right\|_{L^2} \le \exp\lt(\max\{ \|g^n\|_{L^\infty}, \|g^{n-1}\|_{L^\infty}\}\rt)\|g^n - g^{n-1}\|_{L^2}\le C\|g^n - g^{n-1}\|_{L^2}.
\]
Combining all of the above estimates yields
$$\begin{aligned}
\frac{d}{dt}&\left( \|(g^{n+1} -g^n)(\cdot,t)\|_{L^2}^2 + \|(u^{n+1} - u^n)(\cdot,t)\|_{L^2}^2\right)\\
&\le C\left(\|(g^{n+1} -g^n)(\cdot,t)\|_{L^2}^2 + \|(u^{n+1} - u^n)(\cdot,t)\|_{L^2}^2 + \|(g^n -g^{n-1})(\cdot,t)\|_{L^2}^2 + \|(u^n - u^{n-1})(\cdot,t)\|_{L^2}^2 \right)
\end{aligned}$$
for $0 \leq t \le T^*$, where $C>0$ is independent of $n$. We finally apply Gr\"onwall's lemma to conclude the desired result.
\end{proof}

\subsection{Proof of Theorem \ref{local_strong}}
Now, we prove the well-posedness of strong solutions to \eqref{main_pE}. First, Lemma \ref{L6.2} implies that
\[
g^n \to g \quad \mbox{in } \ \mathcal{C}([0,T];L^2(\T^d)) \quad \mbox{and} \quad u^n \to u \quad \mbox{in } \ \mathcal{C}([0,T];L^2(\T^d))
\]
as $n \to \infty$. Moreover, we can extend the convergence in $\mathcal{C}([0,T];L^2(\T^d))$ to the one in $\mathcal{C}([0,T];H^{s-1}(\T^d))$ by interpolating this with the uniform bound in $\mc([0,T];H^s(\T^d))$ from Corollary \ref{C6.1}:
\[
g^n \to g \quad \mbox{in } \ \mathcal{C}([0,T];H^{s-1}(\T^d))  \quad \mbox{and} \quad  u^n \to u \quad \mbox{in } \ \mathcal{C}([0,T];H^{s-1}(\T^d)).
\]
To obtain the $H^s$-regularity of $(g,u)$, we can use a standard argument from functional analysis. For detail, we refer to \cite{CCZ16}.

For the uniqueness, we consider two solutions $(g, u)$ and $(\hat g, \hat u)$ with the same initial data $(g_0, u_0)$. Then, the Cauchy estimate in Lemma \ref{L6.2} gives
\[
\|(g - \hat g)(\cdot,t)\|_{L^2}^2 + \|(u - \hat u)(\cdot,t)\|_{L^2}^2 \le C\int_0^t \lt(\|(g-\hat g)(\cdot,s)\|_{L^2}^2 +  \|(u - \hat u)(\cdot,s)\|_{L^2}^2\rt) ds
\]
for $t \le T^*$. Applying Gr\"onwall's lemma to the above concludes the uniqueness of solutions.

\subsection{Pressureless Euler--Poisson system}\label{sec:npE}

For the pressureless case, by setting $g := \rho - \rho_c$, we can reformulate the system \eqref{main_npE} as 
\begin{align}\label{re_npE}
\begin{aligned}
&\partial_t g + \nabla\cdot ((1+ g) u) = 0, \quad (x,t) \in \T^d \times \R_+,\\
&\partial_t u + (u \cdot \nabla) u = -  u - (\nabla V + \nabla W \star (1+g)) -(\phi\star (1+g)) u + \phi\star((1+g)u),
\end{aligned}
\end{align}
where we simply let $\rho_c=1$, since $\rho_c$ is preserved in time. Then, similarly as before, we construct an approximation sequence to the reformulated system:
$$\begin{aligned}
&\partial_t g^{n+1} + \nabla\cdot ((1+g^{n+1}) u^n) = 0,\\
&\partial_t u^{n+1} + (u^n \cdot \nabla) u^{n+1} = -  u^{n+1} - (\nabla V + \nabla W \star (1+g^n))  -(\phi\star (1+g^n)) u^n + \phi\star((1+g^n)u^n).
\end{aligned}$$
In this case, we use the following estimate for the Poisson interaction term:
\begin{align*}
\int_{\T^d} \nabla (\nabla W \star (1+g^n)) : \nabla u^{n+1} \, dx & = \sum_{i,j=1}^d  \int_{\T^d} \partial_{x_j} (\partial_{x_i}(W\star(1+g^n))) \partial_{x_j} u_i^{n+1}\,dx\\
&= \sum_{i,j=1}^d \int_{\T^d} \partial_{x_j}\partial_{x_j} (W\star(1+g^n))) \partial_{x_i} u_i^{n+1} \,dx\\
&=  \int_{\T^d} \Delta W \star (1+g^n) \nabla \cdot u^{n+1} \,dx = -\int_{\T^d} (1+g^n) \nabla \cdot u^{n+1}\,dx.
\end{align*}
Thus, under suitable assumptions on the confinement potential $\nabla V$ and the communication weight $\phi$, we can use the above estimate to get $H^{s+1}$-estimates for $u$, i.e., for any $M>N$, if
\[
\|g_0\|_{H^s}^2 + \|u_0\|_{H^{s+1}}^2 < N,
\]
then there exists $T^*>0$ such that
\[
\sup_{0 \le t \le T^*}\left( \|g^n(\cdot,t)\|_{H^s}^2 + \|u^n(\cdot,t)\|_{H^{s+1}}\right) < M, \quad \forall \, n \in \bbn.
\]
Then, the similar argument as the above provides the local-in-time existence and uniqueness of strong solutions to the pressureless Euler system \eqref{re_npE}.

\begin{theorem}\label{local_strong2} Let $s > d/2+1$. Suppose that the confinement potential and communication weight satisfy $(\nabla V, \phi) \in H^{s+1}(\T^d) \times \mw^{s+1,\infty}(\T^d)$, and the initial data $(g_0, u_0) \in H^s(\T^d) \times H^{s+1}(\T^d)$ with $g _0 + 1 > 0$. Then for any positive constants $\epsilon_0 < M_0$, there exists a positive constant $T^*$ such that if $\|g_0\|_{H^s} + \|u_0\|_{H^{s+1}} < \epsilon_0$, then the system \eqref{F-1}--\eqref{ini_F-1} admits a unique solution $(g,u) \in \mc([0,T^*]; H^s(\T^d)) \times \mc([0,T^*]; H^{s+1}(\T^d))$ satisfying
\[
\sup_{0 \leq t \leq T^*} \lt(\|g(\cdot,t)\|_{H^s} + \|u(\cdot,t)\|_{H^{s+1}} \rt) \leq M_0.
\]
\end{theorem}

\begin{remark}\label{rmk_reg}
Our analysis can be naturally extended to the case when $\nabla W$ is sufficiently smooth. More precisely, if $\nabla W \in H^s(\T^d)$, then the same result can be obtained for the isothermal Euler system \eqref{main_pE} and $\nabla W \in H^{s+1}(\T^d)$ for the pressureless Euler system \eqref{main_npE}. Furthermore, when $\nabla W \in H^s(\T^d)$, we can also extend the well-posedness result to \eqref{main_pE} to the case when the domain is $\R^d$. For this, we refer to \cite{CCZ16, CH19}.
\end{remark}

%
%
%
%

\section*{Acknowledgments}
	JAC was partially supported by EPSRC grant number EP/P031587/1 and the Advanced Grant Nonlocal-CPD (Nonlocal PDEs for Complex Particle Dynamics: 
	Phase Transitions, Patterns and Synchronization) of the European Research Council Executive Agency (ERC) under the European Union's Horizon 2020 research and innovation programme (grant agreement No. 883363). YPC was supported by NRF grant (No. 2017R1C1B2012918), POSCO Science Fellowship of POSCO TJ Park Foundation, and Yonsei University Research Fund of 2019-22-021. JJ was supported by NRF grant (No. 2019R1A6A1A10073437).

\appendix

%
%
%
%

\section{Proof of Corollary \ref{cor_hydro1.5}: Convergence towards the local Maxwellian} \label{app_cor}
In this part, we provide the details on the proof of Corollary  \ref{cor_hydro1.5}. For this, we recall the definition of relative entropy:
$$\begin{aligned}
\mh(f^\e | M_{\rho, u}) &:= f^\e \log f^\e - M_{\rho, u} \log M_{\rho, u} - (\log M_{\rho,u} + 1)(f^\e - M_{\rho, u})\\
&= f^\e \log\left(\frac{f^\e}{M_{\rho, u}}\right) + M_{\rho, u} - f^\e \cr
&\ge \frac12\min\left\{\frac{1}{f^\e}, \frac{1}{M_{\rho, u}}\right\} |f^\e - M_{\rho, u}|^2.
\end{aligned}$$
Then, we use Cauchy-Schwartz inequality and 
\[
1 \le (x+y)\min\lt\{\frac1x, \frac1y\rt\} \quad \mbox{for}\quad x,y >0
\]
to get
$$\begin{aligned}
&\left(\int_{\Omega\times\R^d} |f^\e - M_{\rho, u}| \,dxdv\right)^2 \cr
&\quad \le \left(\int_{\Omega\times\R^d} (M_{\rho,u} + f^\e)\,dxdv\right) \left(\int_{\Omega\times\R^d} \min\left\{\frac{1}{f^\e}, \frac{1}{M_{\rho, u}}\right\} |f^\e - M_{\rho, u}|^2 \,dxdv\right)\\
&\quad \le 4\int_{\Omega\times\R^d}\mh(f^\e | M_{\rho, u})\,dxdv.
\end{aligned}$$
Thus, for the desired estimate, we investigate the relative entropy $\mh(f^\e | M_{\rho,u})$. We first notice that
\[
\int_{\Omega\times\R^d} \mh(f^\e | M_{\rho,u})\,dxdv = \int_{\Omega} \left(\int_{\R^d}\left(\log f^\e + \frac{|u-v|^2}{2}\right)f^\e\,dv\right)- \rho^\e\lt(\log \rho -\log(2\pi)^{d/2}\rt)dx. 
\]
We then estimate 
$$\begin{aligned}
\frac{d}{dt}\int_{\Omega\times\R^d}f^\e \log f^\e \,dxdv&= \int_{\Omega\times\R^d} \partial_t f^\e \log f^\e \,dxdv\\
&= d + d\int_{\Omega}(\phi\star\rho^\e)\rho^\e \,dx - \frac{1}{\e} \int_{\Omega\times\R^d}(\nabla_v f^\e - (u^\e - v)f^\e)\frac{\nabla_v f^\e}{f^\e}\,dxdv
\end{aligned}$$
and
$$\begin{aligned}
\frac{d}{dt}\int_{\Omega}\rho^\e \log \rho\,dx&= \int_{\Omega}\partial_t \rho^\e \log \rho\,dx + \int_\Omega \rho^\e \frac{\partial_t \rho}{ \rho}\,dx\\
&= \int_\Omega \rho^\e u^\e \cdot \frac{\nabla \rho}{\rho} \,dx + \int_\Omega \rho u \cdot \nabla \left(\frac{\rho^\e}{\rho}\right)\,dx\\
&= \int_\Omega \rho^\e (u^\e - u)\cdot \frac{\nabla \rho}{\rho}\,dx + \int_\Omega u \cdot \nabla \rho^\e \,dx.
\end{aligned}$$
We also find
$$\begin{aligned}
&\int_{\Omega\times\R^d}\frac{|u-v|^2}{2} f^\e \,dxdv - \int_{\Omega\times\R^d}\frac{|u_0-v|^2}{2} f_0^\e \,dxdv\cr
&\quad = \int_0^t\frac{d}{ds}\lt(\int_{\Omega\times\R^d}\frac{|u-v|^2}{2} f^\e \,dxdv\rt)ds\cr
&\quad = \int_0^t \int_\Omega \rho^\e (u-u^\e)\cdot \partial_s u\,dxds + \int_0^t \int_{\Omega\times\R^d}\frac{|u-v|^2}{2}\partial_s f^\e\,dxdvds\\
&\quad =: I + J,
\end{aligned}$$
where $I$ can be estimated as 
$$\begin{aligned}
I &=  \int_0^t\int_\Omega \rho^\e (u^\e - u)\cdot (u \cdot \nabla u)\,dxds +  \int_0^t\int_\Omega \rho^\e (u^\e -u)\cdot \frac{\nabla\rho}{\rho}\,dxds\\
&\quad+  \int_0^t\int_\Omega \rho^\e (u^\e -u)\cdot (\nabla V + u)\,dxds +  \int_0^t\int_\Omega \rho^\e (u^\e -u)\cdot \nabla W \star \rho \,dxds\\
&\quad +  \int_0^t\int_{\Omega \times \om}\phi(x-y)\rho(y)(u(x)-u(y)) \cdot [\rho^\e(u^\e -u)](x)\,dxdyds\\
&\le  \int_0^t\int_\Omega \rho^\e (u^\e -u)\cdot \frac{\nabla \rho}{\rho}\,dxds + C \int_0^t\int_\Omega \rho^\e |u^\e - u|\,dxds.
\end{aligned}$$
For $J$, we have
\begin{align*}
J&= - \int_0^t\int_{\Omega\times\R^d} \nabla_x \cdot (vf^\e)\frac{|u-v|^2}{2}\,dxdvds - \int_0^t\int_{\Omega\times\R^d} (v+\nabla V+\nabla W \star \rho)\cdot (v-u)f^\e\,dxdvds\\
&\quad +  \int_0^t\int_{\Omega\times\R^d} F[f^\e]\cdot (v-u)f^\e\,dxdvds -\frac{1}{\e} \int_0^t\int_{\Omega\times\R^d} (\nabla_v f^\e - (u^\e-v)f^\e)\cdot (v-u)\,dxdvds\\
&=  \int_0^t\int_{\Omega\times\R^d} vf^\e \otimes (u-v):\nabla u\,dxdvds -  \int_0^t\int_{\Omega\times\R^d}v\cdot (v-u)f^\e\,dxdvds \cr
&\quad -\int_0^t\int_\Omega \rho^\e(u^\e - u)\cdot(\nabla V+\nabla W \star\rho)\,dxds \\
&\quad +  \int_0^t\int_{\Omega^2\times\R^{2d}}\phi(x-y)(w-v)f^\e(y,w)\cdot (v-u(x))f^\e(x,v)\,dxdydvdwds\\
&\quad -\frac{1}{\e} \int_0^t\int_{\Omega\times\R^d} (\nabla_v f^\e - (u^\e-v)f^\e)\cdot (v-u)\,dxdvds\\
&= - \int_0^t\int_{\Omega\times\R^d} (u-v)\otimes(u-v)f^\e :\nabla u\,dxdvds + \int_0^t\int_\Omega u \otimes (u^\e - u)\rho^\e :\nabla u\,dxds\\
&\quad - \int_0^t\int_{\Omega\times\R^d} |u^\e-v|^2f^\e\,dxdvds -  \int_0^t\int_\Omega u^\e\cdot (u^\e - u)\rho^\e \,dxds \cr
&\quad- \int_0^t\int_\Omega \rho^\e(u^\e - u)\cdot(\nabla V+\nabla W \star\rho)\,dxds \\
&\quad -\frac{1}{2} \int_0^t\int_{\Omega^2 \times\R^{2d}} \phi(x-y)|v-w|^2f^\e(x,v)f^\e(y,w)\,dxdydvdwds\\
&\quad- \int_0^t\int_{\Omega^2 \times\R^{2d}}\phi(x-y)(w-v)f^\e(y,w)\cdot u(x)f^\e(x,v)\,dxdydvdwds\\
&\quad-\frac{1}{\e} \int_0^t\int_{\Omega\times\R^d}(\nabla_v f^\e - (u^\e-v)f^\e)\cdot (v-u)\,dxdvds\\
&= - \int_0^t\int_{\Omega\times\R^d} (u-v)\otimes(u-v)f^\e :\nabla u\,dxdvds +  \int_0^t\int_\Omega u \otimes (u^\e - u)\rho^\e :\nabla u\,dxds\\
&\quad - \int_0^t\int_{\Omega\times\R^d} |u^\e-v|^2f^\e\,dxdvds -  \int_0^t\int_\Omega u^\e\cdot (u^\e - u)\rho^\e \,dxds\cr
&\quad - \int_0^t\int_\Omega \rho^\e(u^\e - u)\cdot(\nabla V+\nabla W \star\rho)\,dx ds\\
&\quad-\frac{1}{2} \int_0^t\int_{\Omega^2 \times\R^{2d}} \phi(x-y)|v-w|^2f^\e(x,v)f^\e(y,w)\,dxdydvdwds\\
&\quad+\frac{1}{2} \int_0^t\int_{\Omega^2 }\phi(x-y)\rho^\e(x)\rho^\e(y)|u^\e(x)-u^\e(y)|^2\,dxdyds\\
&\quad-\frac{1}{\e} \int_0^t\int_{\Omega\times\R^d}(\nabla_v f^\e - (u^\e-v)f^\e)\cdot (v-u)\,dxdvds\\
&=:\sum_{i=1}^8 J_{i}.
\end{align*}
For $J_1$, we obtain
\begin{align*}
J_1 &= - \int_0^t\int_{\Omega\times\R^d}\left[ (u-u^\e) \otimes (u-u^\e) + (u^\e-v)\otimes (u^\e -v)\right] f^\e : \nabla u\,dxdvds\\
&\le \|\nabla u\|_{L^\infty} \int_0^t\int_\Omega \rho^\e |u^\e-u|^2\,dxds\\
&\quad - \int_0^t\int_{\Omega\times\R^d}((u^\e-v)\sqrt{f^\e} - 2\nabla_v \sqrt{f^\e})\otimes (u^\e-v)\sqrt{f^\e}:\nabla u\,dxdvds\\
&\quad - \int_0^t\int_{\Omega\times\R^d}2\nabla_v \sqrt{f^\e}\otimes(u^\e- v)\sqrt{f^\e}:\nabla u\,dxdvds\\
&\le C \int_0^t\int_\Omega \rho^\e |u^\e-u|^2\,dxds \cr
&\quad + \|\nabla u\|_{L^\infty}\left( \int_0^t\int_{\Omega\times\R^d} |u^\e-v|^2f^\e\,dxdvds\right)^{1/2}\left( \int_0^t\int_{\Omega\times\R^d}\frac{1}{f^\e}|\nabla_v f^\e - (u^\e-v)f^\e|^2\,dxdvds\right)^{1/2}\\
&\quad -  \int_0^t\int_{\Omega\times\R^d}\nabla_v f^\e \otimes (u^\e - v):\nabla u\,dxdvds\\
&\le C\left(\sqrt{\e} +  \int_0^t\int_\Omega \rho^\e |u^\e-u|^2\,dxds\right) +  \int_0^t\int_\Omega \nabla \rho^\e \cdot u\,dxds.
\end{align*}
For $J_2$ and $J_4$, it is easy to get
$$\begin{aligned}
J_2 + J_4&\le \||u||\nabla u|\|_{L^\infty}  \int_0^t\int_\om\rho^\e|u^\e - u|\,dxds + \lt( \int_0^t\int_\om \rho^\e |u^\e|^2\,dxds\rt)^{1/2}\lt( \int_0^t\int_\Omega \rho^\e |u^\e - u|^2\,dxds \rt)^{1/2}  \cr
&\leq C\left( \int_0^t\int_\Omega \rho^\e |u^\e - u|^2\,dxds\right)^{1/2}.
\end{aligned}$$
For $J_5$, we have
\[
\begin{split}
J_5 &\le \|\nabla W \star \rho\|_{L^\infty} \int_0^t\int_\Omega \rho^\e|u^\e-u|\,dxds + \left( \int_0^t\int_\Omega \rho^\e|u^\e-u|^2\,dxds\right)^{1/2}\left( \int_0^t\int_\Omega |\nabla V|^2\rho^\e\,dxds\right)^{1/2}\\
&\le C\left( \int_0^t\int_\Omega \rho^\e|u^\e-u|^2\,dxds\right)^{1/2}.
\end{split}
\]
Moreover, we use the relation from \cite[Proposition 2.1]{KMT15} or \cite[Lemma 7.3]{KMT13} to get
\[
J_6 + J_7 \le C\e + \frac{1}{2\e} \int_0^t\int_{\Omega\times\R^d}\frac{1}{f^\e}|\nabla_v f^\e - (u^\e-v)f^\e|^2\,dxdvds -d \int_0^t\int_\Omega (\phi\star\rho^\e)\rho^\e \,dxds.
\]
Thus, we combine all the previous estimates to get
\begin{align*}
&\int_{\Omega\times\R^d}\mh(f^\e|M_{\rho,u})\,dxdv - \int_{\Omega\times\R^d}\mh(f_0^\e|M_{\rho_0,u_0})\,dxdv\\
&\quad \le C\left(\sqrt{\e} + \left(\int_0^t\int_\Omega \rho^\e|u^\e-u|^2\,dxds\right)^{1/2} + \int_0^t\int_\Omega \rho^\e|u^\e-u|^2\,dxds\right)-\int_0^t\int_{\Omega\times\R^d}|u^\e-v|^2f^\e\,dxdvds\\
&\qquad +dt - \frac{1}{\e}\int_0^t\int_{\Omega\times\R^d}\frac{1}{f^\e}(\nabla_v f^\e - (u^\e-v)f^\e)(\nabla_v f^\e - (u-v)f^\e)\,dxdvds\\
&\qquad +\frac{1}{2\e}\int_0^t\int_{\Omega\times\R^d}\frac{1}{f^\e}|\nabla_v f^\e - (u^\e-v)f^\e|^2\,dxdvds.
\end{align*}
On the other hand, we find
\begin{align*}
dt &=-\int_0^t\int_{\Omega\times\R^d}(v-u^\e)\cdot \nabla_v f^\e\,dxdvds \\
&= -\int_0^t\int_{\Omega\times\R^d}(v-u^\e)\cdot (\nabla_v f^\e - (u^\e -v)f^\e)\,dxdvds + \int_0^t\int_{\Omega\times\R^d}|u^\e-v|^2f^\e\,dxdvds \\
&\le \left(\int_0^t\int_{\Omega\times\R^d} |u^\e-v|^2f^\e\,dxdvds\right)^{1/2}\left(\int_0^t\int_{\Omega\times\R^d} \frac{1}{f^\e}|\nabla_v f^\e - (u^\e -v)f^\e|^2\,dxdvds\right)^{1/2}\cr
&\quad + \int_0^t\int_{\Omega\times\R^d}|u^\e-v|^2f^\e\,dxdvds \\
&\le C\sqrt{\e} + \int_0^t\int_{\Omega\times\R^d}|u^\e-v|^2f^\e\,dxdvds 
\end{align*}
and
\begin{align*}
& -\frac{1}{\e}\int_0^t\int_{\Omega\times\R^d}\frac{1}{f^\e}(\nabla_v f^\e - (u^\e-v)f^\e)(\nabla_v f^\e - (u-v)f^\e)\,dxdvds\\
 &\quad = -\frac{1}{\e}\int_0^t\int_{\Omega\times\R^d}\frac{1}{f^\e}|\nabla_v f^\e - (u^\e-v)f^\e|^2\,dxdvds - \frac{1}{\e}\int_0^t\int_{\Omega\times\R^d}(\nabla_v f^\e - (u^\e-v)f^\e)\cdot (u^\e -u)\,dxdvds\\
 &\quad = -\frac{1}{\e}\int_0^t\int_{\Omega\times\R^d}\frac{1}{f^\e}|\nabla_v f^\e - (u^\e-v)f^\e|^2\,dxdvds.
\end{align*}
Therefore, we use Theorem \ref{thm_hydro1} to have
\begin{align*}
&\int_{\Omega\times\R^d}\mh(f^\e|M_{\rho,u})\,dxdv - \int_{\Omega\times\R^d}\mh(f_0^\e|M_{\rho_0,u_0})\,dxdv\cr
&\quad \le C\sqrt{\e} + C\left( \left(\int_0^t\int_\Omega \rho^\e|u^\e-u|^2\,dxds\right)^{1/2} + \int_0^t\int_\Omega \rho^\e|u^\e-u|^2\,dxds \right)\\
&\quad \le C\e^{1/4} + C\left( \left(\int_\Omega |\nabla W\star(\rho_0^\e - \rho_0)|^2\,dx\right)^{1/2} + \int_\Omega |\nabla W\star(\rho_0^\e - \rho_0)|^2\,dx\right).
\end{align*}
Note that for the case $\nabla W \in L^\infty(\Omega)$, the right-hand side is just $C\e^{1/4}$. This completes the proof.


\section{Proof of Theorem 5.2}\label{app.B}
In this part, we present the proof for Theorem \ref{thm_v}. Here, we only present the proof for the case $\Omega=\R^d$, since the case $\Omega=\T^d$ is analogous. First, the condition (ii) and the integrability of $h$ imply that for every $\e_0>0$, there exists $L>0$ such that
\[
\left(\sup_{n \in \N}\int_{\{|x|>L\}}|h_n|\,dx \right) + \int_{\{|x|>L\}} |h|\,dx < \frac{\e_0}{2}.
\]
Moreover, we can also choose $\delta>0$ such that
\[
\left(\sup_{n \in \N}\int_E |h_n|\,dx \right) + \int_E |h|\,dx < \frac{\e_0}{2}, \quad \mbox{whenever } \ m(E)<\delta.
\]
For those choices of $L$ and $\delta$, we use Egoroff's theorem to get a set $A_\delta$ such that $m(\{|x|\le L\}\setminus A_\delta)<\delta$ and
\[
h_n \to h \quad \mbox{uniformly on } \ A_\delta.
\]
Thus, we have
$$\begin{aligned}
\int_{\R^d} |h_n - h|\,dx &= \int_{\{|x|\le L\}} |h_n -h|\,dx + \int_{\{|x|>L\}} |h_n -h|\,dx\\
&\le \int_{A_\delta} |h_n-h|\,dx + \int_{\{|x|\le L\}\setminus A_\delta} (|h_n| + |h|)\,dx + \int_{\{|x|>L\}}(|h_n|+|h|)\,dx\\
&\le  \int_{A_\delta} |h_n-h|\,dx  + \e_0,
\end{aligned}$$
which implies
\[
\limsup_{n \to \infty} \int_{\R^d}|h_n-h|\,dx \le \e_0.
\]
Since the choice of $\e_0$ was arbitrary, we conclude the proof.



%
%
%
%

\end{document}